\documentclass[a4paper,11pt]{article}
\usepackage[utf8]{inputenc}
\usepackage{needspace}

\usepackage{amsmath,amsfonts,amsthm,amssymb,mathrsfs}
\usepackage[french,english]{babel}
\usepackage{pstricks}
\usepackage{tikz}
\usepackage{graphicx}
\usepackage{enumerate}
\usetikzlibrary{trees}

\makeatletter
\def\CatchFBT@Fin@l#1[#2]{%
   \begingroup
      \makeatletter #2%
      \scantokens\expandafter{%
         \expandafter\CatchFBT@tok\expandafter{\the\CatchFBT@tok}}%
      \CatchFBT@IsAToken{#1}
         {\global#1\expandafter{\the\CatchFBT@tok}}
         {\xdef#1{\the\CatchFBT@tok}}%
      \ifx\CatchFBT@tok#1\else\global\CatchFBT@tok{}\fi
   \endgroup
}
\makeatother
\usepackage{titletoc}

\def\B{\mathscr{B}}

\def\BP{BP}
\def\breadth{<}
\def\breadthinv{>}
\def\breadthl{\leq}

\def \build#1#2#3{\mathrel{\mathop{\kern 0pt#1}\limits_{#2}^{#3}}}

\def\Cmax{C_{\max}}

\def\D{\mathscr{D}}

\def\E{\mathbb{E}}
\def\eg{\emph{e.g.}}

\newcommand{\espx}[1]{\mathbb{E}\left[#1\right]}

\def\F{\mathscr{F}}

\def\gen{\text{generation}}
\def\Gmod{G^\text{mod}}
\def\Gmodk{G^\text{modk}_\text{end}}
\def\Gmodke{G^{\text{modk},\epsilon_1}_\text{end}}

\def\Gmodm{G^{\text{mod}-}}
\def\Gmodme{G^{\text{mod}-}_\text{end}}

\def\ie{\emph{i.e.}}

\def\N{\mathbb{N}}

\def\P{\mathbb{P}}

\renewcommand{\phi}{\varphi}
\def\Poi{\mathrm{Poi}}
\newcommand{\poisson} [1] {\Pi_{#1}}
\def\poissondummy{\mu^{dummy}}

\def\R{\mathbb{R}}

\def\T{\mathscr{T}}

\def\V{\mathcal{V}}

\def\X{\mathscr{X}}
\def\Xset{\mathbb{X}}

\newcommand{\1} [1] {\mathbf{1_{#1}}}
\title{Dynamic Erd\H{o}s-R\'enyi random graph with forbidden degree}
\author{Lucas \textsc{Mercier}}
\date{}
\begin{document}

\selectlanguage{english}

\newtheorem{thm}{Theorem}[section]
\newtheorem{lem}[thm]{Lemma}
\newtheorem{prp}[thm]{Proposition}
\newtheorem{prop}[thm]{Proposition}
\newtheorem{cor}[thm]{Corollary}
\newtheorem{thmf}[thm]{Th\'eor\`eme}
\newtheorem{lemf}[thm]{Lemme}
\newtheorem{prpf}[thm]{Proposition}
\newtheorem{propf}[thm]{Proposition}
\newtheorem{corf}[thm]{Corollaire}

\newtheorem{innercustomthm}{Theorem}
\newenvironment{customthm}[1]
  {\renewcommand\theinnercustomthm{#1}\innercustomthm}
  {\endinnercustomthm}

\theoremstyle{definition}
\newtheorem{defi}{Definition}[section]
\newtheorem*{nota}{Notation}
\newtheorem*{exa}{Example}
\newtheorem{defif}[defi]{D\'efinition}
\newtheorem*{notaf}{Notation}
\newtheorem*{exaf}{Exemple}

\theoremstyle{remark}
\newtheorem*{rem} {Remark}
\newtheorem*{remf} {Remarque}
\psset{unit=0.1}
\maketitle
\tableofcontents{}
%
\bigskip{}

As suggested by Itai Benjamini, we introduced a variant of the Erd\H{o}s-R\'enyi random graph process with a forbidden degree $k$, in which every edge adjacent to a vertex $v$ is removed when the degree of $v$ reaches $k$ (but the removed edges may very well be added again later). We study the existence of a giant component, depending on the forbidden degree $k$ and the time parameter $t$. We prove that for $k>4$ a giant component appears at some point, while for $k<4$, a giant component never occurs. The main tool of our study is the local limit of the random graph process: it provides useful information about the cases $k>4$, but also the  threshold case k=4.


\section{Introduction}
\subsection{The model}
We consider a sequence of  multigraph-valued stochastic processes  $G^k_{n}=(G^k_{n,t})_{t\geq 0}$, in which $n$ and $k$, two positive integers, stand for the number of vertices and for the forbidden degree, respectively. The vertices are labelled from $1$ to $n$, and $E$ denotes the set of their $\binom n 2$ eventual edges. 

The multiplicity of the edges depends on a  Poisson point process  $\poisson{}$ with  intensity $\frac 1 n\1{t\ge 0}dt\otimes \mu$ on $\R_+\times E$, where $\mu$ is the counting measure on $E$. $\poisson{}$ will be seen as a marked point process, with elements in $\R_+$, marked by an edge in $E$. For any $e\in E$, $\poisson{e}$  denotes the point process on $\R_+$ of the elements of $\poisson{}$ marked by $e$, and is the set of times when the multiplicity of the edge $e$ is increased by 1. More precisely, initially, $G^k_{n,0}$ does not contain any edge. At each time $t\in \poisson{e}$,  the multiplicity of the edge $e$ is increased by 1, and, for any endpoint $w$ of $e$ that reaches degree $k$ at this step,  every edge incident to $w$, including $e$, is removed (when both endpoints reach degree $k$ simultaneously, the destruction of edges takes place on both sides)\footnote{Technically, the endpoints will never reach degree $k$, going directly from degree $k-1$ to $0$, but we will describe this situation as reaching degree $k$ and immediately going to degree $0$.\label{conventionk}}. A given edge can be added several times to the process, therefore an edge can be removed, and added again at a later time; multiple edges can also occur. A vertex of $G^k_{n,t}$ is said \emph{saturated} if its degree is $k-1$, the maximum possible degree.

At some point, a discrete time version of the continuous time process   $G^k_{n}$ will be needed: the unmarked version of $\poisson{}$ is a  Poisson point process with  intensity $\frac{n-1}2$ on $\R_+$, so it is a.s. possible to order the points of $\poisson{}$ in increasing order. For any integer $i$, let $\tau_i=\inf\{t\ge 0:\poisson{}([0,t]\times E)\geq i\}$. By the memoryless property of the Poisson point process, the times $\tau_i$ are independent of the discrete process $(G^k_{n,\tau_i})_i$. This discrete process can be described as follows:

\begin{itemize}
\item At step $i=0$, there is no edge.
\item At each step $i\ge 1$, choose an edge uniformly at random among the $\binom n 2$ elements of $E$ and increase its multiplicity by one.
\item For each endpoint reaching degree $k$, remove every edge incident to this endpoint.
\end{itemize}

The multigraph-valued stochastic process in which the edges appear according to $(\poisson{e})_{e\in E}$, but are never erased, is denoted $G^\infty_{n}=(G^\infty_{n,t})_{t\ge0}$. We shall call it the \emph{Erd\H{o}s-R\'enyi multigraph process}, or the  \emph{Erd\H{o}s-R\'enyi multigraph}. For any graph $G$, $\Cmax(G)$ will denote the size (the number of vertices) of the largest connected component of $G$.

\paragraph{Possible generalisations}
Let $G$ be a locally finite simple graph. For every edge $e$, let $\poisson{e}$ be a locally finite subset of $\R_+$. Define the forbidden degree version of $G$ as the multigraph process with the same set of vertices as $G$ evolving in the following way:
\begin{itemize}
\item Initially no edge is present.
\item For every point $t$ of $\poisson{e}$,  increase the multiplicity of $e$ by one at time $t$.
\item If a vertex reaches the degree $k$, remove every adjacent edges.
\end{itemize}
This model is not always well-defined when the graph is infinite. A sufficient condition will be described in Section \ref{partlongrange}.

It is possible to allow the forbidden degree to depend on the vertex by using the following generalisation:
\begin{itemize}
\item Let $(k_v)_{v\in V}$ be a sequence of integers, indexed by the set of vertices. The integer $k_v$ will be called \emph{the forbidden degree of $v$}.
\item The edges are added as previously. Whenever a vertex $v$ reaches degree $k_v$, remove every edge adjacent to $v$.
\end{itemize}

This generalization allows to interpolate between two forbidden degrees, by setting the proportion of vertices of each forbidden degree, \eg{} by having a proportion $\lambda$ of the vertices with forbidden degree $k+1$ and a proportion $1-\lambda$ with forbidden degree $k+1$.

Heuristically, we expect that the resulting random graph process is stochastically increasing with the forbidden degree, even if it is not deterministically increasing, as explained in part \ref{issuewehavetodealwith}.

\subsection{Context}
In \cite{erdosrenyi1960}, Erd\H{o}s and R\'enyi obtained a striking result: the largest component of the Erd\H{o}s-R\'enyi graph with $t n$ edges and $n$ vertices has two radically different behaviors  depending on $t$. If $t\leq \tfrac 1 2$, $\Cmax(G^\infty_{n,m})=o(n)$ a.a.s whereas as soon as $t> \tfrac 1 2$ $\Cmax(G)=\Theta(n)$. For this reason we will study the evolution of $\Cmax(G^k_{n,t})$, depending on $(k,t,n)$.  

Another model with a degree constraint is the graph process with degree restriction \cite{rucinskiwormald}. There also exists several models with local constraints, \eg{} the triangle-free process\cite{bohman2009} and  $H$-free processes\cite{erdos1995,bollobas2001,osthus2001} where $H$ is a given subgraph. These models differ with the forbidden degree on many aspects, but in our eye  the essential difference is that the graph processes are increasing in the former models, while in the forbidden degree model the graph process is not, as edges are routinely removed. Another model with edge removal is the Drossel-Schwab forest-fire model\cite{drossel1992}, where full components are removed with a rate proportional to their size.

\subsection{Issues we have to deal with}
\label{issuewehavetodealwith}
Several points need to be considered when studying this process:
\begin{enumerate}
\item\emph{Chronology.} For a fixed $t$, the Erd\H{o}s-R\'enyi multigraph $G^\infty_{n,t}$ does not depend on the chronology of the apparition of edges, but only on the set of added edges, while due the forbidden degree constraint, chronology suddenly matters for $G^k_{n,t}$.

\item\emph{Monotony.} $G^k_{n}$ is not an increasing process, and the existence of a giant component at a given time does not imply the existence  of a giant component at a later time. Actually, having more edges  at a given time  can lead to a smaller graph at a later time.

\item\emph{Locality.} The degree of a given vertex depends on which adjacent edges have been removed. In turn, this depends on the degrees of the other endpoints of these edges, and these degrees depend on the neighbors of these endpoints, and so on... As a consequence, the existence of a local limit is  questionable, much more than  for the Erd\H{o}s-R\'enyi random multigraph.

\end{enumerate}

\subsection{Results}
The main results obtained in this article are the following:

\begin{thm}
\label{k3}For every $k\leq 3$ and for every sequence of non negative real numbers $(t_n)_{n\geq 0}$, \[\Cmax(G^k_{n,t_n})=o(n).\]
\end{thm}
Theorem \ref{k3} means that there is no giant component for $k\leq 3$.

\begin{thm}
\label{k5}
If $k\ge 5$, there exists an interval $I$ such that for any time  $t\in I$, $\Cmax(G^k_{n,t })=\Theta (n)$.
\end{thm}
According to Theorem \ref{k5}, there exists a giant component at some point if $k\geq 5$, but this does not entail that a giant component still exists at a later time.

\begin{thm}
\label{local}
Let $v$ be a random vertex of $G_{n,t}$, chosen uniformly among the vertices of $G_{n,t}$ independently of $(G^\infty_{n,t})_{t\geq 0}$. Then for any $k\in \N \cup\{\infty\}$, $(G^k_{n,t},v)$ converges in distribution, for the local limit topology, to $T^k_t$, when $n$ tends to $\infty$.
\end{thm}
$T^\infty_t$ is a Galton-Watson tree with Poissonian offspring distribution. $T^k_t$ is the forbidden degree version of $T^\infty_t$, and turns out to be a two-stages multitype branching process. The convergence is with respect to the local topology, as introduced by Benjamini and Schramm \cite{MR1873300}.

\begin{thm}
\label{thmequivalence}
For any $t\ge 0$ and $k\in \N$, 
$$\frac{\Cmax(G^k_{n,t})}n\xrightarrow[n\rightarrow \infty]{p}\P(|T^k_t|=\infty).$$
\end{thm}
Expectedly, the behavior of the local limit predicts somehow the existence of the giant component, as the supercriticality of the local limit $T^k_t$ is equivalent to the existence of a giant component.

Section \ref{sectk3} will be about Theorem \ref{k3}, Theorem \ref{k5} will be discussed in Section \ref{sectk5}, the existence and properties of the local limit in Section \ref{sectlocal} and the Theorem \ref{thmequivalence} will be discussed in Section \ref{sectequivalence}.

\section{There is no giant component if $k\leq 3$}
\label{sectk3}

If $k\le2$, no component has more than  $2$ vertices. Thus, this section dealing only  with the case $k=3$, $G^{3}_{n,t}$ is denoted $G_{n,t}$. The largest allowed degree being  $2$,  the connected components are either paths or cycles, which limits their growth. Theorem \ref{k3}  follows from the next proposition.
\begin{prp}
\label{lemk3}  
There exists a constant $A$ such that for any $\ell\in\N$, and for $n\ge 10$, $$\espx{\Cmax(G_{n,\tau_{\ell}})^2}\leq An.$$
\end{prp}
\begin{proof}[Proof of Theorem \ref{k3}]
Let $|a|$ denote the cardinality of a finite set $a$, and set $N(t)=|\poisson{}\cap[0,t]\times E|$, resp. $N_e(t)=|\poisson{e}\cap[0,t]|$. Since the two families $(\tau_i)_{i\in \N}$ and $(G_{n,\tau_{i}})_{i\in \N}$ are independent, we obtain:
\begin{eqnarray*}
\espx{\Cmax(G_{n,t_n})^2}&\le&\sum_i\espx{\Cmax(G_{n,t_n})^2\left|N(t_n)=i\right.}\P(N(t_n)=i)
\\
&\le&\sum_i\espx{\Cmax(G_{n,\tau_{i}})^2}\P(N(t_n)=i)
\\
&\le& An.
\end{eqnarray*}
\end{proof}
\begin{proof}[Proof of Proposition \ref{lemk3}]
For any $n$ and $\ell $, let $B_{\ell}$ (resp. $A_{\ell}$, $C_{\ell}$) denote the set of connected components of $G_{n,\tau_{\ell}}$  (resp. the set  of  acyclic connected components of $G_{n,\tau_{\ell}}$, the set of cycles of $G_{n,\tau_{\ell}}$). Consider $Z_{\ell}$ defined by:
\begin{eqnarray}
\label{Zorro}
Z_{\ell}&=&\sum_{a\in A_{\ell}} |a|^2+2 \sum_{b\in C_{\ell}} |b|^2.
\end{eqnarray}
The reason for counting cycles twice will be clear later. Proposition \ref{lemk3} follows at once from the next Proposition.
\begin{prp}
There exists a positive constant $A$ such that for all integers $n$ and $\ell $, $\E(Z_{\ell})\leq An$.
\end{prp}
\begin{proof}
Set
$$u_\ell=\espx{\tfrac{Z_\ell}n}.$$
Let $\F_{\ell}$ be the $\sigma$-algebra generated by $(G_{n,t})_{t\leq \tau_{\ell}}$, and let  $\Delta$ denote the variation of $Z_{\ell}$:
$$\Delta_{\ell +1}=Z_{\ell +1}-Z_{\ell}.$$
We shall prove that, for suitable constants $\alpha$ and $\beta$, both larger than 1,
\begin{equation}
\label{slope}
\E(\Delta_{\ell +1}|\F_{\ell})\le \alpha + \beta\tfrac{Z_{\ell}}n- \tfrac{1}4 \tfrac{Z_{\ell}^2}{n^2}.
\end{equation}
As a consequence,
\begin{eqnarray*}
n(u_{\ell+1}-u_{\ell})&\le&\alpha +\beta u_{\ell}-\tfrac{1}{4} u_{\ell}^2,
\end{eqnarray*}
Note that $u_0=1$ and let $r$ denote the positive root of $\alpha +\beta X-\tfrac{1}{4} X^2$. Now,
\begin{itemize}
\item if $r\le u_{\ell}\le r+\alpha+\beta^2$, then \eqref{slope} entails that $u_{\ell+1}\le u_{\ell}$;
\item  if $0\le u_{\ell}\le r$, then \eqref{slope} entails that $u_{\ell+1}-u_{\ell}\le\tfrac1n(\alpha+\beta^2)$.
\end{itemize}
Thus $A= r+\alpha+\beta^2$ is a suitable choice since $u_0=1\le A$.
\end{proof}
Now, due to \eqref{Zorro},
$$Z_{\ell}\ge\Cmax(G_{n,\tau_{\ell}})^2,$$
which concludes the proof of Proposition \ref{lemk3}, assuming relation \eqref{slope}.
\end{proof}

\begin{proof}[Proof of  \eqref{slope}] Consider a graph process $\Gamma=(\Gamma_{\ell})_{\ell\ge 0}$ that starts from $n$ vertices and no edges. At each step $\ell\ge 1$  a first random vertex $v_{\ell}$ is picked uniformly and a second random vertex $w_{\ell}$ different from $v_{\ell}$ is picked uniformly too. Then the multiplicity of the edge $e_{\ell}=\{v_\ell,w_\ell\}$ is increased by 1, provided that the forbidden degree rule allows it. If either $v_{\ell}$ or  $w_{\ell}$ has degree 2 in  $\Gamma_{\ell-1}$ then its edges are erased and $e_{\ell}$ is not added. Since $\Gamma$ has the same distribution as  $(G_{n,\tau_{\ell}})_{\ell\ge0}$, we shall consider, in what follows, that the process $(Z_{\ell})$ is a functional of  $\Gamma$, rather than a functional of $(G_{n,\tau_{\ell}})_{\ell\ge0}$.

Let us decompose the variation $\Delta_{\ell}$ according to the connected components of $v_{\ell}$ and of $w_{\ell}$. For $c,d\in B_{\ell}$, set
$$X_{c,d}=\Delta_{\ell} \1{v_\ell \in c}\1{w_\ell\in d},$$
so that
$$\Delta_{\ell}=\sum_{(c,d)\in B_{\ell}^2 }X_{c,d}.$$
Since $A_{\ell}$, $B_{\ell}$, $C_{\ell}$ are measurable with respect to $\F_{\ell}$,
$$\espx{\Delta_{\ell}|\F_{\ell}}=\sum_{(c,d)\in B_{\ell}^2 }\espx{X_{c,d}|\F_{\ell}}.$$
Let us list the cases in which  $Z_{\ell}$ increases:
\begin{enumerate}
\item \textbf{The edge $e_{\ell}$ makes $c$ a cycle.}
This entails that $c =d\in A_{\ell}$ and  that $v_{\ell}$ and $w_{\ell}$ are the two endpoints of $c$, which happens with probability at most $\frac 2 {n(n-1)}$. In this case 
$$\Delta_{\ell}=|c|^2.$$
\item \textbf{The edge $e_{\ell}$ merges $c $ and $d $.} 
The components $c$ and $d$ merge only if $c,d\in A_{\ell}$, $c\neq d$ and $v_\ell $ and $w_\ell $ are endpoints of $c $ and $d $. There are at most $2$ endpoints per component, so, given $c$ and $d$, this happens with probability at most $\frac 4 {n(n-1)}$. In this case $$\Delta_{\ell}=(|c |+|d |)^2-|c |^2-|d |^2=2|c ||d |.$$
\end{enumerate}
Now let us list two of the cases in which  $Z_{\ell}$ decreases:
\begin{enumerate}
\setcounter{enumi}{2}
\item \textbf{$c $ is a cycle.} If $c \neq d $, the cycle $c $ is split into a path of length $|c |-1$ on one hand and an isolated vertex on the other hand. If $c =d $ the cycle can even be split in smaller parts, but in both cases $X_{c ,d }\leq -|c |^2$ with a probability $\frac{|c |}{n}$.
\item \textbf{$c $ has endpoints but $v_\ell$ is not one of them.} Let us say that $v_\ell$ is the  $m$th vertex of $c$. If $c \neq d$, the path $c $ is split into three paths, of length $m-1$, $|c|-m$ and $1$. If $c =d $, one of the three previous paths can be split again due to the effect of $w_\ell $. Thus, for any $m$ with $1<m<|c|-1$, with probability $\frac{1}{n}$, 
$$X_{c ,d }\leq (m-1)^2+(|c |-m)^2+1-|c |^2\leq 2m(m-|c |).$$
\end{enumerate}
These 4 cases cover the possible contributions of $(c, v_\ell)$ to $\Delta_\ell$. The two first cases, in which the two sides of $e_\ell$ play symmetric r\^oles,  also cover the positive contributions of  the other side of $e_\ell$, $(d, w_\ell)$, to $\Delta_\ell$. As we aim to provide an upper bound for   $\Delta_\ell$, we do not need to discuss the negative contributions of the other side, and considering the  4 cases, we obtain:
\begin{eqnarray*}
\sum_{(c ,d )\in  B_{\ell}^2}\espx{X_{c ,d }|\F_{\ell}}&\le&\sum_{1\le i\le 4}D_i,
\end{eqnarray*}
in which
\begin{eqnarray*}
D_1&=&\tfrac2{n(n-1)}\sum_{c \in A_{\ell}}|c |^2
\ \le\ \frac{2Z_\ell}{n(n-1)},
\\
D_2&=&\tfrac8{n(n-1)}\sum_{c ,d \in A_{\ell}}|c ||d |\ \le\ \frac{8n}{n-1},
\\
D_3&=&-\tfrac 1 n\sum_{c \in C_{\ell}}|c |^3
\end{eqnarray*}
and
\begin{eqnarray*}
D_4&=&\tfrac 1 n\sum_{c \in A_{\ell}}\sum_{m=2}^{|c |-1}2m(m-|c |)
\\
&=&\tfrac1{3n}\sum_{c \in A_{\ell}}(-{|c|^3}+7|c|-6)\1{|c|\ge 2}
\\
&\le&\tfrac73-\tfrac1{3n}\sum_{c \in A_{\ell}}|c|^3.\end{eqnarray*}
Thus, for $n\ge9$,
\begin{align*}
\espx{\Delta_{\ell}|\F_{\ell}}&\le 12+\frac{3Z_\ell}{n^2}-\frac1{4n} \left(\sum_{c \in A_{\ell}}|c |^3+\sum_{c \in C_{\ell}}4|c |^3\right).
\end{align*}
Let $(a_c)_{c\in B_{\ell}}$ and $(b_c)_{c\in B_{\ell}}$ denote two sequences defined as follows:
\begin{itemize}
\item if $c\in A_{\ell}$, $a_c=\sqrt{|c|}$ and $b_c=|c|^{\tfrac 3 2}$;
\item if $c\in C_{\ell}$, $a_c=\sqrt{|c|}$ and $b_c=2|c|^{\tfrac 3 2}$.
\end{itemize}
Then, by the Cauchy-Schwarz inequality, we have:
$$(Z_{\ell})^2=(\sum_{c\in B_{\ell}} a_cb_c)^2\leq \sum_{c\in B_{\ell}}a_c^2\sum_{c\in B_{\ell}}b_c^2=n(\sum_{c \in A_{\ell}}|c |^3+\sum_{c \in C_{\ell}}4|c |^3).$$
As a consequence, \eqref{slope} holds true for $(\alpha,\beta)=(12,3)$.\end{proof}

\section{At some point, there is a giant component if $k \ge 5$}
\label{sectk5}
In this part, rather than considering the process with forbidden degree, we shall consider a lower bound, i.e. a new process in which all the edges incident to a vertex to which at least $k$ edges have been added are removed. For $k\geq 5$, this new process is supercritical, allowing us to prove Theorem~\ref{k5}.
\subsection{Approximation by a simpler model}
Consider the graph $g^k_{n,t}$ obtained when one erases all the edges of  $G^\infty_{n,t}$ that are incident to a vertex with degree $k$ or more. As opposed to $G^k_{n,t}$, $g^k_{n,t}$ depends on $\poisson{}$  only through $(N_{e}(t))_{e\in E}$, thus it does not depend on the chronology.  Let $T_k$ denote the transformation that maps $(N_e(t))_{e\in E}$ to $g^k_{n,t}$:

$$g^k_{n,t}=T_k(N_\cdot(t)).$$

\begin{lem} For any $n$ and $t$,
$$g^{k}_{n,t}\le G^k_{n,t}\le G^\infty_{n,t}.$$
\end{lem}
If $G(e)$ denotes the multiplicity of the edge $e$ in $G$,  the order we consider in this section is the product order for $(G(e))_{e\in E}$.
\begin{proof}The second inequality is clear. For the first one, note that if $0\le G^k_{n,t}(e)<G^\infty_{n,t}(e)$, then one of the endpoints of $e$ reaches the threshold $k$ in $G^\infty_{n}$ at some point $s$ before $t$, and, as a consequence, $g^k_{n,t}(e)=0$. Else none of the endpoints' degrees of $e$ reach the threshold $k$ in $G^\infty_{n,t}$, and as a consequence $g^{k}_{n,t}(e)= G^k_{n,t}(e)= G^\infty_{n,t}(e)$. 
\end{proof}

The graph $g^k_{n,t}$, conditionally on its degree sequence, is distributed as a configuration model conditioned on being loopless, as will be proven in	Corollary \ref{cordegreesequence}. The configuration model, introduced by Bender and Canfield \cite[page 297]{bendercanfield} is defined as follow:

\begin{defi}
Let $(c_v)_{1\leq v \leq n}$ be a finite sequence of non-negative integers, such that $\sum_v c_v$ is even. The configuration model associated to $(c_v)_{1\leq v\leq n}$ is the random graph constructed as follow:
\begin{itemize}
\item The set of vertices $V$ is  $\{1,\dots,n\}$.
\item Let $L$ be the multiset containing $c_v$ copies of each vertex $v$ of $V$.
\item A random uniform pairing $E$ of the elements of $L$ is chosen.
\item The set of edges is  defined by the pairing, with each pair $(x,y)$ corresponding to an edge between $x$ and $y$.
\end{itemize}
By construction, the vertex $v$ has degree $c_v$ in the resulting graph.
\end{defi}
\begin{rem}
The configuration model is sometimes defined from the sequence $(d_j)_{j\geq 0}$ where $d_j=\#\{v:c_v=j\}$ denotes the number of vertices of  degree $j$, instead of the degree sequence $c$.
\end{rem}

Let $c=(c_i)_{1\leq i \leq n}$ be a sequence of integers smaller than $k$ such that $\sum c_i$ is even and let $\mathscr G_c$ denote the set of loopless multigraphs with degree sequence $c$. An element of $\mathscr G_c$ is described by the sequence $(g_e)_{e\in E}$ of its edges' multiplicity.
\begin{lem}
\label{degreesequence}
For $g_1$, $g_2\in\mathscr G_c$, let $m^j_i$ denote the number of edges with multiplicity $i$ in $g_j$. Then for any $t\ge 0$,
$$\P(g^{k}_{n,t}=g_1)\prod_{i\geq 2}(i!)^{m_i^1}=\P(g^{k}_{n,t}=g_2)\prod_{i\geq 2}(i!)^{m_i^2}.$$
\end{lem}
\begin{cor}
\label{cordegreesequence}
The graph $g^k_{n,t}$, conditionally given its degree sequence, is a configuration model conditioned to be loopless. 
\end{cor}
In the configuration model, the number of configurations corresponding to a given multigraph is $\frac{\prod_v c_v!}{\prod_i (i!)^{m_i}}$, thus Lemma \ref{degreesequence} implies Corollary \ref{cordegreesequence}. 

\begin{proof}[Proof of Lemma \ref{degreesequence}]
Set $B=\{v\in V:c_v>0\}$. Let $g\in \mathscr G_c$. Let $s(e)_{e\in E}$ be a multigraph on the set of vertices $\{1,\dots,n\}$.

\begin{lem}
\label{antecendentcarac}
For all $e\in E$, let $r(e)=s(e)-g(e)$.
$T_k(s)=g$ if and only if the following three conditions hold:
\begin{enumerate}
\item For all $e$ in $E$, $r(e)\geq 0$.
\item For all $v\in B$, the degree of $v$ in the graph $r$ is strictly smaller than $k-c_v$.
\item For all $e\in E$ such that $r(e)>0$, one of the endpoints of $e$ is not in $B$ and has degree at least $k$ in $r$.
\end{enumerate}
\end{lem}
\begin{proof}
Let us first assume that $T_k(s)=g$. The graph $g$ is obtained from $s$ by removing every edge adjacent to a vertex of degree at least $k$ in $s$. Therefore $g$ is a subgraph of $s$, $r$ denotes the removed edges, and Condition~1 holds. If the degree of $v$ is at least $k$ in $s$, then every edge adjacent to $v$ is removed when building the graph $T_k(s)$. Therefore, if $v$ is in $B$ (\ie{} if the degree of $v$ is positive in $g$), the degree of $v$ in $s$ is strictly smaller than $k$. As $v$ has degree $c_v$ in $g$, $v$ has degree strictly smaller than $k-c_v$ in $r$, proving Condition 2. If $r(e)>0$, this means that the edge $e$ is removed when building $T_k(s)$, \ie{} that one of its endpoint $v$ has degree at least $k$ in $s$. As by Condition 2, $v$ cannot be in $B$, this means that $c_v=0$ and therefore that $v$ has degree at least $k$ in $r$, proving Condition 3.

Conversely, let us assume that $r$ satisfies conditions 1-3. Then by Condition 1, $g$ is a subgraph of $s=g+r$. The endpoints of the edges $e$ of $g$ have positive degree in $g$ and therefore they belong to $B$. By Condition~2, their degree in $s$ is strictly smaller than $k$, and therefore the edges of $g$ are not removed when building $T_k(s)$. The edges of $r$ are removed, as by Condition 3, one of their endpoints has degree at least $k$ in $r$ and therefore degree at least $k$ in $s=g+r$.

\end{proof}
Let $\mathcal{E}_c$ denote the set of graphs satisfying Conditions 1-3 of Lemma \ref{antecendentcarac}. It should be noted that $\mathcal{E}_c$ only depends on $c$, not on $g$. It should also be noted that for any $r\in \mathcal{E}_c$ and $g\in \mathscr G_c$, the set of edges of $r$ and $g$ are disjoint, as the endpoints of edges of $g$ are in $B$, whereas at least one endpoint of the edges of $r$ is not in $B$, by Condition 3.

Since $N_e(t)$ is a Poisson random variable with parameter $\lambda=\tfrac t n$, for any graph $s$:
\begin{eqnarray*}\P(G^{\infty}_{n,t}=s)&=&\prod_{e\in E}\P(N_e(t)=s(e))
\\
&=&e^{-(n-1)t/2}\lambda^{\sum_e s(e)}\prod_{e\in E}\frac 1 {s(e)!}
\end{eqnarray*}
Therefore, for $g\in \mathscr G_c$, 
\begin{eqnarray*}
\P(g^{k}_{n,t}=g)&=&\sum_{s:T_k(s)=g}\P(G^{\infty}_{n,t}=s)\\
&=&\sum_{r\in\mathcal {E}_c}\P(G^{\infty}_{n,t}=g+r)\\
&=&
e^{-(n-1)t/2}\sum_{r\in \mathcal{E}_c}\lambda^{\sum_e g(e)+r(e)}\prod_{e\in E}\frac 1 {(g(e)+r(e))!}
\\
&=&e^{-(n-1)t/2}\frac{\lambda^{\sum_e g(e)}}{\prod_{e\in E}g(e)!}\sum_{r\in \mathcal{E}_c}\lambda^{\sum_e r(e)}\prod_{e\in E}\frac 1 {r(e)!}.
\end{eqnarray*}
Note that for any edge $e$, $r(e)g(e)=0$, hence $(g(e)+r(e))!=g(e)!r(e)!$ holds. Thus
\begin{eqnarray*}
\P(g^{k}_{n,t}=g)\prod_{i\geq 1}(i!)^{m_i}&=&\P(g^{k}_{n,t}=g)\prod_{e\in E}g(e)!
\\
&=&e^{-(n-1)t/2}\lambda^{\sum_e g(e)}\Psi
\\
&=&e^{-(n-1)t/2}\lambda^{\sum_{i=1}^nc_i/2}\Psi,
\end{eqnarray*}
in which $\Psi$, defined by
$$\Psi=\sum_{r\in \mathcal{E}_c}\lambda^{\sum_e r(e)}\prod_{e\in E}\frac 1 {r(e)!}$$
depends only on $c$, as the set $\mathcal{E}_c$ does.
\end{proof}
In the rest of the section, $t$ is fixed, so $G^\infty_{n,t}$ and $g^{k}_{n,t}$ are abridged in $G^\infty_{n}$ and $g^{k}_{n}$ for readability. 
In \cite{molloyreed}, Molloy and Reed first study the configuration model associated to a degree sequence, and then extend the results to the configuration model conditioned on being simple (this conditional model is called random graph with given degree sequence). In our case, we will first study the configuration model associated to the degree sequence of $g^k_n$, using \cite[Theorem~1]{molloyreed} and then extend the result to $g^k_n$ by mirroring the proof used for \cite[Lemma 2]{molloyreed}. For a vertex $v$, the degree of $v$ in $G^\infty_{n}$ (resp. in $g^{k}_{n}$) is denoted  $C_n(v)$ (resp. $c_n(v)$). Set
$$d_{n}(i)=\#\{v:c_n(v)=i\}.$$

Let $G^{conf}_n$ be the configuration model associated to the degree sequence~$c_n$. We shall see that the sequence $d_n$ satisfies the assumptions of \cite[Theorem 1]{molloyreed}, i.e.  for each $i$, there exists a constant  $p_{t,i}$ such that  $$\lim_{n}\frac{d_{n}}n(i)=p_{t,i}.$$
 As the degree in the graph is bounded by $k-1$, the sequence $d_n$ is sparse and well-behaved (with the vocabulary of \cite{molloyreed}).

Let $\mathfrak{PGW}(t)$ denote a Galton-Watson tree whose offspring is  Poisson distributed  with parameter $t$.  By \cite[Proposition 2.6]{dembo2010}, for  any given vertex $v$,   the rooted graph  $(G^\infty_{n},v)$ converges in distribution to $\mathfrak{PGW}(t)$ when $n$ tends to infinity. Since $c_n(v)$ depends only on the vertices of the ball with radius 2 in $(G^\infty_{n,t},v)$ and on the edges between them,  the asymptotic distribution of $c_n(v)$ can be read on the first two levels of $\mathfrak{PGW}(t)$. Two cases arise:
\begin{itemize}
\item{if $C_n(v)\ge k$, the degree of $v$  in $G^\infty_{n}$ exceeds the threshold $k$, and $c_n(v)=0$;}
\item{if $C_n(v)<k$, an edge $\{v,w\}$ incident to $v$ in $G^\infty_{n}$ is erased in $g^k_{n}$ if and only if $C_n(w)$  exceeds the threshold $k$.}
\end{itemize}
Owing to \cite[	Proposition 2.6]{dembo2010},  $(C_n(v),C_n(w))$ weakly converges to $(C(v),C(w))$ such that $(C(v),C(w)-1)$ are i.i.d. and Poisson distributed with parameter $t$, for  both its offspring and its father $v$ contribute to the degree of $w$ in $\mathfrak{PGW}(t)$. Thus the degree of a neighbor of $v$ is less than $k-1$  with an asymptotic probability:
$$\pi_k(t)=  e^{-t}\sum_{i=0}^{k-2}\frac {t^i}{i!},$$
and the asymptotic distribution $c(v)$ of $c_n(v)$ is obtained by the following algorithm:
\begin{enumerate}

\item{draw a Poisson random variable $C(v)$ with parameter $t$;}
\item{build a  Poisson random variable $Y$ with parameter $t \pi_k(t)$ by  a thinning of $C(v)$ with parameter $\pi_k(t)$, so that the conditional distribution of   $Y$ given $C(v)$ is binomial  with parameters $C(v)$ and $\pi_k(t)$; }
\item{set $c(v)=Y\1{C(v)\le k-1}$.}
\end{enumerate}
Set $$p_{t,i}=\P(c(v)=i)=\lim_n\P(c_n(v)=i).$$
\begin{lem}
\label{smoothwb}
For $0\le i<k$,
$d_{n}(i)/n\xrightarrow[n\rightarrow \infty]{(P)}p_{t,i}$.
\end{lem}
\begin{proof}
Since  $\P(c_n(1)=i)\xrightarrow[n\rightarrow\infty]{}p_{t,i}$, and the vertices play symmetric r\^oles,
$$\espx{d_{n}(i)}=\sum_{i=1}^n\P(c_n(v)=i)=n\P(c_n(1)=i) ,$$ and the limit holds for expectations. Let us show that it holds in probability.
\begin{eqnarray*}
\espx{d_{n}(i)^2}&=&\sum_{v,w\in [\![n]\!]}\P(c_n(v)=i=c_n(w))\\
&=&\sum_v\P(c_n(v)=i)+\sum_{v \ne w}\P(c_n(v)=i=c_n(w))\\
&=&n\P(c_n(1)=i)+n(n-1)\P(c_n(1)=i=c_n(2))
\end{eqnarray*}
As the birooted local limit of $G^\infty_{n,t}$ is a couple of independent $\mathfrak{PGW}(t)$,  $\P(c_n(1)=c_n(2)=i)$ converges to $p_{t,i}^2$, and:
$$\espx{d_{n}(i)^2}= n^2p_{t,i}^2+o(n^2).$$
Thus
$Var(d_{n}(i))=o(n^2)$, leading to the desired result.
\end{proof}
Let 
$$Q_t=\sum_i i(i-2)p_{t,i}.$$
According to \cite[Theorem 1]{molloyreed}, if $Q_t>0$, there exists a constant $\alpha$ 

$$\P(\Cmax(G^{conf}_n)\geq \alpha n|c_{n})\xrightarrow[n\rightarrow \infty]{(P)} 1.$$

Moreover, the law of $G^{conf}_n$ conditionally on $c_n$ and on $G^{conf}_n$ being loopless is equal to the law of $g^k_t$ conditionally on $c_n$. By the main result of \cite{MR790916}, there exists a constant $\beta>0$ such that $\P(G^{conf}_n\text{ is loopless}|c_n)>\beta$ a.a.s, and therefore, if $Q_t>0$:
$$\P(\Cmax(g^k_n)\geq \alpha n|c_{n})\xrightarrow[n\rightarrow \infty]{(P)} 1.$$

To conclude, we need to to compute the sign of
$$Q_t=\espx{c(v)(c(v)-2)}=\espx{c(v)(c(v)-1)}-\espx{c(v)}.$$
According to the points 2 and 3 of the description of the law of $c(v)$ above, $\espx{c(v)|C(v)}=\pi_k(t)C(v)\1{C(v)<k}$ and $\espx{c(v)}=\pi_k(t)\espx{C(v)\1{C(v)<k}}$. Similarly, $\espx{c(v)(c(v)-1)}=\pi_k(t)^2\espx{C(v)(C(v)-1)\1{C(v)<k}}$. Moreover, we have:
\begin{eqnarray*}
\E(C(v)\1{C(v)<k})&=&e^{-t}\sum_{i=1}^{k-1}i\frac{t^i}{i!}\\
&=&e^{-t}\sum_{i=1}^{k-1}\frac{t^{i}}{(i-1)!}\\
&=&e^{-t}\sum_{i=0}^{k-2}\frac{t^{i+1}}{i!}\\
&=&t \pi_k(t)
\end{eqnarray*}
Thus the equation for the existence of a giant component become successively:
\begin{eqnarray*}
\pi_k(t)^2\E(C(v)(C(v)-1)\1{C(v)<k})&>&\pi_k(t)\E(C(v)\1{C(v)<k})\\
\pi_k(t)^2\E(C(v)(C(v)-1)\1{C(v)<k})&>&t \pi_k(t)^2\\
e^{-t}\sum_{i=2}^{k-1}i(i-1)\frac {t^i}{i!}&>&t\\
te^{-t}\sum_{i=0}^{k-3}\frac {t^i}{i!}&>&1
\end{eqnarray*}

\begin{itemize}
\item{If $k\ge 5$, there exists a $t$ such that the condition is satisfied. (e.g. $t=2$).}
\item{If $k\le 4$, the condition does not hold for any $t$.}
\end{itemize}

This ends the proof of Theorem \ref{k5}.
\begin{rem}
  This part does not prove the existence (nor the non-existence) of a giant component for $k=4$, and only proves that there is a giant component for some bounded interval of $t$ (and not, as it could be expected, for every $t$ larger than some $t_0$) for $k\geq 5$.
\end{rem}

\section{The local limit}
\label{sectlocal}
The aim of this part is to prove the existence of a local limit, and to study the link between this local limit and the existence of a giant component.

\subsection{Local topology}
For the purpose of this study, the objects used are multigraphs with labelled edges, where the edges are labelled by the time of addition.  A root of a graph $G$ is either a vertex or an edge of $G$. We consider the local topology, as introduced by Benjamini and Schramm\cite{MR1873300}.

\begin{defi}

Given a graph $G$, an edge $e=(v_1,v_2)$ of $G$, a vertex $v$ of $G$, and a non-negative integer~$l$, $B_l(G,v)$ denotes the set of vertices of the ball of radius $l$ centered at $v$ in $G$, and $B_l(G,e)$ denotes $B_l(G,v_1)\cup B_l(G,v_2)$. The notation $B_l(G,v)$ (resp. $B_l(G,e)$) will also denote the subgraph of $G$ induced by the set of vertices $B_l(G,v)$ (resp. $B_l(G,e)$), and $S_l(G,v)$ (resp. $S_l(G,e)$) will denote the sphere of radius $l$ centered at $v$ (resp. $e$) \ie{} the set of vertices $B_l(G,v)\setminus B_{l-1}(G,v)$ (resp. $B_l(G,e)\setminus B_{l-1}(G,e)$), where by convention $B_{-1}=\emptyset$. The ball (resp. sphere) of radius $l$ will be called $l$-ball (resp. $l$-sphere).
\end{defi}

\begin{defi}
For an integer $l$, a $l$-rooted multigraph $(G,r)$ with labelled edges is the data of a multigraph $G$ with labelled edges equipped with a $l$-tuple $r=(r_1,\dots,r_l)$ of roots of $G$: the set of roots is ordered, and a given root can appear several times. 

For an integer $j$ and an $l$-rooted graph $(G,r)$, $B_j(G,(r))$ denotes the $l$-rooted graph $\cup_{1\leq i \leq l} B_j(G,r_i)$.
\end{defi}
\begin{defi}[Isomorphism and distance]
\ \\*[-0.6cm]
\begin{itemize}
\item Two $l$-rooted multigraphs $(G,r)$ and $(\tilde G,\tilde r)$ are said to be isomorphic if there is a graph isomorphism $\Phi$ from $G$ to $\tilde G$ preserving the edges' labels such that for all $i$, $\Phi(r_i)=\tilde r_i$.
\item For any integer $j$, $(G,r)$ and $(\tilde G,\tilde r)$ are said to be $j$-isomorphic if their $j$-balls are isomorphic.
\item The pseudo-distance between two $l$-rooted graphs $G$ and $\tilde G$ is defined as $2^{-j}$, with $j$ the largest integer such that $G$ and $\tilde G$ are $j$-isomorphic (as $l$-rooted labelled multigraphs). The convergence according to this distance is called the local convergence (this is a straightforward extension of the definition of the local convergence, according to Benjamini \emph{et al}). 
\end{itemize}
\end{defi}

If $(G,r)$ and $(\tilde G,\tilde r)$ are isomorphic, then they are $j$-isomorphic for any $j$. The converse holds if there is at least one root in each connected component of each graph, but can be false otherwise. 

\begin{defi}
Let $T^\infty_t$ be the labelled random graph defined by:
\begin{itemize}
\item The shape of $T^\infty_t$ is a Galton-Watson tree with Poisson offspring with mean $t$.
\item Conditionally on the shape of $T^\infty_t$, the edges' labels are uniform on $[0,t]$ and independent.
\end{itemize}
\end{defi}
\begin{lem}
\label{localconvergenceER}
\begin{enumerate}
For any integer $l$, the \emph{unlabelled} graph $G^\infty_{n,t}$, rooted at $l$ vertices chosen independently uniformly at random, converges toward $l$ independent copies of a Galton-Watson tree with Poisson offspring with mean $t$, while the \emph{labelled} graph $G^\infty_{n,t}$, rooted at $l$ vertices chosen independently uniformly at random, converges toward $l$ independent copies of $T^\infty_t$.
\end{enumerate}
\end{lem}

The one-rooted version of the first part of Lemma \ref{localconvergenceER} is a well-known result, for various versions of the Erd\H{o}s-Rényi graph. One instance can be found in \cite[Proposition 2.6]{dembo2010}, for $G^\infty_{n,t}$ with no multiple edges. For completeness, a proof of this multi-roots, multigraph version of the result in \cite{dembo2010} can be found below:

\begin{proof}
For $l$ and $j$ some positive integers, consider $(G,r)$, $r=(r_i)_{1\leq i}$, a $l$-rooted finite graph with radius not larger than $j$ (such that ${\cup_{1\leq i\leq l}B_j(G,r_i)=G}$). Let us compute $p_n(G)$, the probability that $B_j(G^\infty_{n,t},v)$ is isomorphic to $(G,r)$, in which $(v_i)_{1\leq i\leq l}$ is a sequence of $l$ vertices chosen independently at random. We need a few notations:
\begin{itemize}
\item $V_G$ is the vertex set of $G$, and $N_V=|V_G|$;
\item $E_G$ is the edge set of $G$ and $N_E=|E_G|$;
\item for any edge $e\in E_G$, let $m_e$ denote the multiplicity of $e$ in $G$.
\item $N_V^b$ is the number of vertices of $G$ that are in $B_{j-1}(G,r)$;
\item $(w_i)_{1\leq i \leq N_V}$ is a specific ordering of $V_G$, such that $(w_i)_{1\leq i \leq N^b_V}$ are the elements of $B_{j-1}(G,r)$, and  $(w_i)_{N^b_V+1\leq i \leq N_V}$ are the remaining vertices of $G$. 
\end{itemize} 
There are at most $l$ components of $G=\cup_{1\leq i \leq l}B_j(G,r_i)$, therefore $l+N_E\geq N_V$. If $l+N_E=N_V$, then $G$ is a forest of $l$ simple trees.
\begin{lem}
\label{lemsimple}
If $l+N_E>N_V$, then $p_n(G)\rightarrow 0$.
\end{lem}

\begin{proof}[Proof of Lemma \ref{lemsimple}]

Let $\Phi$ be an injection from $V_G$ to $\{1,\dots,n\}$. If $\Phi$ induces an isomorphism between $(G,r)$ and $B_j(G^\infty_{n,t},v)$, then the following conditions holds:
\begin{enumerate}
\item For every $1\leq i \leq l$, $v_i=\Phi(r_i)$.
\item For every edge $(w_i,w_j)$ of $G$, there is an edge, with the same multiplicity, between $\Phi(w_i)$ et $\Phi(w_j)$, in $G^\infty_{n,t}$.
\end{enumerate} 
These conditions are only necessary, not sufficient. For any given $\Phi$, the probability of the first condition is $\frac 1 {n^l}$. The probability of the second condition, being of the following form:
$$\prod_{e\in E_G}\frac{t^{m_e}}{n^{m^e}m^e!}$$
\noindent{}is smaller than $(\frac t n)^{N_E}$. Therefore, by independence between $G^\infty_{n,t}$ and $v$, the probability that $\Phi$ induces an isomorphism between $(G,r)$ and $G(^\infty_{n,t},v)$ is smaller than $\frac{t^{N_E}}{n^{l+N_E}}$. By union bounds, as there are less than $n^{N_V}$ injections between $(w_i)_{1\leq i\leq N_V}$ to $\{1,\dots,n\}$, we obtain:
$$p_n(G)\leq t^{N_E}n^{N_V-l-N_E}$$
\noindent proving Lemma \ref{lemsimple}.
\end{proof}

A Galton-Watson is usually described as a genealogy tree, \ie{} the planer embedding of a rooted tree, or "\emph{plane tree}": a plane tree is a rooted tree with an order relation between the children of the same node (see \cite{flajoletsedgewick}[Annex A.9]). An isomorphism of plane trees (resp. of a sequence of plane trees) is an isomorphism preserving the root (resp. the sequence of roots, with order) and the children's order.  Assuming that $B_j(G^{\infty}_{n,t},v)$ is a forest with $l$ components, let $B_j^*(G^{\infty}_{n,t},v)$ be the plane representation of $B_j(G^{\infty}_{n,t},v)$, where the children are ordered according to  their original label, in $\{1,\dots,n\}$. The rooted graph $B_j(G^\infty_{n,t},v)$ and $G$ are isomorphic as rooted graphs, if and only if the plane representation of $B_j(G^{\infty}_{n,t},v)$ is isomorphic to one of the plane representations of $(G,r)$, as a plane tree or forest. Therefore
$$\P(B_j(G^{\infty}_{n,t},v)\sim (G,r))=\sum_{(G^*,r)}\P(B_j^*(G^{\infty}_{n,t},v)\underset{plane}{\sim}(G^*,r))$$
\noindent where the summation is taken over the plane representations of $(G,r)$.

Let us compute the number of injections $\Phi$ from $V_G$ to $\{1,\dots,n\}$ preserving the order among the children of each node of $G$. First we choose the $N_V$ elements of $\Phi(V_G)$, without ordering (there are $\binom n {N_V}$ possible choices). From there, choosing $\Phi$ is equivalent to choosing a plane forest isomorphic to $G^*$ on $N_V$ vertices. The number of such plane forests is:
$$\frac{N_V!}{\displaystyle\prod_{i=1}^{N^b_V}d^{out}_{w_i}!}$$
\noindent where $d^{out}_{w_i}$ is the outdegree of $w_i$, \ie{} the number of children of $w_i$ (cf. \cite{spencer}[p.2], \cite{chassaingmarckert}[p.5] or \cite{chassaingflajolet}[p.12]). Therefore the number of injections  $\Phi$ from $V_G$ to $\{1,\dots,n\}$ preserving the order among the children of each node of $G$ is
\begin{equation}
\label{denombrementinjection}
\frac{n!}{\displaystyle(n-N_V)!\prod_{i=1}^{N^b_V}d^{out}_{w_i}!}.
\end{equation}
It should be noted that
\begin{equation}
\label{outdegree}
N_E=\sum_{i=1}^{N^b_V}d^{out}_{w_i}
\end{equation}

Such an injection $\Phi$ induces an isomorphism of plane forests between $(G^*,r)$ and $(B_j^*(G^{\infty}_{n,t},v)$ if and only if:
\begin{enumerate}
\item For each $i$, the random uniform root $v_i$ is equal to $\Phi(r_i)$;
\item For any $1\leq i \leq N^b_V$ and $1\leq j \leq N_V$, there is the same number of edges (either $0$ or $1$) between $w_i$ and $w_j$ in $G$ and between $\Phi(w_i)$ and $\Phi(w_j)$ in $G^\infty_{n,t}$;
\item For any $i\leq N^b_V$ and vertex $j\in \{1,\dots, n\}\setminus \Phi(V_G)$, there is no edge between $\Phi(w_i)$ and $j$ in $G^\infty_{n,t}$.
\end{enumerate}
For a given $\Phi$, the probability of the first property is $\frac{1}{n^l}$. As the multiplicity of each edge of $G^\infty_{n,t}$ is a Poisson random variable of parameter $\frac t n$, the probability of the last two properties is:
\begin{equation}
\left(e^{-\frac t n}\frac{(\frac t n)^1}{1!}\right)^{N_E}\left(e^{-\frac t n}\frac{(\frac t n)^0}{0!}\right)^{(n-1)N^b_V-N_E}=e^{-t\frac{n-1}{n}N^b_V}\left(\frac{t}{n}\right)^{N_E}.
\label{probaprop23}
\end{equation}
Combining (\ref{denombrementinjection}), (\ref{outdegree}), (\ref{probaprop23}), the fact that the probability of the first property is $\frac{1}{n^l}$ and $N_v=l+N_e$, we obtain:
\begin{eqnarray*}
\P\left(B_j^*(G^{\infty}_{n,t},v)\underset{plane}{\sim}(G^*,r)\right)&=&\frac{n!}{\displaystyle(n-N_V)!\prod_{i=1}^{N^b_V}d^{out}_{w_i}!}e^{-t\frac{n-1}{n}N^b_V}\left(\frac{t}{n}\right)^{N_E}\\
&\xrightarrow[n\rightarrow \infty]{}&\frac{e^{-tN_V^b}t^{N_E}}{\prod_{i=1}^{N^b_V}d^{out}_{w_i}!}\\
&=&\prod_{i=1}^{N^b_V}\left(e^{-t}\frac{t^{d^{out}_{w_i}}}{d^{out}_{w_i}!}\right)\\
&=&\P(GW^l_j\underset{plane}{\sim}(G^*,r))
\end{eqnarray*}
\noindent where $GW^l_j$ denotes $l$ independent copies of a Galton-Watson tree with Poisson offspring with mean $t$. By summing over all the plane representations of $G$, one obtains that:
$$\P\left(B_j(G^{\infty}_{n,t},v)\sim(G,r)\right)\rightarrow \P\left(GW^l_j\sim(G,r)\right)$$
\noindent for all $G$ and $j$, ending the proof of the first part of Lemma \ref{localconvergenceER}.

By the properties of Poisson point processes, conditionally given its unlabelled version, the labels of the edges of $G^\infty_{n,t}$ are independent and uniform on $[0,t]$. Therefore, the labelled graph $G^\infty_{n,t}$ $l$-rooted at $l$ independent uniform random roots converges weakly toward $l$ independent copies of $T^\infty_t$.

\end{proof}
By the standard properties of Poisson processes, $T^\infty_t$ can be described as a branching tree such that the law of the set of labels of outgoing edges is a Poisson point process of intensity $1$ on $[0,t]$. 

Let $T^\infty_{\infty}$ denote the \emph{Poisson-Weighted Infinite Tree} in dimension $1$, or \emph{PWIT}, as defined by Aldous and Steele \cite{aldoussteele}. The distribution of $T^\infty_t$ can also be described as follows:
\begin{itemize}
\item Remove every edge of $T^\infty_\infty$ whose label is larger than $t$.
\item Let $T^\infty_t$ be the connected component of the root in the resulting subgraph.
\end{itemize}
It should be noted that the notion of convergence used here is not exactly the same as the one used in \cite{aldoussteele}, as the latter allows the edges' labels to  converge toward their limit, whereas the former requires the edges' labels to be eventually constant.

\subsection{The forbidden degree version of infinite graphs.}

\label{partlongrange}
\subsubsection{Neighborhood with boundaries}
\begin{defi}
Let $\Omega$ be the set of locally finite graphs with labelled edges such that no two edges have the same label.  Let $\Omega_{<\infty}$ be the set of finite graphs in $\Omega$. For any non-negative $t$ and graph $G$ , let $G_t$ be the subgraph of $G$ restricted to its edges with label smaller than $t$.
\end{defi}

For any  graph $G\in\Omega_{<\infty}$,  the forbidden degree version $G^k=\Phi(G)$ of $G$ is defined by adding the edges of $G$ in their labelling order and by removing the edges adjacent to any vertex reaching degree $k$. If $G\in \Omega\setminus \Omega_{<\infty}$, the labels are not necessarily well-ordered, so some care is needed to extend $\Phi$ to a larger class of graphs of $\Omega$. This is the aim of this section.
\begin{defi}
Let $e$ be an edge of $G$. $(G_B,B,S)$ is called a \emph{neighborhood with boundary} of $e$ in $G$ if:
\begin{itemize}
\item $S$ and $B$ are subset of the set of vertices of $G$.
\item $G_B$ is the induced subgraph of $G$ restricted to $B$.
\item The edge $e$ is in $G_B$.
\item $S$ is a subset of $B$.
\item $S$ contains the boundary of $B$ in $G$ (\ie{} the sets of vertices of $B$ with a neighbor not in $B$ in the graph $G$)
\end{itemize}

\end{defi}
\begin{rem}

The last condition means that elements of $B\setminus S$ in $G$ are never endpoints of edges between $B$ and $G\setminus G_B$, while vertices in $S$ may be so.
Sometimes we shorten $(G_B,B,S)$ in $(G_B,S)$ as $B$ can be retrieved from $G_B$.

\end{rem}
\begin{exa}

For any graph $G$, integer $l$, and vertex or edge $x$ of $G$, $(B_l(G,x),S_l(G,x))$ is a neighborhood with boundary.

\end{exa}
\begin{defi}
\label{deficertain}
	Let $e$  be an edge of $G$, $\V=(G_B,B,S)$ a neighborhood with boundary of $e$ in $ G$. Let us consider the set $\mathcal S_\V$ of finite graphs $\tilde G$ with an edge $\tilde e\in \tilde G$ and a neighborhood with boundary $(\tilde G_{\tilde B},\tilde B,\tilde S)$ of $\tilde e$ in $\tilde G$ such that $(G_B,B,S,e)$ is isomorphic to $(\tilde G_{\tilde B},\tilde B, \tilde S,\tilde e)$. If either,
\begin{enumerate}
\item $\tilde e$  is present in $\Phi(\tilde G)$ for every $\tilde G\in \mathcal S_\V$,
\item or $\tilde e$ is removed in $\Phi(\tilde G)$ for every $\tilde G\in\mathcal S_\V$,
\end{enumerate}
then we say that the knowledge of $\V$ is sufficient to know whether $e$ is removed in the forbidden degree version  of $G$.
\end{defi}
The set $\mathcal S_\V$ depends both on $(G_B,B)$ and on $S$: the subgraph of $\tilde G$ induced by $\tilde B$ is isomorphic to $G_B$, but one can obtain $\tilde G$ only by growing new edges on the vertices of $\tilde S$. That is, only the vertices of $\tilde S$ may have a different degree in $\tilde G_{\tilde B}$ and in $\tilde G$. Thus $S$ matters in the definition above, in a somehow hidden way.  If finally we are able to extend $\Phi$ to some infinite graph $G$, we expect, obviously, that the status of $e$ in $\Phi(G)$ (present or absent) is the same as the status of its counterpart $\tilde e$ in $\Phi(\tilde G)$.
\subsubsection{Locality of the forbidden version}
For each edge $e$ of $G$, let $l(e)$ denote the smallest integer $l$ such that, for all $t$, the knowledge of $(B_l(G,e),S_l(G,e))$ is sufficient to ascertain whether $e$ is removed in the forbidden degree version of $G_t$. If there is no such $l$, let $l(e)=\infty$.

For any vertex $v$ of $G$, let $l(v)$ be the smallest integer such that the knowledge of $(B_l(G,v),S_l(G,v))$ is sufficient to know which edges adjacent to $v$ are removed in the forbidden degree version of $G_t$ for all $t$.

In this section, we shall prove that for any $T\geq 0$, a.s. $l(v)<\infty$ (resp. $l(e)<\infty$) for any vertex $v$ (resp. any edge $e$) of $T^\infty_T$.
\begin{defi}
\label{defiExtensionPhi}
Let $\Omega^+\subset \Omega$ be the set of labelled graphs $G$ such that for all edges $e$ of $G$, $l(e)<\infty$.

For any graph $G\in \Omega^+$,  $\Phi(G)$ is defined as follows:
\begin{itemize}
\item the set of vertices of $\Phi(G)$ is $V$;
\item the set of edges of $\Phi(G)$ is a subset of $E$;
\item for each $e\in E$, $e$ is an edge of $\Phi(G)$ if and only if $e$ is present in $\Phi(B_{l(e)}(G,e))$.
\end{itemize}
\end{defi}

The main result of this part is the following lemma:
\begin{lem}
\label{longrange}
Almost surely, for all $t\geq 0$, $T^\infty_t\in \Omega^+$.
\end{lem}
Lemma \ref{longrange} will be proven in part \ref{partPropagationPaths}. By Lemma \ref{longrange}, $\Phi(T^\infty_t)$ is well defined, and, as a subgraph of the tree $T^\infty_t$, $\Phi(T^\infty_t)$ is a forest. Let $T^k_t$ be the connected component of the root $\emptyset$ in $\Phi(T^\infty_t)$. 

\begin{cor}
For any $i$, $G^k_{n,t}$ rooted at $m$ independent uniform vertices locally converges in distribution toward $m$ independent copies of $(T^k_t,\emptyset)$.
\end{cor}
\begin{proof}
By Lemma \ref{localconvergenceER}, $G^\infty_{n,t}$ rooted at $m$ random uniform vertices $(v_i)_{1\le i \le m}$ locally converges in distribution toward $\T:=\cup_{1\leq i \leq m}(T_i,r_i)_{1\le i\le m}$, where $(T_i,r_i)$ are $m$ independent copies of $(T^\infty_t,\emptyset)$. 

\begin{lem}
\label{Psicontinous}
For any positive integer $m$, let $\Omega^+_m$ be the set of $m$-rooted elements of $\Omega^+$. Then the application
\begin{eqnarray*}
\Omega^+_m&\rightarrow&\Omega^+_m\\
(G,(r_i)_{1\leq i \leq m})&\rightarrow &(\Phi(G),(r_i)_{1\leq i \leq m})
\end{eqnarray*} is continuous for the local topology.
\end{lem}
Lemma \ref{Psicontinous} entails that $(G^k_{n,t},v)$ locally converges toward $(\Phi(\T),r)$, as $G^k_{n,t}=\Phi(G^\infty_{n,t})$. As $\Phi(\T)=\cup_i(\Phi(T_i))$, and $T^k_t$ is the connected component of the root in $\Phi(T^\infty_t)$, $(G^k_{n,t},v)$ locally converges toward $m$ independent copies of $T^k_t$.
\end{proof}
\begin{proof}[Proof of Lemma \ref{Psicontinous}]
Let $(G^n,(r_i^n)_{1\leq i \leq m})$ be a sequence of $n$-rooted graphs converging toward $(H,(s_i)_{1\leq i \leq m})\in \Omega^+_m$ for the local topology and let $j$ denote a positive integer. As $H\in \Omega^+$, $\cup_i B_j(H,s_i)$ is finite, and $l(e)$ is finite for every edge $e$ of $H$. Therefore, $L=\max_{e\in \cup_i B_j(H,s_i)} l(e)<\infty$. By definition of the local convergence, for $n$ large enough, there exists an isomorphism $\Psi_n$ between $\cup_i B_{j+L+1}(G^n,r^n_i)$ and $\cup_i B_{j+L+1}(H,s_i)$. By definition of $L$ and $l(e)$, every edge $e$ of $\cup_i B_j(G^n,r^n_i)$ is present in $\Phi(G^n)$ if and only if $\Psi_n(e)$ is present in $\Phi(H)$. Therefore $\Psi_n$ induces an isomorphism between $\cup_i B_j(\Phi(G^n),r^n_i)$ and $\cup_i B_j(\Phi(H),s_i)$. As this construction works for any integer $j$, $(\Phi(G^n),(r_i^n)_{1\leq i \leq m})$ locally converges toward $(\Phi(H),(s_i)_{1\leq i \leq m})$.
\end{proof}

\subsubsection{Propagation paths}
\label{partPropagationPaths}

\begin{defi}
\label{defpathpropagation}
For any element $G$ of $\Omega$, a \emph{propagation path} of length $l\in \N$ is the data of a self-avoiding path $(v_0,e_1,v_1,e_2,\dots,e_l,v_l)$ of length $l$ and a sequence of edges $(\tilde e_1,\tilde e_2,\dots,\tilde e_{l-1})$ of length $l-1$ in $G$ such that
\begin{itemize}
\item for all $1\leq i\leq l-1$, $\tilde e_i$ is adjacent to $v_{i}$;
\item for all $i\ne i'$, $\tilde e_i\ne e_{i'}$;
\item the labels of the sequence $(\tilde e_i)$ are decreasing.
\end{itemize}
For any vertex $v$  of $G$, let $l'(v)\in \N\cup\{\infty\}$ be the supremum of the lengths of the propagation paths starting at $v$, \ie{} such that $v=v_0$.
\end{defi}

\begin{lem}
\label{lemBornel}
Let $G\in \Omega$. For every vertex $v$ in $G$:
$$l(v)\leq l'(v)+1.$$
\end{lem}
Lemma \ref{lemBornel} has a useful consequence:

\begin{cor}
\label{leminclusionomega}
Let $\Omega^-$ be the set of graphs $G\in \Omega$ such that for all vertex $v\in G$, $l'(v)<\infty$. Then

$$\Omega^-\subset \Omega^+.$$
\end{cor}
Lemma \ref{longrange} is a consequence of Corollary \ref{leminclusionomega} and the following lemma:
\begin{lem}
\label{lemTinOmegaMoins}
For all $t$, a.s. $T^\infty_t\in \Omega^-$.
\end{lem}
 The rest of this subsection is devoted to the proof of Lemmata \ref{lemBornel} and \ref{lemTinOmegaMoins}.

\begin{proof}[Proof of Lemma \ref{lemBornel}]
Let $G$ be a graph in $\Omega$, $e^*$ be an edge of $G$ and $l$ a positive integer. In order to simplify the notations, $B$ will denote $B_l(G,e^*)$ and $S$ will denote $S_l(G,e^*)$.
\begin{defi}

An edge $e$ of $B$ will be said \textbf{certain at time t} if, for all $t'\leq t$, the knowledge of $(B,S)$ allows one to determine if $e$ is removed in the forbidden degree version of $G_{t'}$ and \textbf{uncertain at time $t$} otherwise. 

The edges in $G\setminus B$  are said uncertain at any time.
\end{defi}
By definition, $l(e^*)> l$ if and only if $e^*$ is uncertain at some time $t$. $G$ is a locally finite graph, therefore $B$ is a finite graph. Let $J$ denote the finite set of labels of $B$. The application $t\rightarrow B_t$ is a right-continuous piecewise constant function jumping only at elements of $J$.

For every edge $e\in B$, let $y_e$ denote the infimum of the times where $e$ is uncertain, and $t_e$ denote the label of $e$.

\begin{lem}
\label{lemuncertaintypropa}
For every edge $e\in B\setminus S$, there exists:
\begin{itemize}
\item an edge $e'$ adjacent to one endpoint $v$ of $e$ such that $t_{e'}=y_e$;
\item an edge $e''\ne e'$ adjacent to $v$ such that $y_{e''}< y_e$.
\end{itemize}
\end{lem}

\begin{proof}

Let $e$ be an edge in $B\setminus S$. $y_e$ corresponds to a jump of $t\rightarrow B_t$, therefore there exists an edge $e'$ such that $t_{e'}=y_e$. As $e$ is certain at time $y_e^-$ and uncertain at time $y_e$, then,, $\tilde e$ is removed at time $y_e$ in some graph $\tilde G$, and it is not in some other graph $\tilde G$ (according to the notations of Definition~\ref{deficertain}). Let $e'$ denote the only edge with label $y_e$ ( $e'$ is unique by definition of $\Omega$). Only the endpoints of $e'$ can reach the forbidden degree at time $y_e$, and only edges adjacent to an endpoint of $e'$ can be removed at time $y_e$ thus $e'$ and $e$ share a common endpoint~$v$. Moreover, in order for the removal of $e$ to be uncertain, the degree of $v$ must be uncertain before time $y_e=t_{e'}$ in the forbidden degree version of $G_\cdot$\footnote{it is possible that $e$ and $e'$ share both their endpoints, e.g. if $e=e'$. In that case, the degree of at least one of their endpoints must be uncertain. Let $v$ denote such an endpoint.}, \ie{} one of the edge adjacent to $v$ is uncertain before time $y_e$. Let $e''$ denote such an edge. As an edge cannot be uncertain before its label, $e''\neq e'$. 
 \end{proof}

\begin{lem}[Necessary condition for the uncertainty to spread]
\label{lemNecessarayUncertainty}
Given an edge $e\in B$ such that $y_e<\infty$, there exists a self-avoiding path $e_1,v_1,e_2,\dots e_{p},v_p$ in $B$ with $e_1=e$  and a sequence of edges $(\tilde e_i)_{1\leq i\leq p-1}$ in $B$ such that:
\begin{enumerate}
\item for every $i \in\{1,p-1\}$, $\tilde e_i$ is adjacent to $v_{i}$;
\item for every $i\ne j$,  $\tilde e_i\ne e_{j}$;
\item $t_{\tilde e_1}>t_{\tilde e_2}\dots >t_{\tilde e_{p-1}}$;
\item the vertex $v_p$ is in $S$.
\end{enumerate}
\end{lem}
\begin{proof}
Let $\Xi$ denote the set of integers $p$ such that there exists a path $e=e_1,v_1,e_2,\dots e_{p},v_p$ in $ B$ and a sequence of edges $(\tilde e_i)_{1\leq i\leq p-1}$ such that:
\begin{enumerate}
\item for every $i \in\{1,p-1\}$, $\tilde e_i$ is adjacent to $v_{i}$;
\item[2'.] for every $i$,  $\tilde e_i\ne e_{i+1}$;
\item[3a.] $y_{e_1}>y_{e_2}\dots >y_{e_{p-1}}$;
\item[3b.] for every $i \in	\{1,p-1\}$, $t_{\tilde e_i}=y_{e_i}$.
\end{enumerate}
Such a path trivially exists for $p=1$. The edges $e_i$ in the path are distinct by condition 3a, and the graph $B$ is finite, thus $\Xi$ is bounded from above. Let $p_{\max{}}=\max{} \Xi$. By Lemma \ref{lemuncertaintypropa}, if $e_{p_{\max{}}}\in B\setminus S$, the path can be extended one step further, by taking $e=e_p$, $\tilde e_p=e'$ and $e_{p+1}=e''$, and therefore $p$ is not maximal. This entails that for any maximal path, either $v_{p_{\max{}}}$ or $v_{p_{\max{}}-1}$ is in $S$, and therefore there exists a path satisfying conditions 1, 2', 3a, 3b and 4 (and condition 3, as a consequence of 3a and 3b). The shortest path satisfying conditions 1, 2', 3a, 3b and 4 is self-avoiding (otherwise it would not be the shortest). If this path does not satisfy condition~2, while satisfying condition 2', there exists $(i,j)$ such that $j\notin \{i,i+1\}$ and $\tilde e_i=e_j$. In that case, $e_j$ is incident to $v_i$, and the path is not self-avoiding.
\end{proof}
As a consequence of Lemma \ref{lemNecessarayUncertainty}, if $l(e^*)>l$, then $e^*$ is uncertain at some time $t$, and there exists a propagation path of length $l$ starting at an endpoint of $e^*$, proving Lemma \ref{lemBornel}.

\end{proof}
\begin{proof}[Proof of Lemma \ref{lemTinOmegaMoins}]
 We will now prove that for all $t\geq 0$, $T^\infty_t\in\Omega^{-}$ a.s. If $T^\infty_t$ contains only the root, $T^\infty_t\in\Omega^-$. With positive probability, the root has at least one child. Let $\T$ be the tree $T^\infty_t$ conditioned on having at least one edge, and let $v$ be one children of the root of $\T$ chosen uniformly at random. The law  of $\T$ rerooted at $v$ is absolutely continuous with respect to the law of $\T$. Therefore, it is sufficient to prove that almost surely $l'(\emptyset)<\infty$ in $\T$, and therefore in $T^\infty_t$.

\begin{defi}
A possible propagation path of length $l$ is the data of:
\begin{itemize}
 \item a self-avoiding path $\emptyset=v_0,e_1,v_1,e_2,v_2,\dots e_{l},v_l$ in $B_l$ from $\emptyset$ to a vertex $v_l$ of $S_l$.
\item a sequence of edges $(\tilde e_i)_{1\leq i\leq l-1}$ in $B_l$ such that:
\begin{itemize}
\item for every $i\in\{1,l-1\}$, $\tilde e_i$ is adjacent to $v_{i}$;
\item for every $i\neq i'$, $\tilde e_i\neq e_{i'}$.
\end{itemize}
\end{itemize}
Let $H_l$ be the set of such possible propagation paths.
\end{defi}

A possible propagation path is a propagation path if the labels of $\tilde e$ are decreasing. Let $p(l)$ denote the probability that there exists a propagation path of length~$l$ starting at the root of $T^\infty_t$. The set $H_l$ only depends on the unlabelled version of the tree $T^\infty_t$. Conditionally on the unlabelled version of the tree $T^\infty_t$, the labels of the edges are i.i.d. uniform on $[0,t]$, thus the conditional probability that a given possible propagation path is actually a propagation path is $\frac{1}{(l-1)!}$.  Therefore, by union bound, $p(l)\leq \frac{\E(|H_l|)}{(l-1)!}$.

The following lemma gives an explicit formula for the expected size of $H_l$:

\begin{lem}
Let $X_t$ be a Poisson random variable of parameter $t$. Then:
$$\E(|H_l|)=\E(X_t)\E(X_t^2)^{l-1}=t(t+t^2)^{l-1}$$  
\end{lem}
\begin{proof}

  The proof is done by induction on $l$. 

$H_{1}$ is the set of self-avoiding paths of length $1$ starting from the root, equal to the degree of the root, distributed as $X_t$.

The set $H_{l+1}$ can be constructed from $H_{l}$ in the following way:
\begin{itemize}
\item Take an element of $H_{l}$ and denote its endpoint by $v_{l}$.
\item Take a child $v_{l+1}$ of $v_l$ and extend the self-avoiding path to $v_{l+1}$.
\item Choose $w$ a neighbor of $v_l$ different from $v_{l+1}$ (either a child of $v_l$, or $v_{l-1}$). Let $\tilde e_{l}$ be the edge between $v_{l}$ and $w$.
\end{itemize}
This construction gives every element of $H_{l+1}$ once. Therefore an element of $H_{l}$ gives $d^2$ elements of $H_{l+1}$ with $d$ the number of children of $v_l$. By the branching property of $T^\infty_t$, this implies that $\E(|H_{l+1}|)=\E(X_t^2)\E(|H_l|)$.
\end{proof}

By a quick computation:

$$p(l)\leq\frac{\E(|H_l|)}{(l-1)!}=\frac{t(t+t^2)^{l-1}}{(l-1)!}\xrightarrow[l\rightarrow\infty]{}0$$
\noindent{} ending the proof of Lemma \ref{lemTinOmegaMoins}.

\end{proof}
Lemma \ref{longrange} allows us to study the local limit, and therefore the convergence of $l$ balls of radius $i$, with $i$ and $l$ fixed. In Section \ref{sectequivalence}, more precise results will be needed, in which $i$ and $l$ can depend on $n$. For this reason, the following lemma will be useful:

\begin{lem}
There exist two sequences $b_n=o(\ln n)$ and $c_n=o(1)$ such that, asymptotically almost surely:
\label{range}
\begin{enumerate}
\item  For every $v$ in $G^\infty_{n,t}$, $l(v)\le b_n$.
\item For every $v$ in $G^\infty_{n,t}$, the ball of radius $b_n$ centered at $v$ in $G^\infty_{n,t}$ contains at most $x_n$ vertices, where $x_n=n^{c_n}$.
\end{enumerate}
\end{lem}
\begin{defi}[Good event characteristic functions]
To avoid dealing with problematic events of vanishing probability, we will introduce several characteristic functions, denoted $GE_1,GE_2,\dots$ ($GE$ stands for \emph{``Good Event''}) such that $GE_i$ is a Bernoulli variable of parameter tending to $1$. Let $GE_1$ be the characteristic function of the events of Lemma \ref{range}.
\end{defi}
\begin{proof}
By Lemma \ref{lemBornel}, proving that a.a.s. there is no propagation path of length $b_n$ in $G^\infty_{n,t}$ implies the first part of Lemma \ref{longrange}. The total number of edges in $G^\infty_{n,t}$ is a Poisson variable of parameter $\binom n 2 \frac t n=\frac{t(n-1)}2$ and is smaller than $tn$ with high probability. As a consequence, $G^\infty_{n,t}$ is a.a.s. a subgraph of $G^\infty_{tn}=G^\infty_{n,\tau_{\lfloor nt\rfloor }}$, the graph $G^\infty_{n,\cdot}$ stopped the first time $\tau_{\lfloor tn\rfloor}$ there are $\lfloor tn\rfloor$ edges in $G^\infty_{n,\cdot}$, thus it is sufficient to prove the absence of propagation path of length $b_n$ in $G^\infty_{nt}$ instead of $G^\infty_{n,t}$, and for this proof we shall rely on combinatorial arguments.

To detect a propagation path, we do not need the label of each edge, only its rank among the $\lfloor nt\rfloor$ edges of $G^\infty_{nt}$. It is convenient to see $G^\infty_{nt}$ as the result of a random allocation of $\lfloor nt\rfloor$ balls in $\frac{n(n-1)}2$ urns (the edge $r\leq \lfloor nt\rfloor$ being incident to $2$ random vertices $\{v(r),\tilde v(r)\}$).

\begin{rem}[Alternative description of propagation paths]
\label{defipathpropaG}
A propagation path of length $l$ can be described by a sequence of vertices $(v_0,v_1,\dots,v_l)$  without repetition together with a $l$-uple without repetition $(r_1,\dots,r_l)$ of integers smaller than $tn$ and a strictly decreasing $l-1$-uple $(\tilde r_1,\dots,\tilde r_{l-1})$ of integers smaller than $tn$ such that:
\begin{enumerate}[(C1)]
\item For every $i\in\{1,\dots, l\}$, $\{v(r_i),\tilde v(r_i)\}=\{v_{i-1},v_i\}$.
\item For every $i\in\{1,\dots,l-1\}$, $v_i\in\{v(\tilde r_i),\tilde v(\tilde r_i)\}$.
\item If $j\neq j'$, then $\tilde r_j\neq r_{j'}$.
\end{enumerate}
\end{rem}
 This description is equivalent to the description in Definition \ref{defpathpropagation}, where the edge $e_i$ (resp. $\tilde e_i$) in Definition~\ref{defpathpropagation} corresponds to the edge $r_i$ (resp. $\tilde r_i$) in this remark.
 
The collision set of a propagation path is the set of integers $i$ such that $r_i=\tilde r_i$. The collision set is a subset of $\{1,\dots,l-1\}$. 

\begin{defi}
A potential propagation path with collision set $S$ is the data of a $l$-uple without repetition $(r_1,\dots,r_l)$ of integers smaller than $tn$ and a strictly decreasing $l-1$-uple $(\tilde r_1,\dots,\tilde r_{l-1})$ of integers smaller than $tn$ such that:
\begin{itemize}
\item $r_i=\tilde r_i\Leftrightarrow i\in S$.
\item If $i\neq j$, then $\tilde r_i\neq r_{j}$.
\end{itemize}
\end{defi}
A potential propagation path is only two sequences of edges (denoted by their corresponding integer), without any constraint on their endpoints: the sequence of edges described by $(r_i)_{1\leq i \leq l}$ does not need to constitute a path nor does the edges $r_i$ and $\tilde r_i$ need to share a common endpoint. A potential propagation path will be a propagation path if and only if there exists a sequence of vertices $(v_0,v_1,\dots,v_l)$ without repetition such that conditions (C1) and (C2) are satisfied.

The number of potential  propagation paths with collision set $S$ is:
$$\binom {\lfloor tn\rfloor} {l-1}\frac{(\lfloor tn\rfloor-l)!}{(\lfloor tn\rfloor-2l+|S|)!}\leq \frac{(tn)^{2l-1-|S|}}{(l-1)!}$$
This formula is obtained by choosing first $(\tilde r_i)_{1\leq i \leq l-1}$, and then choosing the $l-|S|$ elements of $(r_i)_{1\leq i \leq l}$ that are still unknown. 

Given two vertices $v_1$ and $v_2$ and an integer $r$,  $\P(\{v(r),\tilde v(r)\}=\{v_1,v_2\})={\binom {n}2}^{-1}$, and $\P(v_1\in\{v(r),\tilde v(r)\})=\frac 2 n$. Therefore, given a potential propagation path $((r_i)_i,(\tilde r_i)_i)$  and a sequence of vertices $(v_0,v_1,\dots,v_l)$ without repetition, the probability that conditions (C1) and (C2) hold is:

$${\binom {n}2}^{-l}\left(\frac 2 n\right)^{l-1-|S|}=\left(n-1\right)^{-l}\left(\frac 2 n\right)^{2l-1-|S|}.$$

There are less than $n(n-1)^l$ possible choices for $(v_i)_{0\leq i \leq n}$, therefore, by union bound, the probability $p_l$ that there exists a propagation path of length $b_n$ in $G^\infty_{tn}$ is bounded by:
\begin{eqnarray*}
p_l&\leq &\sum_{S\subset \{1,\dots,l-1\}}\frac{(tn)^{2l-|S|-1}}{(l-1)!}n(n-1)^l(n-1)^{-l}\left(\frac 2 n\right)^{2l-1-|S|}\\
&=&\sum_{S\subset \{1,\dots,l-1\}}n\frac{(2t)^{2l-|S|-1}}{(l-1)!}\\
&=&\frac n {(l-1)!}\sum_{i=1}^{l-1} \binom {l-1} i(2t)^{2l-i-1}\\
&=&\frac{n(2t)^l(2t+1)^{l-1}}{(l-1)!}
\end{eqnarray*}

Taking $l=\frac{\ln n}{\sqrt{\ln \ln n}}=:b_n$, and using Stirling formula, one obtains that $p_l\rightarrow 0$.

We now need to bound the volume of the ball of radius $b_n$ in $G^\infty_{n,t}$. For any $w\in S_l(G^\infty_{n,t},1)$, let $d^*(w)$ be the number of edges that connects $w$ to some vertex of $G^\infty_{n,t}\setminus B_l(G^\infty_{n,t},1)$. Conditionally on $B_l(G^\infty_{n,t},1)$, $(d^*_{w})_{w\in  S_l(G^\infty_{n,t},1)}$ is an i.i.d. family of Poisson variables of parameter  $\left(n-|B_l(G^\infty_{n,t},1)|\right)\frac{t}n\leq t$. Therefore the volume of the ball of radius $b_n$ in $G^\infty_{n,t}$ is stochastically dominated by the volume of the first $b_n$ generations of $T^\infty_t$. We assume that $t>1$.

Let $Z_l$ be the number of vertices in the $l$th generation of $T^\infty_t$. Conditionally on $Z_i$, $Z_{i+1}$ is the sum of $Z_i$ independent Poisson variables of parameter~$t$, \ie{} a Poisson variable of parameter $tZ_i$. Using large deviations theory (or direct computation), there exists a positive constant $\alpha$ such that:
\begin{equation}
\label{largeDevPoisson}
\forall x\geq 1,\,\P(\Poi(x)\geq 2x)\leq exp(-\alpha x)
\end{equation}

For all integers $i$, let $p_i=\P\left(Z_{i+1}>(2t)^{i+1}\ln^2 n\text{ and } Z_i\leq (2t)^{i}\ln^2 n\right)$.
\begin{eqnarray*}
p_i&=&\E\left(\P\left(Z_{i+1}>(2t)^{i+1}\ln^2 n|Z_i\right)\1{Z_i\leq (2t)^{i}\ln^2 n}\right)\\
&=&\E\left(\P\left(Poi(tZ_i)>(2t)^{i+1}\ln^2 n|Z_i\right)\1{Z_i\leq (2t)^{i}\ln^2 n}\right)\\
&\leq&\E\left(\P\left(Poi(t(2t)^{i}\ln^2 n)>(2t)^{i+1}\ln^2 n|Z_i\right)\1{Z_i\leq (2t)^{i}\ln^2 n}\right)\\
&\leq&\P\left(Poi(t(2t)^{i}\ln^2 n)>(2t)^{i+1}\ln^2 n\right)\\
&\leq&exp(-\alpha t(2t)^{i}\ln^2 n)\\
&\leq&exp(-\alpha \ln^2 n)\\
&=&o(\frac 1 {n^2})
\end{eqnarray*}
Therefore $P_n:=\sum_{i=0}^{b_n}p_i=o(\frac 1 n)$. With probability larger than $1-P_n$, for all $i\leq b_n$, $Z_i\leq (2t)^i\ln^2 n$, and therefore $\sum_{i=0}^{b_n}Z_i\leq \frac{(2t)^{b_n+1}-1}{2t-1}\ln^2 n=:n^{c_n}$, where $c_n=o(1)$. Therefore, $B_{b_n}(G^\infty_{n,t},1)$ contains less than $n^{c_n}$ vertices with probability larger than $1-P_n$. By union bounds, all the balls of radius $b_n$ in $G^\infty_{n,t}$ contain less than $n^{c_n}$ vertices with probability larger than $1-nP_n\rightarrow 1$.
\end{proof}

\subsection{The local limit is a branching process}
\label{branching}

\begin{defi}
For any $G\in \Omega^+$ and $t\geq 0$, let $\Phi_t(G)=\Phi(G_t)$, where $G_t$ still denotes the subgraph of $G$ restricted to edges with labels smaller than or equal to $t$.
\end{defi}

Given a vertex $v$ of $T^\infty_\infty$, different from the root, let $w$ denote the parent of $v$ and $T^{\infty,v}_\infty$ denote the connected component of $v$ in $T^\infty_\infty\setminus\{v,w\}$ \ie{} the  subtree of $T^\infty_\infty$ starting at~$v$, and let $\tilde T^{\infty,v}_\infty=(T^{\infty,v}_\infty\cup\{v,w\})$ and $\tilde T^{\infty,v}_t=(T^{\infty,v}_\infty)_t$. See figure \ref{Tv} for an example.

\begin{figure}[h]
 \begin{tikzpicture}[level/.style={level distance=1.5cm},{sibling distance=2cm},
        level 2/.style={sibling distance=0.4cm},
        level 3/.style={sibling distance=0.4cm},level 4/.style={sibling distance=0.4cm}, edge from parent/.style={draw},]
 \node [draw,circle] {$\emptyset$}
    child{ node[draw,circle] {}
		child{ node[draw,circle] {}}
	    child{node[draw,circle] {}}
		child{node[draw,circle] {}}
	    child{node[draw,circle] {}}
	    child{node{}edge from parent[dotted]}    		
    		}
    child{node[draw,circle] {}
		child{ node[draw,circle] {}}
		child{node[draw,circle] {}}
	    child{node[draw,circle] {}}
	    child{node[draw,circle] {}}
	    child{node{}edge from parent[dotted]}    		
    		}
    child{node[draw,circle] {}
		child{node[draw,circle] {}}
		child{ node[draw,circle] {}}
	    child{node[draw,circle] {$v$}
			child{ node[draw,circle] {}}
			child{node[draw,circle] {}}
	    		child{node[draw,circle] {}}
	    		child{node[draw,circle] {}}
	    		child{node{}edge from parent[dotted]}    		
	    		}
	    child{node[draw,circle] {}}
	    child{node{}edge from parent[dotted]}    		
    		}
    child{node[draw,circle] {}
		child{node[draw,circle] {}}
		child{ node[draw,circle] {}}
	    child{node[draw,circle] {}}
	    child{node[draw,circle] {}}
	    child{node{}edge from parent[dotted]}    		
    		}
    child{node{}edge from parent[dotted]}
    ;

\begin{scope}[shift={(4,-1.5)}]
\node  {}
    child{ node[draw,circle] {$v$}
		child{ node[draw,circle] {}}
	    child{node[draw,circle] {}}
		child{node[draw,circle] {}}
	    child{node[draw,circle] {}}
	    child{node{}edge from parent[dotted]}    		
    		edge from parent[draw=none]}
    		;
\end{scope}

\begin{scope}[shift={(6.5,-1.5)}]
\node [draw,circle] {}
    child{ node[draw,circle] {$v$}
		child{ node[draw,circle] {}}
	    child{node[draw,circle] {}}
		child{node[draw,circle] {}}
	    child{node[draw,circle] {}}
	    child{node{}edge from parent[dotted]}    		
    		}
    		;
\end{scope}

 \end{tikzpicture}

 \caption{Example of $T^{\infty}_\infty$, $T^{\infty,v}_\infty$ and $\tilde T^{\infty,v}_\infty$.}\label{Tv}

 \end{figure}
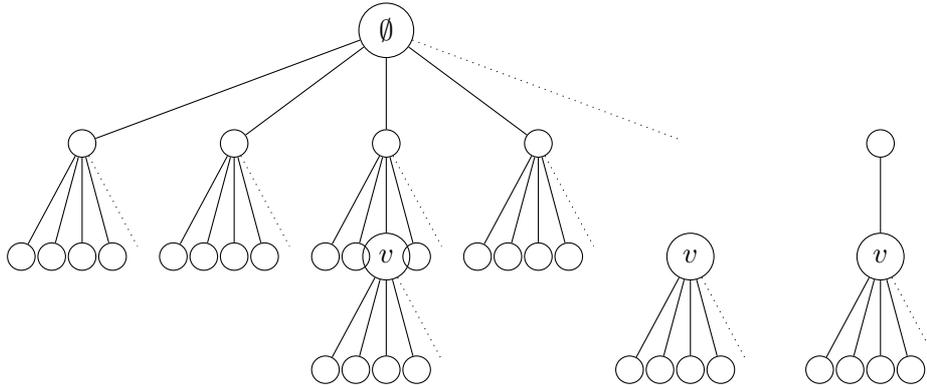

\begin{lem}
\label{lemDecoupage}
For a graph $G\in \Omega^+$, an edge $e\in G$ and a non-negative $t$, we say that $e$ is removed in $\Phi_t(G)$ if $e\in G_t$ and $e\notin \Phi_t(G)$.

Let $t$ be a non-negative number and $v$  vertex of $T^\infty_\infty$. If  $T^\infty_t\in \Omega^-$, then, at least one of the following propositions hold:
\begin{itemize}
\item $\Phi(T^\infty_t)\cap \tilde T^{\infty,v}_\infty=\Phi(\tilde T^{\infty,v}_t)$;
\item the edge between $v$ and its parent $w$ is removed in $\Phi_t(T^\infty_\infty)$.
\end{itemize}

\end{lem}
An immediate corollary of Lemma \ref{lemDecoupage} is:
\begin{cor}
\label{corDecoupage}
With the notation of Lemma \ref{lemDecoupage}, if $(v,w)$ is removed in $\Phi(\tilde T^{\infty,v}_t)$, then it is removed in $\Phi(T^\infty_t)$.
\end{cor}
In other words, until $(v,w)$ is removed in $\Phi_t(T^\infty_\infty)$, the knowledge of $\tilde T^{\infty,v}_\infty$ is sufficient to know the evolution of the subtree starting at $v$ in the forbidden degree version of $T^\infty_\infty$.
\begin{proof}
The graphs $\Phi(T^\infty_t)\cap \tilde T^{\infty,v}_t$ and $\Phi_t(\tilde T^{\infty,v}_t)$ have the same vertex set (the vertex set of $\tilde T^{\infty,v}_t$), therefore any difference comes from the set of edges. Given two graphs $G_1$ and $G_2$ with the same set of vertices and an edge $e$, we are going to say that $e$ \emph{separates} $G_1$ and $G_2$ if $e$ is present in one of the graphs but not the other. Let us assume that $\Phi(T^\infty_t)\cap \tilde T^{\infty,v}_t \neq\Phi(\tilde T^{\infty,v}_t)$, and let $e$ be an edge separating $\Phi(T^\infty_t)\cap \tilde T^{\infty,v}_t$ and $\Phi(\tilde T^{\infty,v}_t)$.  

The graph $T^\infty_t$ is in $\Omega^-$ a.s., so is $\tilde T^{\infty,v}_t$, and therefore these two graphs are in $\Omega^+$. By definition of $\Omega^+$, there exists an integer $l$ such that the knowledge of the balls of radius $l$ centered at $e$ and $(v,w)$ in $T^\infty_t$ and $\tilde T^{\infty,v}_t$ are sufficient to know whether $e$ and $(v,w)$ are present in $\Phi(T^\infty_t)$ and $\Phi(\tilde T^{\infty,v}_t)$. Let $B^a=B_l(T^{\infty}_t,e)\cup B_l(T^{\infty}_t,(v,w))$ and $B^b= B_l(\tilde T^{\infty,v}_t,e)\cup B_l(\tilde T^{\infty,v}_t,(v,w))$. By Definition \ref{deficertain}, $e$ and $(v,w)$ are present in $\Phi(T^\infty_t)$ (resp. $\Phi(\tilde T^{\infty,v}_t)$) if and only if they are present in $\Phi(B^a)$ (resp. $\Phi(B^b)$). It should be noted that $B^b=B^a\cap \tilde T^{\infty,v}_t$. $B^a$ is a finite graph, therefore the set of labels of $B^a$ (and $B^b$) is finite. Let $s$ be the first time such that $\Phi_s(B^a)\cap \tilde T^{\infty,v}_t\neq \Phi_s(B^b)$ holds, and let $e'$ be an edge of $T^{\infty,v}_t$ that separates $\Phi_s(B^a)\cap \tilde T^{\infty,v}_t$ and $\Phi_s(B^b)$, \ie{} $e'$ is removed at time $s$ in either $\Phi_\cdot(B^a)$ or $\Phi_\cdot(B^b)$.  Let $x$ be an endpoint of $e'$ that reaches the forbidden degree at time~$s$ in one of the graphs, but not the other. Therefore the degree of $x$ is different in $\Phi(B^a)$ and in $\Phi(B^b)$ strictly before time $s$. This implies that an edge $e''$ adjacent to $x$ is present in either $\Phi_{s^-}(B^a)$ or $\Phi_{s^-}(B^a)$ but not in the other graph. If $x\neq w$, every edge adjacent to $x$ in either $B^a$ or $B^b$ is in $\tilde T^{\infty,v}_t$, and $e''$ contradicts the definition of $s$ and $e'$. Therefore $x=w$ and $e'=(v,w)$ (as $(v,w)$ is the only edge of $\tilde T^{\infty,v}_t$ adjacent to $w$). As the degree of $w$ in $B^b$ is $1$, $w$ cannot reach the forbidden degree in $\Phi(B^b)$. Therefore $w$ reaches the forbidden degree in $\Phi(B^a)$ and $(v,w)$ is removed from $\Phi(B^a)$ \ie{} $(v,w)$ is removed in $\Phi(T^\infty_t)$.

\end{proof}

Let us add to the set of rooted graphs an element $\X$, used to denote a graph that is no longer \emph{relevant}: let $\tilde T^{k,v}_t$ (resp. $T^{k,v}_t$) be defined as equal to $\Phi(\tilde T^{\infty,v}_t)$ (resp. $\tilde T^{k,v}_t\cap T^{\infty,v}_t$) if the edge between $v$ and its parent has not been removed from $\Phi(\tilde T^{\infty,v}_t)$ and equal to $\X$ otherwise. Let $\tau_1,\dots,\tau_l$ denote the labels of the edges adjacent to the root in $T^{\infty}_t$, ordered in increasing order, and let $v_1,\dots,v_l$ denote the other endpoints of these edges. For any $s\leq t$, let $\tilde S_s$ denote the set of integers such that $\{v_i:i\in \tilde S_s\}$ is the set of neighbors of the root in $\Phi_s(T^\infty_t)$ \ie{} $\tilde S_s$ records which edge is present in the forbidden degree version of $T^\infty_t$ at time $s$. For any $i\leq l$, let $\rho_i=\inf\{s:T^{k,v_i}_s=\X\}$. Let $\Gamma_\tau=\{\tau_i,1\leq i \leq l\}$ and $\Gamma_\rho=\{\rho_i,1\leq i \leq l \}$. 

\begin{lem}
\label{lemtaurho}
Almost surely, for every $1\leq i<j\leq l$, $\tau_i\neq \tau_j$, $\rho_i\neq \tau_j$, and either $\rho_i\neq \rho_j$ or $\rho_i=\rho_j=\infty$.
\end{lem}
\begin{proof}
By definition, $\tau_i$ is the label of $(\emptyset,v_i)$ and $\rho_i$ is the first time $(\emptyset,v_i)$ is removed in $\Phi(\tilde T^{\infty,v_i}_t)$, \ie{} is either infinite or equal to the label of an edge adjacent to $v_i$. As $T^\infty\in \Omega$, no two edges have the same label.
\end{proof}
Lemmata \ref{lemDecoupage} and \ref{lemtaurho} allow us to consider the following dynamic set $(S_s)_{0\leq s \leq t}$:
\begin{itemize}
\item at time $0$, $S_0$ is empty;
\item $s\rightarrow S_s$ is piecewise constant, and its set of jumps is included in $\Gamma_\rho\cup\Gamma_\tau$;
\item for every $\tau_i\notin \Gamma_\rho$, if $|S_{\tau_i^-}|=k-1$, then $S_{\tau_i}=\emptyset$, otherwise $S_{\tau_i}=S_{\tau_i^-}\cup\{i\}$;
\item for every finite $\rho_i\notin \Gamma_\tau$,  $S_{\rho_i}=S_{\rho_i^-}\setminus\{i\}$;
\item for every $i$ such that $\tau_i=\rho_i$, if $|S_{\tau_i^-}|=k-1$, then $S_{\tau_i}=\emptyset$, otherwise $S_{\tau_i}=S_{\tau_i^-}$.	
\end{itemize}
In other words, the element $i$ is added at time $\tau_i$, removed at time $\rho_i$, and $S$ becomes empty whenever it reaches size $k$, even fleetingly\footnote{Actually, $S$ never reaches size $k$, going directly from size $k-1$ to $0$, but we informally say that $S$ reaches size $k$ for an instant, as described in the footnote \ref{conventionk} page \pageref{conventionk}}.
\begin{cor}
\label{corbranching}
For any time $s$, $S_s=\tilde S_s$ and the connected component of the root is the same in $\Phi(T^\infty_s)$ and in $\cup_{i\in S_s} \Phi(\tilde T^{\infty,v_i}_s)$.
\end{cor}
\begin{proof}
Let first assume that for some $s$, $S_s\neq \tilde S_s$. The set~$S$ can only change at the times $(\tau_i)_{1\leq i \leq l}$ and $(\rho_i)_{1\leq i \leq l}$; and the set~$\tilde S$ can only change at times equal to a label of an edge in $B_2(T^\infty_t,\emptyset)$. These two sets of times are finite, therefore there exists a smallest $s$ such that $S_s\neq \tilde S_s$. Let $i\in S_s\Delta \tilde S_s$, in which $\Delta$ denotes the symmetric difference. The element~$i$ is added to $S$ and $\tilde S$ at time $\tau_i$, therefore the difference eventually comes from the removal times: when $i$ is in $S$ and $\tilde S$ at time $s^-$ and is removed from one, but not the other at time $s$. An element can be removed from $S$ or $\tilde S$ at time $s$ for the following reasons:
\begin{enumerate}
\item $S$ (resp. $\tilde S$) reaches size $k$ at time $s$. Therefore $|S_{s^-}|=k-1$ (resp. $|\tilde S_{s^-}|=k-1$) and an element is added to $S$ (resp. $\tilde S$) at time $s$. As $S_{s^-}=\tilde S_{s^-}$ by definition of $s$, and the times of addition are identical for $S$ and $\tilde S$, $S$ reaches size $k$ at time $s$ if and only if $\tilde S$ reaches size $k$ at time $s$.
\item $\rho_i=s$, that is $(\emptyset,v_i)$ is removed from $\Phi(\tilde T^{\infty,v}_\cdot)$ at time $s$. By Corollary~\ref{corDecoupage}, $(\emptyset,v_i)$ is also removed from $\Phi(T^\infty_s)$, and $i$ is removed from $\tilde S$.
\item $(\emptyset,v_i)$ is removed from $\Phi(T^\infty_\cdot)$ at time $s$ because the vertex $v_i$ reaches degree $k$ in $\Phi(T^\infty_\cdot)$, \ie{} $v_i$ has degree $k-1$ in $\Phi(T^\infty_{s^-})$, and an edge is added to $v_i$. $i\in S_{s^-}$, therefore $(\emptyset,v_i)$ is present in $\Phi(\tilde T^{\infty,v_i}_{s^-})$. By Lemma \ref{lemDecoupage},   $\Phi(T^\infty_{s^-})\cap \tilde T^{\infty,v_i}_\infty=\Phi(\tilde T^{\infty,v_i}_{s^-})$, and therefore $v_i$ has degree $k-1$ in $\Phi(\tilde T^{\infty,v_i}_{s^-})$, and reaches degree $k$ at time $s$. Therefore $\rho_i=s$, and $i$ is removed from $S$ at time $s$.
\end{enumerate}
Therefore, for all $s$, $S_s=\tilde S_s$. Moreover, for any $i\in S_s$, $(\emptyset,v_i) \in \Phi(\tilde T^{\infty,v_i}_s)$, therefore,  by Lemma \ref{lemDecoupage}, $\Phi(T^\infty_{s})\cap \tilde T^{\infty,v_i}_\infty=\Phi(\tilde T^{\infty,v_i}_{s})$, ending the proof of Corollary~\ref{corbranching}.

\end{proof}
\begin{defi}
Given  a rooted graph $(G,\emptyset)$ in $\Omega$ and a non-negative number $y$, we introduce the following notations:
\begin{itemize}
\item  $\Theta^y(G,\emptyset)$ denotes the graph $G$ rooted at $\emptyset$ with an extra vertex $w$ and an edge between $w$ and $\emptyset$, labelled by $y$;
\item if $\Theta^y(G,\emptyset)\in \Omega^+$, then $\tilde \Phi^y(G,\emptyset)$ and $\Phi^y(G,\emptyset)$ are defined as follows: if $(w,\emptyset)$ is removed from $\Phi(\Theta^y(G,\emptyset))$, then let $\tilde \Phi^y(G,\emptyset)=\X$ and $\Phi^y(G,\emptyset)=\X$; otherwise, let $\tilde \Phi^y(G,\emptyset)=\Phi(\Theta^y(G,\emptyset))$ and  $\Phi^y(G,\emptyset)= \tilde \Phi^y(G,\emptyset)\cap G$;
\item let $T^{\infty,y}_t=\Theta^y(T^\infty_t)$, $\tilde T^{k,y}_t=\tilde \Phi^y(T^\infty_t)$ and $T^{k,y}_t=\Phi^y(T^\infty_t)$.
\end{itemize}
\end{defi}

Let $m(t,y)=\P(T^{k,y}_t\neq \X)$. The tree $T^{\infty}_t$ contains no edge with probability $e^{-t}>0$ and in that case $T^{k,y}_t\neq \X$, therefore $m(t,y)>e^{-t}>0$. Let $T^{k+,y}_t$ (resp. $\tilde T^{k+,y}_t$) be the random tree $T^{k,y}_t$ (resp. $\tilde T^{k,y}_t$) conditioned on not being equal to $\X$.

\begin{lem}
\label{lembranchingproperty}
Conditionally on $(\tau_i)_{1\leq i \leq l}$ and $(S_s)_{0\leq s \leq t}$:
\begin{itemize}
\item the graphs $(T^{k,v_i}_t)_{i\in S_t}$ (resp. $(\tilde T^{k,v_i}_t)_{i\in S_t}$) are independent:
\item for each $i\in S_t$, $T^{k,v_i}_t$ (resp. $\tilde T^{k,v_i}_t$) has the same law as $T^{k+,\tau_i}_t$ (resp. $\tilde T^{k+,\tau_i}_t$).
\end{itemize}
\end{lem}
\begin{proof}

 Let $\F$ be the $\sigma$-algebra generated by the sequences $(\tau_i)_{1\leq i \leq l}$ and $(\rho_i)_{1\leq i \leq l}$, and $\F_1$ be the $\sigma$-algebra generated by $(\tau_i)_{1\leq i \leq l}$ and $(S_s)_{0\leq s \leq t}$. As $\F_1\subset \F$, it is sufficient to prove Lemma \ref{lembranchingproperty} with $\F$ instead of $\F_1$. By the branching property, conditionally on $(\tau_i)_{1\leq i \leq l}$, the trees $(\tilde T^{\infty,v_i}_t)_{1\leq i \leq l}$ are independent, and for each $i$, $\tilde T^{\infty,v_i}_t$ is a copy of $\tilde T^{\infty,\tau_i}_t$. Each $\rho_i$ only depends on $\tilde T^{\infty,v_i}_t$, and $S_t$ is $\F$-measurable. Therefore, conditionally on $\F$:  the trees $(\tilde T^{\infty,v_i}_t)_{1\leq i \leq l}$ are independent and for each $i$, $\tilde T^{\infty,v_i}_t$ has the same law as $\tilde T^{\infty,\tau_i}_t$ conditionally on $\inf\{s:\tilde T^{k,\tau_i}_s=\X\}=\rho_i$. It follows that conditionally on $\F$:
 the trees $(\tilde T^{k,v_i}_t)_{i\in S_t}$ are independent and for each $i$ in $S_t$, $\tilde T^{k,v_i}_t$ has the same law as $T^{k+,\tau_i}_t$.
  
As $T^{k,v_i}_t=\tilde T^{k,v_i}_t \cap T^{\infty,v_i}_t$ and $T^{k,\tau_i}_t=\tilde T^{k,\tau_i}_t \cap T^{\infty,\tau_i}_t$, Lemma \ref{lembranchingproperty} with $\tilde T^{k,v_i}_t$ imply  Lemma \ref{lembranchingproperty} with $T^{k,v_i}_t$.
\end{proof}

Lemma \ref{lembranchingproperty} implies the following theorem:
\begin{thm}
Let $B(y)$ be the law of the set of the labels of  the edges adjacent to the root of $T^{k+,y}_t$.
Let $\BP$ be the multitype branching process with offspring law $B(\cdot)$. Then $T^{k+,y}_t$ has same law as $\BP$ starting with a root of label $y$.
\end{thm}
\begin{proof}
$T^{k,y}_t$ is either equal to $T^k_t$ or to $\X$. With the notations of the proof of Lemma \ref{lembranchingproperty}, the event $T^{k,y}_t=\X$ is $\F_1$-measurable, therefore, conditionally on $T^{k,y}_t\neq \X$ and the sequence $(\tau_1,\dots,\tau_l)$ of labels of the edges adjacent to the root, the subtrees starting at the children of the root are independent copies of $T^{k+,\tau_i}_t$.
\end{proof}
Lemma \ref{lembranchingproperty} also implies that $T^k_t$ is a two-stages branching process:
\begin{thm}
Conditionally on the set $(\tau_i)_{i\in S_t}$ of labels of edges adjacent to the root in $T^k_t$, the subtrees starting at the root's children are independent, and copies of $T^{k+,\tau_i}_t$.
\end{thm}

\subsubsection{Properties of the branching process}
Let $\mu_{t,y}$ (resp. $\mu_{t,y}^+$, $\nu_t$) be the law of the set of labels of edges adjacent to the root in $T^{k,y}_t$ (resp. $T^{k+,y}_t$, $T^{k}_t$). This measure can be decomposed according to the degree of the root in these trees:
\begin{eqnarray*}
\mu_{t,y}&=&\sum_{i=-1}^{k-2}\mu^i_{t,y};\\
 \mu^+_{t,y}&=&\sum_{i=0}^{k-2}\mu^{i+}_{t,y};\\
\nu_{t}&=&\sum_{i=0}^{k-1}\nu^{i}_{t}.
\end{eqnarray*}
\noindent{}in which $\mu_{t,y}^{i}$ (resp. $\mu_{t,y}^{i+}$,  $\nu_{t}^i$) only puts mass on sets of cardinality $i$ and $\mu_{t,y}^{-1}$ is equal to $(1-m(t,y))\delta_\X$.
%

\begin{defi}
For $X$ be a finite subset of $[0,t]$, let us define the random rooted tree $T^{\infty,X}_t$ as follows:
\begin{itemize}
\item let $P$ be a Poisson point process of intensity $1$ on $[0,t]$;
\item conditionally on $P$, let $(T^z)_{z\in P\cup X}$ be an i.i.d family of copies of $T^\infty_t$;
\item then, adding a root $\emptyset$ to the forest $(T^z)_{z\in P\cup X}$, and, for each $z\in P\cup X$, adding an edge with label $z$ between $\emptyset$ and the root of $T^z$, one obtains $T^{\infty,X}_t$. $T^{\infty,X}$ is rooted at $\emptyset$.
\end{itemize}
Let $E^{X}_t$ denote the event \emph{"the only edges adjacent to the root in  $\Phi(T^{\infty,X}_t)$ are the edges labelled by $X$"}.
If $y\notin X$, let $E^{y,X}_t$ denote the event \emph{"$\Phi^y(T^{\infty,X}_t)\neq \X$ and the only edges adjacent to the root in  $\Phi^y(T^{\infty,X}_t)$ are the edges labelled by $X$"}.
\end{defi}
Informally, $T^{\infty,X}_t$ is the conditional distribution of the tree $T^\infty_t$, given that the root is incident to edges with labels in $X$. This informal explanation will be justified rigorously with the help of Campbell formulas in the following pages. If $y\notin X$, a.s., $\Theta^y(T^{\infty,X}_t)\in \Omega^+$ and $\Phi^y(T^{\infty,X}_t)$ is well-defined.

\begin{defi}
Let $0\leq j\leq i\leq k-2$ be two integers. $Q^j_i$ is the set $\{(t,y,x_1,\dots,x_i):0\leq x_1\leq\dots x_j\leq y\leq x_{j+1}\leq \dots\leq x_i\leq t\}$, and $Q_i$ is the simplex $\{(t,x_1,\dots,x_i):0\leq x_1\leq\dots \leq x_i\leq t\}$.
\end{defi}

\begin{lem}
\label{continuiteproba}
For any integer $i$ and $j$, $(t,y,x_1,\dots,x_i)\rightarrow \P(E^{y,\{x_1,\dots,x_i\}}_t)$ (resp.  $(t,x_1,\dots,x_i)\rightarrow \P(E^{\{x_1,\dots,x_i\}}_t)$) is continuous on $\mathring Q^{j}_i$ (resp. $\mathring Q_i$), and both are larger than $exp(-(i+1)t)$.

\end{lem}

\begin{lem}
\label{probaegaldensite}
$(y,t)\rightarrow m(t,y)$ is continuous on $\{0\leq y<t\}$ and is larger than $e^{-t}$.
For any integer $i$ and $y<t$, $\mu^i_{t,y}$, $\mu^{i+}_{t,y}$ and $\nu^i_t$ are absolutely continuous with respect to the Lebesgue measure on $[0,t]^i$, with respective densities:
\begin{eqnarray*}
\frac{\partial \mu^{i}_{t,y}}{\partial x}=f_i(t,y,x)&=&\P(E^{y,\{x_1,\dots,x_i\}}_t)\\
\frac{\partial \mu^{i+}_{t,y}}{\partial x}=g_i(t,y,x)&=&\P(E^{y,\{x_1,\dots,x_i\}}_t)/m(t,y).\\
\frac{\partial \nu^{i}_{t}}{\partial x}=h_i(t,x)&=&\P(E^{\{x_1,\dots,x_i\}}_t)
\end{eqnarray*}
\end{lem}
The functions $f_i$, $g_i$ and $h_i$ are not defined on the null set where two or more coordinates are equal. They are by definition symmetric in the variables $(x_j)_{1\leq j \leq i}$. By Lemma \ref{continuiteproba}, $f_i$ and $g_i$ are continuous on every set $\mathring Q^j_i$, and $h_i$ is continuous on $Q_i$, and $f_i$, $g_i$ and $h_i$ are larger than $exp(-(i+1)t)$.

The continuity in Lemma \ref{continuiteproba} will be proven by a coupling argument:
\begin{itemize}
\item Let $P^{coupl}$ be a Poisson point process of intensity $1$ on $(0,\infty)$. 
\item Let $(T^{z,coupl})_{z\in P^{coupl}}$ be a family of independent copies of the process $T_\infty^\infty$.
\item Let $(T^{j,coupl})_{j\leq i}$ be $i$ independent copies of $T_\infty^\infty$.
\end{itemize}

Let $T^{\infty,\{x_1,\dots,x_i\}}$ be constructed in the following way:
\begin{itemize}
\item $P$ is the restriction of $P^{coupl}$ to $(0,t)$.
\item For any $z\in P$, $T^z=T^{z,coupl}_t$.
\item For any $x_j$, $T^{x_j}=T^{j,coupl}_t$.
\end{itemize} 
This construction allows to couple every $T^{\infty,\{x_1,\dots,x_i\}}_t$ by using the same randomness.
Let $T$ be a positive real number, and let $(t,y,x_1\dots,x_l)$ and $(\tilde t,\tilde y,\tilde x_1\dots,\tilde x_l)$ be two elements of $\mathring Q^{j}_i$ such that $T\geq t$ and $T\geq \tilde t$.  Using the notation $[a,b]$  to describe the non-empty interval $[\min(a,b),\max(a,b)]$, let $I=[\tilde t, t]\cup[y,\tilde y]\cup_i[x_i,\tilde x_i]$. Let $l$ be a positive integer. Let $X=\{x_1\dots,x_i\}$ and $\tilde X=\{\tilde x_1\dots,\tilde x_i\}$

\begin{lem}
\label{couplagecontinuite}
 If the following properties hold, 

\begin{enumerate}
\item $P^{coupl}\cap I=\emptyset$;
\item For every $z\in P^{coupl}\cap [0,T ]$, no edge in the first $l$ generation of $T^{z,coupl}_T$ has a label in~$I$.
\item For every $j$, no edge in the first $l$ generation of $T^{j,coupl}_T$ has a label in~$I$.
\item For every $z\in P^{coupl}\cap [0,T ]$, there is no propagation path of length $l$ starting from the root of $T^{z,coupl}_T$.
\item For every $j$, there is no propagation path of length $l$ starting from the root of $T^{j,coupl}_T$.
\end{enumerate}
then $E^{X,y}_{t}$ holds if and only if $E^{\tilde X,\tilde y}_{\tilde t}$ holds, and $E^{X}_t$ holds if and only if $E^{\tilde X}_{\tilde t}$ holds.
\end{lem}

\begin{proof}
The knowledge of $B^l(\Theta^y(T^{X}_t),\emptyset)$ (resp. $B^l(\Theta^{\tilde y}(T^{\tilde X}_{\tilde t}),\emptyset)$) is sufficient to know which edges are adjacent to the root in $\Phi(\Theta^y(T^{X}_t))$ (resp. $\Phi(\Theta^{\tilde y}(T^{\tilde X}_{\tilde t}))$), by properties 4 and 5. By properties 1, 2 and 3, the unlabelled versions of  $B^l(\Theta^y(T^{X}_t),\emptyset)$ and $B^l(\Theta^{\tilde y}(T^{\tilde X}_{\tilde t}),\emptyset)$ are equal. For any graph $G\in \Omega^+$, the set of removed edges in $\Phi(G)$ only depends on the unlabelled graph $G$ and the respective order of the edges' labels, not on the actual labels. By hypothesis, the elements of $(t,y,x)$ and $(\tilde t,\tilde y,\tilde x)$ are in the same respective order. Therefore, by properties 2 and 3, the respective order of the labels is the same in $B^l(\Theta^y(T^{X}_t),\emptyset)$ and $B^l(\Theta^{\tilde y}(T^{\tilde X}_{\tilde t}),\emptyset)$, and therefore $E^{x,y}_t$ holds if and only if $E^{\tilde x, \tilde y}_{\tilde t}$ holds.

The proof for $E^X_t$ is identical upon removing $\Theta^y$.
\end{proof}
\begin{proof}[Proof of Lemma \ref{continuiteproba}]

For a given $T$, the probability of the properties 4 and 5 tends to $1$ when $l$ tends to $\infty$. For a given $T$ and $l$, the properties of properties 1, 2 and 3 tends to $1$ when $(\tilde t,\tilde y,\tilde x_1,\dots\tilde x_l)$ tends to  $( t, y, x_1,\dots x_l)$, proving the continuity of $\P(E^{y,X}_t)$ and $\P(E^{X}_t)$.

If $P=\emptyset$ and every subtree starting from a children is empty, the events $E^{y,X}_t$ and $E^X_t$ hold. $|P|$ is a Poisson variable of parameter $t$, and $T^\infty_t$ is empty with probability $exp(-t)$, therefore $\P(E^{y,X}_t)=$ (resp. $\P(E^X_t)$) is larger than $exp(-(i+1)t)$ for all set $X$ of size $i\leq k-2$ and any $y\notin X$ (resp. for all set $X$ of size $i\leq k-1$).
\end{proof}

\begin{proof}[Proof of Lemma~\ref{probaegaldensite}]

For any finite set $X\subset [0,t]$ and $y\in [0,t]\setminus X$, we consider the following construction:
\begin{itemize}
\item Let $(T_t^z)_{z\in X}$ be an i.i.d. family of copies of $T^\infty_t$.
\item Let $T^{\infty,(X)}_t$ is the tree obtained by taking a vertex $\emptyset$, every tree $(T_t^z)_{z\in X}$ and adding an edge, labelled by $z$, between $\emptyset$ and the root of $T_t^z$. $T^{\infty,(X)}_t$ is rooted at $\emptyset$.
\item Let $Z^y_X\subset X$ (resp. $Z_X$) be equal to the random set of labels of the edges adjacent to the root of $\Phi^y(T^{\infty,(X)}_t)$ (resp. $\Phi(T^{\infty,(X)}_t)$), if defined.
\end{itemize}
\begin{rem}
The difference between $T_t^{\infty,X}$ and $T_t^{\infty,(X)}$ is that the latter have only edges labelled by elements of $X$ adjacent to the root, whereas the former have edges labelled by $X$ and additional edges, according to the Poisson point process $P$.
\end{rem}
By the branching property, conditionally on $P$, the set of edges adjacent to the root of $T^\infty_t$, $T^\infty_t$ has same distribution as $T^{\infty,(P)}_t$, and therefore the set of edges adjacent to $\Phi^y(T^\infty_t)$ has same distribution as $Z^y_P$.

Let $\lambda$ be a bounded positive continuous function $[0,t]^i\rightarrow [0,\infty)$. 
$$\int_{[0,t]^i}\hspace{-0.5cm}\lambda(x_1,\dots,x_i)\mu^i_{y,t}(d(x_1\dots x_i))=\int_{[0,t]^i}\hspace{-0.5cm}\lambda(x_1,\dots,x_i)f(t,y,x_1\dots,x_l)dx_1\dots dx_i.$$
Let $A$ denote the left-hand side:
\begin{eqnarray*}
A&=&\E\left(\sum_{x_1,\dots,x_i\in P}^{\neq} \lambda(x_1,\dots,x_i)\1{Z^y_P=\{x_1,\dots,x_i\}}\right)\\
&=&\E\left(\E\left(\sum_{x_1,\dots,x_i\in P}^{\neq} \lambda(x_1,\dots,x_i)\1{Z^y_P=\{x_1,\dots,x_i\}}\bigg|P\right)\right)\\
&=&\E\left(\sum_{x_1,\dots,x_i\in P}^{\neq}\lambda(x_1,\dots,x_i)\P(Z^y_P=\{x_1,\dots,x_i\})\right)\\
&=&\E\left(\sum_{x_1,\dots,x_i\in P}^{\neq}\gamma(x_1,\dots,x_i,P)\right)
\end{eqnarray*}
\noindent{}in which $\gamma(x_1,\dots,x_i,Y)=\lambda(x_1,\dots,x_i)\P(Z^y_Y=\{x_1,\dots,x_i\})$ for every set $Y$ and real numbers $x_1,\dots,x_i$. By the reduced Campbell formula, \cite[formula (9.12)]{Baccelli}:
$$\E\left(\sum_{x_1,\dots,x_i\in P}^{\neq}\gamma(x_1,\dots,x_i,P)\right)=\int_{[0,t]^i}\E(\gamma(x_1,\dots,x_l,P_{x_1,\dots,x_i}))M^{(i)}(dx_1\dots dx_n).$$
\noindent{}in which, for a Poisson point process of intensity $\Lambda$ equal to the Lebesgue measure, $M^{(i)}=\Lambda^i$, by \cite[Proposition 9.1.3]{Baccelli} and $P_{x_1,\dots,x_i}$ has same law as $P\cup \{x_1,\dots,x_i\}$ by \cite[Corollary 9.2.5]{Baccelli}. Therefore:
\begin{eqnarray*}
A&=&\int_{[0,t]^i}\E(\gamma(x_1,\dots,x_l,P_{x_1,\dots,x_i}))dx_1\dots dx_i\\
&=&\int_{[0,t]^i}\lambda(x_1,\dots,x_i)\P\left(Z^y_{\{x_1,\dots,x_i\}\cup P}=\{x_1,\dots,x_i\}\right)dx_1\dots dx_i\\
&=&\int_{[0,t]^i}\lambda(x_1,\dots,x_i)\P(E^{y,X}_t)dx_1\dots dx_i
\end{eqnarray*}
Therefore $\mu^i_{y,t}$ is absolutely continuous with respect to the Lebesgue measure with density $(x_1,\dots,x_i)\rightarrow \P(E^{y,\{x_1\dots,x_i\}}_t)$. This density is a probability, and is therefore bounded by $1$, and is continuous by Lemma \ref{continuiteproba}. 
Therefore $(t,y)\rightarrow m(t,y)=\sum_{i= 0}^{k-2}\mu^i_{y,t}(Q_i)$ is continuous.

As the random tree $T^{k+,y}_t$ is equal to the tree $T^{k,y}_t$ conditioned on not being equal to $\X$, $\mu^{i+}_{y,t}$ is absolutely continuous with respect to the Lebesgue measure with density $(x_1,\dots,x_i)\rightarrow \P(E^{y,X}_t)/m(t,y)$. 

The proof for $\nu^i_t$ is identical to the proof for $\mu^i_{y,t}$ upon replacing $\Phi^y$ by $\Phi$, $E^{y,X}_t$ by $E^X_t$ and $Z^y_X$ by $Z_X$.

\end{proof}

\subsubsection{Link between $\mu_{.,t}$ and $\nu_t$}
\label{link}
\begin{lem}
\label{lemlink}
For every integer $i\geq 1$ and $i$uple $(x_1,\dots x_i)\in\,(0,t)^i$, $h_i(t,x_1\dots, x_i)=m(t,x_1)f_{i-1}(t,x_1,\dots,x_l)$.
\end{lem}

Let $X=\{x_1,\dots,x_i\}$ and $\tilde X=\{x_2,\dots,x_i\}$. By Lemma \ref{probaegaldensite}, the density $h_i(t,x_1\dots,x_i)$ is equal to $\P(E^X_t)$, \ie{} the probability that the only edges adjacent to the root in $\Phi(T^{\infty,X}_t)$ are the edges labelled by elements of $X$. 

Let $v$ be the other endpoint of the edge labelled by $x_1$, and $T^{\infty,x_1}_t$ the subtree of $T^{\infty,X}_t$ starting at $v$ and rooted at $v$. L	et $T^{\infty,X,\setminus x_1}_{t}$ be the tree $T^{\infty,X}_t\setminus T^{\infty,x_1}_t$. The figure \ref{figureTX} illustrates these trees.

\begin{figure}[h]
\label{figureTX}
 \begin{tikzpicture}[level/.style={sibling distance = 1.3cm/#1,
  level distance = 1.5cm}, edge from parent/.style={draw},]
\tikzstyle{every node}=[circle,draw,minimum width=0.6cm]
 \node {$\emptyset$}
    child{ node {$v$}
	    child{ node {}}
	    child{ node {}}
	    edge from parent node[left,draw=none]{$x_1$}
	    }
    child{node {}
   	    edge from parent node[left,draw=none]{$x_2$}}
    child{ node {}
	    child{ node {}}
	    child{ node {}}
	    child{ node {}}
	    edge from parent node[right,draw=none]{$x_3$}
	    }
    ;
\begin{scope}[shift={(5,0)}]
 \node {$\emptyset$} 
    child{ node {$v$} 
	    child{ node {}}
	    child{ node {}}
	    edge from parent[draw=none]
	    }
    child{node {}
   	    edge from parent node[left,draw=none]{$x_2$}}
    child{ node {}
	    child{ node {}}
	    child{ node {}}
	    child{ node {}}
	    edge from parent node[right,draw=none]{$x_3$}
	    }
    
;
\end{scope}
 \end{tikzpicture}
 \caption{Example of $T^{\infty,X}_t$, $T^{\infty,x_1}_t$ and $T^{\infty,X,\setminus x_1}_t$, from left to right.}
 \end{figure}
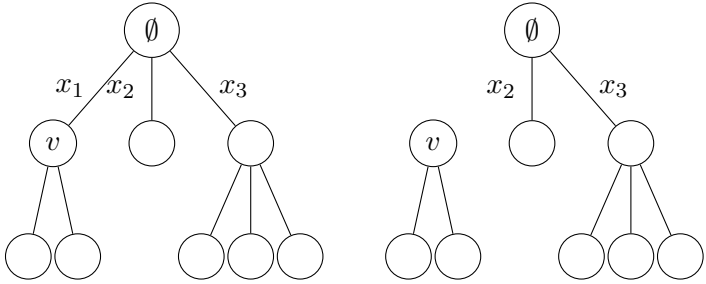
By definition of $T^{\infty,X}_t$, $T^{\infty,x_1}_t$ and $T^{\infty,X,\setminus x_1}_t$ are independent, and have same distribution as respectively $T^\infty_t$ and $T^{\infty,\tilde X}_t$.

$E^X_t$ implies that the edge labelled by $x_1$ is present in $\Phi(T^{\infty,X}_t)$, therefore, by Lemma \ref{lemDecoupage}, $E^X_t$ is equivalent to:
\begin{itemize}
\item the edge labelled by $x_1$ is present in $\Phi(\Theta^{x_1}(T^{\infty,x_1}_t))$, \ie{} $\Phi^{x_1}(T^{\infty,x_1}_t)\neq \X$, and
\item the edges adjacent to the root in $\Phi(\Theta^{x_1}(T^{\infty,X,\setminus x_1}))$ are the edges labelled by $x_1,\dots,x_i$.
\end{itemize}
This two events are independent and of probability $m(t,x_1)$ and $\P(E^{\{x_2,\dots,x_i\}}_{x_2})$, therefore $\P(E^X_t)=m(t,x_1)\P(E^{x_1,\tilde X}_t)$, proving Lemma \ref{lemlink}.

This formula allows to compute the density of any $ \nu^i_t$ for $i\geq 1$ from the density of $\mu^i_{t,y}$. As the probability that the root of $T^k_t$ is isolated is $1-\sum_{i=1}^{k-1}\nu^i(Q_i)$, it is sufficient to study the measure $\mu$ to know $\nu$. This computation also implies that $m(t,x_1)f_{i-1}(t,x_1,x_2\dots,x_i)$ is symmetric in the $x_i$.

\section{The equivalence between supercriticality of the local limit and the existence of the giant component}
\label{sectequivalence}

Let $a^k_t$ be the probability that the tree $T^k_t$ is infinite.
The goal of this section is to prove Theorem \ref{thmequivalence}.

\begin{customthm}{\ref{thmequivalence}}
For all $t\geq 0$, $\frac{\Cmax\left(G^k_{n,t}\right)}n$ converges in probability to $a^k_t$.
\end{customthm}
For readability, $ \Cmax\left(G^k_{n,t}\right)$ is abridged in $\Cmax$.
\begin{rem} 
Theorem \ref{thmequivalence} is meaningful both in the subcritical case and in the supercritical case. 

If $T^k_t$ is critical or subcritical, $a^k_t=0$, and therefore the largest component's size is $o_p(n)$.

If $T^k_t$ is supercritical, $a^k_t>0$, so there is a giant component of size equivalent to $a_t^k n$.

This result is analogous to the result on the size of the largest component in the Erd\H{o}s-R\'enyi graph, where $a_t$ is the probability of survival of a Galton-Watson whose offspring is Poisson distributed with parameter~$t$.
\end{rem}

The local limit results imply that $a^k_t$ is close to the expected proportion of vertices in large components. To prove Theorem \ref{thmequivalence}, one needs to prove that $a^k_t$ is a.s. close to the actual proportion of vertices in large components (not only in expectation), and that almost every vertices in large components are in the same component.

\subsection{The subcritical or critical case}

For any integer $i$, let $N^i$ be the number of vertices of $G^k_{n,t}$ in components of size at least $i$.

\begin{lem}
\label{lemmeNi}
For all $\epsilon>0$, $\P(\frac{N^{\sqrt {n}}}n \geq a^k_t+\epsilon)\xrightarrow[n\rightarrow\infty]{} 0$.
\end{lem}
Lemma \ref{lemmeNi} allows to bound $\Cmax$ from above:
\begin{cor}
\label{coroMajorationCmax}
For all $\epsilon>0$, $\P(\frac{\Cmax}n \geq a^k_t+\epsilon)\xrightarrow[n\rightarrow\infty]{} 0$.
\end{cor}

 For all $i$, if $\Cmax\geq i$, then $N^i\geq \Cmax$. Therefore for all integer $i$, ${\Cmax\leq \max(N^i,i)\leq i+N^i}$. Using $i=\sqrt n$ gives Corollary \ref{coroMajorationCmax}.

Corollary \ref{coroMajorationCmax} is sufficient to prove Theorem \ref{thmequivalence} in the subcritical or critical case, as in that case $a^k_t=0$ and $\frac \Cmax n$ is a positive random variable.

\begin{proof}[Proof of Lemma \ref{lemmeNi}] 	
Let $v_1$ and $v_2$ be two independent uniform random vertices of $G^k_{n,t}$, and $i$ a non-negative integer. Let $C(v_1)$ (resp. $C(v_2)$) denote the component of $v_1$ (resp. $v_2$) in $G^k_{n,t}$:
\begin{eqnarray*}
\P(|C(v_1)|\geq i|G^k_{n,t})&=&\frac{N^i}{n}\\
\P(|C(v_1)|\geq i)&=&\E\left(\frac{N^i}{n}\right)\\
\P(|C(v_1)|\geq i\text{ and } |C(v_2)|\geq i|G^k_{n,t})&=&\left(\frac{N^i}{n}\right)^2\\
\P(|C(v_1)|\geq i\text{ and } |C(v_2)|\geq i)&=&\E\left(\frac{N^i}{n}\right)^2.
\end{eqnarray*}

Being in a component of size at least $i$ is a local event, therefore by the birooted local convergence proven in Lemma \ref{localconvergenceER}:
\begin{eqnarray*}
\P(|C(v_1)|\geq i)&\xrightarrow[n\rightarrow \infty]{}& \P(|T^k_t|\geq i)\\
\P(|C(v_1)|\geq i\text{ and } |C(v_2)|\geq i)&\xrightarrow[n\rightarrow \infty]{} &\P(|T^k_t|\geq i)^2.
\end{eqnarray*}
Therefore, for all $i$,
\begin{eqnarray*}
\E\left(\frac{N^i}n\right)&\xrightarrow[n\rightarrow \infty]{}&\P(|T^k_t|\geq i)\\
\text{Var}\left(\frac{N^i}n\right)&\xrightarrow[n\rightarrow \infty]{}&0\\
\text{and therefore }\frac{N^i}n&\xrightarrow[n\rightarrow \infty]{p}&\P(|T^k_t|\geq i).
\end{eqnarray*}
$N^i$ is non-increasing in $i$, and $\P(|T^k_t|\geq i)\xrightarrow[i\rightarrow \infty]{} a^k_t$, therefore for all $\epsilon>0$, $\P\left(\frac{N^{\sqrt {n}}}n \geq a^k_t+\epsilon\right)\xrightarrow[n\rightarrow\infty]{} 0$.
\end{proof}

\subsection{The supercritical case}
If $a_t^k>0$, a lower bound is needed. The following lemma gives such a lower bound.

\begin{lem}
\label{lemmelien}
Let $v_1$ and $v_2$ be two independent uniform vertices of $G^k_n$. Then $\liminf_n\P(\text{$v_1$ and $v_2$ are in the same component})\geq (a^k_t)^2$.
\end{lem}

Let us first see why Lemma \ref{lemmelien} implies Theorem \ref{thmequivalence}. Let $(C_i)_{i\geq 1}$ be the sequence of the sizes of the component in $G^k_{n,t}$ (in any order). Conditionally on $G^k_{n,t}$, the probability that $v_1$ and $v_2$ are in the same component is $\frac 1 {n^2}\sum C_i^2$.

Let $\epsilon>0$. Lemmata \ref{lemmeNi} and \ref{lemmelien} imply that for $n$ large enough:
\begin{eqnarray}
\notag(a^k_t)^2-\epsilon&\leq&\frac 1 {n^2}\E(\sum_i C_i^2)\\
\notag&=&\frac 1 {n^2}\E\left(\sum_{i:C_i\geq\sqrt n}C_i^2+\sum_{i:C_i<\sqrt n}C_i^2\right)\\
\notag&\leq&\frac 1 {n^2}\E\left(\sum_{i:C_i\geq\sqrt n}C_i^2+\sqrt n\sum_{i:C_i<\sqrt n}C_i\right)\\
\notag&\leq&\frac 1 {n^2}\E\left(\sum_{i:C_i\geq\sqrt n}C_i^2+\sqrt n\cdot n\right)\\
\notag&\leq&\E\left(\frac{\Cmax}n\frac{\displaystyle\sum_{i:C_i\geq\sqrt n}C_i}n\right)+n^{-\frac 1 2}\\
\notag&=&\E\left(\frac{\Cmax}n\frac{N^{\sqrt n}}n\right)+n^{-\frac 1 2}\\
\label{akt}(a^k_t)^2-\epsilon&\leq&\E\left(\frac{\Cmax}n\right)(a^k_t+\epsilon)+\epsilon
\end{eqnarray}
The last inequality holding for $n$ large enough because $\frac{\Cmax}n\leq 1$ a.s., $\frac{N^{\sqrt n}}n\leq 1$ a.s. and $\P(\frac{N^{\sqrt n}}n\leq a^k_t+\epsilon)\rightarrow 1$.

By taking $\epsilon\rightarrow 0$, (\ref{akt}) implies that $\liminf \E(\frac {\Cmax}n)\geq a^k_t$. As $\frac {\Cmax}n$ is smaller than $1$ a.s. and smaller than $a^k_t+\epsilon$ with high probability by Corollary \ref{coroMajorationCmax}, this implies that $\frac{\Cmax}n\xrightarrow{p}a^k_t=\P(|T^k_t|=\infty)$.

The proof of Lemma \ref{lemmelien} will follow these steps:
\begin{enumerate}
\item With probability close to $(a^k_t)^2$, $v_1$ and $v_2$ are in components of $G^k_{n,t'}$ of size larger than a threshold (that depends on the number of vertices) at time $t'=t-\epsilon$;
\item In that case, with probability close to one, a path between $w_1$ and $w_2$ will exist at time $t$ in $G^k_t$.
\end{enumerate}
In order to do the first step, and use the local limit, we need to construct a graph that approximates $T^k_{t'}$ and $G^k_{n,t'}$ simultaneously.

\subsubsection{A few standard results on multitype branching processes}
\label{resultmultitype}
This part will summarize the results on multitype branching processes that will be used, using the results and notations of \cite{Harris}. Let $T$ be a multitype branching process, with type set $\Xset=[0,\tau]$ for some positive $\tau$.  For every integer $i$, let $Z_i$ be the random set equal to the labels of the elements of  $i$th generation of $T$. Let $\P_x$ (resp. $\E_x$) denote the probability (resp. expectation) when the root of $T$ has type $x$. For all $A\subset \Xset$, let $M(x,A)=\E_x(|A\cup Z_1|)$. We will assume that the branching processes considered always satisfy the following two conditions:
\begin{enumerate}[(C1)]
\item For all $x\in \Xset$, $|Z_1|\leq k$ $\P_x$-a.s.
\item There exists two positive real numbers $a$ and $b$ such that for all $x\in \Xset$, $M(x,\cdot)$ is absolutely continuous with respect to the Lebesgue measure, and its density $m(x,y)$ is a uniformly positive bounded function, ${0<a\leq m(x,y)\leq b<\infty}$.
\end{enumerate}

Conditions (C1) and (C2) will be direct consequences of Lemmata \ref{continuiteproba} and \ref{probaegaldensite} for all branching processes we will consider. (C1) and (C2) implies that $T$ satisfies technical conditions 10.1 and 13.1 with the notations of \cite{Harris}.

By \cite[Theorem 10.1]{Harris}, the operator $M$ has a real positive eigenvalue~$\rho$, larger than any other eigenvalue. This eigenvalue~$\rho$ will be called the spectral radius associated to $T$. Let $q(x)=\P_x(|T|<\infty)$ denote the extinction probability function of $T$. By \cite[Theorem 12.1, Theorem 14.1 and following remarks]{Harris}:
 \Needspace*{3\baselineskip}
\begin{lem}
\label{lemmevitessecroissance}
\leavevmode
\begin{enumerate}
\item If $\rho\leq 1$, then for all $x\in \Xset$, $q(x)=1$
\item If $\rho>1$:
\begin{itemize}
\item For all $x\in \Xset$, $q(x)<1$.
\item For all $x\in \Xset$, conditionally on $|T|=\infty$, $\frac{|Z_i|}{\rho^i}$ converges $P_x$-a.s. toward a random non-zero variable $W$ when $i$ tends to $\infty$.
\end{itemize}
\end{enumerate}
\end{lem}

For all non-negative function $s$, let $\phi_x(s)$ be defined by:
$$\phi_x(s)=\E_x\left(\prod_{y\in Z_1} s(y)\right).$$
\begin{rem}
With the notation of \cite{Harris}, this is actually $\Phi(-\ln s)$, but this version is more suited for our use.
\end{rem}
By \cite[Theorem 15.1]{Harris}:
\begin{lem}
\label{lemmeprobasurvie}

If $\rho>1$, 
\begin{itemize}
\item $q$ is the only uniformly positive and uniformly less than $1$ function satisfying
$q(x)=\phi_x(q)$ for all $x\in \Xset$.
\item If $q^0$ is a positive and uniformly less than $1$ function, and if we define the sequence of functions $(q^i)_{i\geq 1}$ by $q^{i+1}(x)=\phi_x(q^i)$, then $q^i$ converges everywhere toward $q$.
\end{itemize}
\end{lem}
\begin{cor}
\label{coromonotonie}
\leavevmode
\begin{itemize}
\item If $q^0$ is a positive and uniformly less than $1$ function such that for all $x$, $q^0(x)\geq \phi_x(q^0)$, then for all $x\in\Xset$, $q(x)\leq  q^0(x)$.
\item If $q^0$ is a positive and uniformly less than $1$ function such that for all $x$, $q^0(x)\leq \phi_x(q^0)$, then for all $x\in\Xset$, $q(x)\geq  q^0(x)$.
\end{itemize}
\end{cor}
\begin{proof}
We define $(q^i)_{i\geq 0}$ as in Lemma \ref{lemmeprobasurvie}. Then by Lemma \ref{lemmeprobasurvie}, for all $x$, $q^i(x)$ converges toward $q(x)$. For the first part of Corollary \ref{coromonotonie}, $q^1(x)\leq q^0(x)$. As $\phi_x(s)$ is an increasing function of the positive function $s$, the sequence $(q^i)_{i\geq 0}$ is a non-increasing sequence  and is therefore larger than its limit. Similarly, in the second part, the sequence $q^i$ is non-decreasing, and therefore smaller than its limit.
\end{proof}

Let $T'$ be a two stages-branching process, with a different law for the root, and same law as $T$ for the remaining of the tree. Let $Z_i'$ be the random set of labels of the $i$th generation of $T'$. Recall that $q$ and $\rho$ denote the extinction probability function and the spectral radius of $T$. 
\Needspace*{3\baselineskip}
\begin{lem}
\label{lemmeprobasurviedeuxetages}
\leavevmode
\begin{itemize}
\item $\P(|T'|<\infty)=\E\left(\displaystyle\prod_{y \in Z'_1}q(y)\right).$
\item If $\rho>1$, conditionally on $|T'|=\infty$, $\frac{[Z'_i|}{\rho^i}$ converges toward a random non-zero variable $W$.
\end{itemize}
\end{lem}

For such a two-stages branching process $T'$, $\rho$ will be called the spectral radius associated to $T'$.
\begin{proof}
Conditionally on the first generation $Z_1$, the subtrees starting at elements of $Z_1$ are independent, so the probability of extinction is $\prod_{y \in Z'_1}q(y)$. The first part of Lemma \ref{lemmeprobasurviedeuxetages} is obtained by taking the expectation. The second part is obtained by using Lemma \ref{lemmevitessecroissance} on the surviving subtrees.
\end{proof}

For all $t$, let $\rho_t$ be the spectral radius associated to $T^{k+,\cdot}_t$ or $T^k_t$. 
\begin{lem}
\label{lemmecontinuite1}
The spectral radius $\rho_t$ is a upper semi-continuous function of~$t$.
\end{lem}
Lemma \ref{lemmecontinuite1} will be proven in part \ref{preuvecontinuiteprocessus}.
\subsection{Approximating $G^k_{n,t}$ by an idealised graph}

Let $t$ be such that $a^k_t>0$, \ie{} $\rho_t>1$. Let $\frac 1 2>\epsilon_1>0$ and $t'=t-\epsilon_1$. Let $\epsilon_2>0$. Let $K_n=\frac{(1-\epsilon_2)\ln n}{\ln (\rho_{t'})}$. Using Lemma \ref{lemmecontinuite1}, we can choose $\epsilon_1$ small enough so $\rho_{t'}>1$.

The graph $\Gmod$ will be constructed dynamically while exploring $G^\infty_{n,t'}$. $\Gmod$ will denote the growing  graph process and $\Gmod_\text{end}$ the graph $\Gmod$ at the end of its construction. At any point, $\Gmod$ will satisfy the following properties:
\begin{itemize}
\item $\Gmod$ is either empty, one planar tree with labelled edges or two planar trees with labelled edges, such that no two edges have the same label.
\item In each planar tree, the children of a given vertex are ordered, for the planar tree order, in increasing order of the labels of the outgoing edges.
\end{itemize}

The graph $\Gmod_\text{end}$ will have the same law as two independent copies of $T^\infty_{t'}$. Informally, $\Gmod_\text{end}$ will be built in such a way that the components of the roots in $\Phi(\Gmod_\text{end})$ are close to the components of $v_1$ and $v_2$ in ${G^k_{n,t'}=\Phi(G^\infty_{n,t'})}$. In order to achieve this, we define an exploration process, that will only look at the parts of the graph $G^\infty_{n,t'}$ that are useful to construct the component of $v_1$ and $v_2$ in $G^k_{n,t'}$. At the beginning of its exploration, $G^\infty_{n,t'}$ (seen from $v_1$ and $v_2$) looks like two independent copies of $T^\infty_{t'}$. But, as the exploration continues, these two graphs differ: $G^\infty_{n,t'}$ can have cycles or multiple edges, and the numbers of neighbors of the vertices decreases (in law), as a growing fraction of $G^\infty_{n,t'}$ is known. The following construction explains how to couple $G^\infty_{n,t'}$ with its idealized branching version $\Gmod_\text{end}$ in such a way that the differences between $\Phi(\Gmod_\text{end})$ and $G^k_{n,t'}$ are small.

Let $V^G$ denote the set of vertices of $G^\infty_{n,t}$. In order to construct $\Gmod$, the process will use $G^\infty_{n,t'}$ and some extra randomness. Initially, $\Gmod$ is empty, and will be created dynamically along with the exploration of $G^\infty_{n,t'}$. The process will add vertices to $\Gmod$, from $V^G$ and extra vertices. The extra vertices will be called \emph{dummy vertices}. 

As $\Gmod_\text{end}$ is two planar trees, we can define the following notions:
\begin{defi}
At any point, $\Gmod$ is a rooted planar tree or two rooted planar trees. For any  vertex $w$ of $\Gmod$ and integer $i$, \emph{the $i$-children} of $w$ are defined as the vertices at distance $i$ of $w$ in the subtree starting at $w$.

The set of vertices of $\Gmod_\text{end}$ is ordered with the breadth-first order, denoted by $\breadth$. For any two vertices $w$ and $w'$ of $\Gmod_\text{end}$, $w\breadth w'$ if either:
\begin{itemize}
\item $w$ is in the first component of $\Gmod_\text{end}$ and $w'$ is in the second component.
\item $w$ and $w'$ are in same component, and the distance between $w$ and the root is strictly smaller than the distance between $w'$ and the root.
\item $w$ and $w'$ are in the same component, in the same generation, and the father of $w$ is strictly smaller than the father of $w'$ for the breadth-first order.
\item $w$ and $w'$ are two children of the same vertex $w''$, and the label of the edge between $w$ and $w''$ is smaller than the label of the edge between $w'$ and $w''$.
\end{itemize}
\end{defi}

The construction of $\Gmod_\text{end}$ will use the following increasing sets:

\begin{itemize}
\item $A^G\subset V^G$ will denote the set of vertices of $G^\infty_{n,t'}$ that have already been discovered.
\item $B^G\subset A^G$ will denote the set of the vertices whose neighbors are known. These vertices will be called \emph{used}.
\item $A^C\subset A^G$ will denote the set of \emph{corrupted} vertices. They indicate where $\Gmod$ and $G^\infty_{n,t'}$ are different.
\item $B^C\subset A^C$ will denote the set of corrupted vertices whose children in $\Gmod_\text{end}$ have already been constructed.
\item $A^{Dum}$ will denote the set of dummy vertices. At any point, $A^{Dum}\cap V^G=\emptyset$. 
\item $B^{Dum}\subset A^{Dum}$ will denote the subset of dummy vertices whose children in $\Gmod_\text{end}$ have already been constructed.

\item $A=A^G\cup A^{Dum}$ is the set of the vertices of $\Gmod$. The sets $A^G$ and $A^{Dum}$ are disjoint.
\item $B=B^G\cup B^C\cup B^{Dum}$ is the set of the vertices of $\Gmod$ whose children in $\Gmod$ have already been constructed. The sets $B^G\cup B^C$ and $B^{Dum}$ are disjoint, but $B^G$ and $B^C$ might not be disjoint.
\end{itemize}

 During the process, all these sets will only increase. 
 
$\Gmodm$ is the isomorphism class of the graph $\Gmod$, with respect to the isomorphism of graphs with labelled edges and unlabelled vertices. This allows for example to consider $\Gmodm$ without knowing if a given vertex is a dummy vertex. As the labels of the edges of $\Gmod$ are all different, the only automorphism of $\Gmod$, seen as a graph with labelled edges and unlabelled vertices, is the identity, and therefore a unique way to map $\Gmodm$ back to $\Gmod$, allowing us to consider the vertex of $\Gmod$ associated to a given vertex of $\Gmodm$.

Let $\F$ denote the increasing $\sigma$-algebra, generated by $A^G$, $B^G$, $A^C$, $B^C$, $A^{Dum}$, $B^{Dum}$, the set of labelled edges with at least one endpoint in $B$ in the graph $G^\infty_{n,t'}$ and the set of labelled edges with at least one endpoint in $B$ in $\Gmod_\text{end}$. $\F$ represents the current knowledge obtained with the exploration of $G^\infty_{n,t'}$ and the construction of $\Gmod$.

The following tools will be used in the construction of $\Gmod$.

\subsubsection{Split a vertex/edge:}
There is no multiple edge in $T^\infty_{t'}$, whereas there can be multiple edges in $G^\infty_{n,t'}$. To solve this issue, any multiple edge will be substituted upon discovery by the appropriate number of simple edges. If a multiple edge $e$ of multiplicity $l$ is discovered between $w$ and $w'$ while looking at the neighbors of $w$, add $w'$ to $\Gmod$ and only one edge between $w$ and $w'$, labelled by the smallest label of $e$. For any label $y$ among the $l-1$ other labels, considered in increasing order, add a dummy vertex $w_y$ to $\Gmod$ and $A^{Dum}$, and add an edge between $w$ and $w_y$ labelled by $y$ in $\Gmod$. An example can be found in figure \ref{figuresplit}. This operations allows to avoid adding multiple edges to $\Gmod$ without altering the degree of $w$ (it does alter the degree of $w'$).

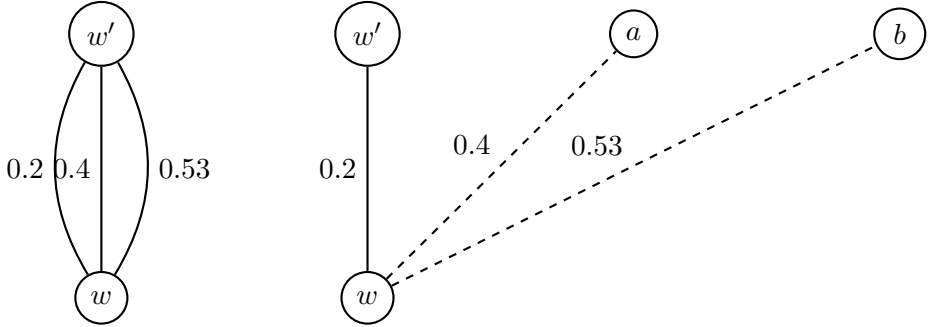
\begin{figure}[h]
 \begin{tikzpicture}[-,auto,node distance=3.5cm,  thick,main node/.style={circle,draw}]

   \node[main node] (1) {$w$};
   \node[main node] (2) [above of=1] {$w'$};
   \node[main node] (3) [right of=1] {$w$};
  \node[main node] (4) [above of=3] {$w'$};
  \node[main node] (5) [right of=4]{$a$};
   \node[main node] (6) [right of=5]{$b$};

   \path
   (1) edge [bend left,left]node  {0.2} (2)
  (1) edge node {0.4} (2)
  (1) edge [bend right,right] node  {0.53} (2)
  (3) edge node  {0.2} (4)
  (3) edge[dashed] node  {0.4} (5)
  (3) edge [dashed]node  {0.53} (6);

 \end{tikzpicture}
 \caption{The triple edge between $w$ and $w'$ in $G^\infty_{n,t'}$ is split in $\Gmod$ by adding two dummy vertices $a$ and $b$.}
 \label{figuresplit}
 \end{figure}

\subsubsection{Probing a vertex $w$:} 

This operation will be done for a vertex $w\in A^G\setminus B$. The goal is to approximate  the set of labelled  edges adjacent to $w$ in $G^\infty_{n,t'}$ by a Poisson point process of intensity $1$ on $[0,t']$.
\begin{itemize}
\item Let $\poissondummy_w$ be an independent Poisson point process of intensity $\frac{|A^G|}n$ on $[0,t']$. For each point $s$ of $\mu_w$, in increasing order, add a dummy vertex $w'$ and an edge labelled by $s$ between $w$ and $w'$ to $\Gmod$. Add $w'$ to $A^{Dum}$. If there is at least one such edge, add $w$ to $A^C$ and $B^C$.
\item  For every vertex $w'\in V^G\setminus A^G$, look at the edge between $w'$ and $w$:
\begin{itemize} 
\item If there is a single edge between $w$ and $w'$, add this edge and $w'$ to $\Gmod$. Add $w'$ to $A^G$.
\item If there is a multiple edge between $w$ and $w'$, split it. Add $w'$ to $A^G$ and $A^C$. Add $w$ to  $A^C$ and $B^C$.
\item If there is no edge between $w$ and $w'$, do not do anything.
\end{itemize}
\item For every vertex $w'$ in $A^G\setminus B^G$, do not add the edges between $w$ and $w'$ to $\Gmod$. If there is at least one such edge, add $w$ to $A^C$ and $B^C$ and  add $w'$ to $A^C$. 
\item Add $w$ to $B^G$.
\end{itemize}

\subsubsection{Fake-probing a vertex $w$:} 

\begin{enumerate}
\item
If $w$ is a vertex of $A^{Dum}\setminus B^{Dum}$, fake-probing $w$ is the following operations:
\begin{itemize}
\item  Add $w$ to $B^{Dum}$.
\item Let $Y_w$ be an independent Poisson point process of intensity $1$ on $[0,t']$. For each point $y$ of $Y_w$, add a dummy vertex $w_y$ to $\Gmod$ and $A^{Dum}$, and add an edge labelled by $y$ between $w$ and $w_y$ to $\Gmod$.
\end{itemize}
\item
If $w$ is a vertex of $A^G\setminus B$, fake-probing $w$ is the following operations:
\begin{itemize}
\item  Add $w$ to $A^C$ and $B^{C}$.
\item Let $Y_w$ be an independent Poisson point process of intensity $1$ on $[0,t']$. For each point $y$ of $Y_w$, add a dummy vertex $w_y$ to $\Gmod$ and $A^{Dum}$, and add an edge labelled by $y$ between $w$ and $w_y$ to $\Gmod$.
\end{itemize}

\end{enumerate}

For any vertex $w$, and any set,  $\sigma$-algebra or graph $X$, let $X_w$ denote the value of $X$  before the probing or the fake-probing of $w$. Similarly, let $X_{w+}$ denote the value of $X$ after the probing or fake-probing of $w$.

The following lemma summarizes the properties that will hold through the construction of $\Gmod_\text{end}$:
\begin{lem}
\label{probeproperties}
\leavevmode
\begin{enumerate}
\item For every vertex $w'\in \Gmod_w$, $w'\neq w$, the set of edges adjacent to $w'$ in $\Gmod$ is not modified during the probing or fake-probing of $w$.
\end{enumerate}
\leavevmode
If the following properties hold before the probing or the fake-probing of a vertex $w$ in $A_w\setminus B_w$, then the following properties will also hold after the probing or fake-probing of $w$.

\begin{enumerate}[\hspace{5mm}1.]
\setcounter{enumi}{1}
\item The vertex set of $\Gmod$ is $A$.
\item For every vertex $w'\in B^G\setminus B^C$, the set of edges adjacent to $w'$ is the same in $\Gmod$ and $G^\infty_{n,t'}$.
\item For every vertex $w'\in A\setminus B$, $w'$ has no child in $\Gmod$. 
\item If $w'\in A^G\setminus A^C$, the set of edges between $w'$ and elements of $B^G$ is the same in $\Gmod$ and in $G^\infty_{n,t'}$.
\item If $w'\in V^G\setminus A$ then there is no edge between $w'$ and any vertex of $B^G$ in $G^\infty_{n,t'}$.

\end{enumerate}
We assume that the choice of $w$ and of whether to do a probing or a fake-probing are $\F$-measurable. If the following property holds before the probing or fake-probing of $w$,  then it holds after:
\begin{enumerate}[\hspace{5mm}1.]
\setcounter{enumi}{6}
\item Let $E^B\subset E$ be the set of edges with at least one endpoint in $B^G$. For any $e\in E^B$, the set of labels of $e$ in $G^\infty_{n,t'}$ is $\F$-measurable. Conditionally on $\F$, the set of labels of edges in $E\setminus E^B$ is an i.i.d. family of Poisson point processes of intensity $\frac 1 n$ on $[0,t']$.
\end{enumerate}

If the properties 2-7 hold before the probing of $w$ and if one of these two properties holds before the probing of a vertex, they still holds after:
\begin{enumerate}[\hspace{5mm} {8}a.]
\item $\Gmod$ is a tree rooted at $\tilde v_1$.
\item $\Gmod$ is two trees rooted at $\tilde v_1$ and $\tilde v_2$.
\end{enumerate}
If the properties 2-7 hold before the probing of $w$, then
\begin{enumerate}[\hspace{5mm}1.]
\setcounter{enumi}{8}
\item conditionally on $\F_w$, the law of set of all labels of edges between $w$ and the children of $w$ in $\Gmod$, including dummy vertices, is a Poisson point process of intensity $1$ on $[0,t']$.
\end{enumerate}

\end{lem}

\begin{proof}

The  properties 1-6 and 8a/8b are direct consequences of the description of the probing and fake-probing of a vertex.

If $w\in A^G\setminus B^G$ is probed, the only vertex added to $B^G$ during the probing of $w$ is~$w$, and the edges of $G^\infty_{n,t'}$ looked at are the edges adjacent to~$w$. If $w$ is fake-probed, no vertex is added to $B^G$ and no edge is looked at during the fake-probing of $w$. Therefore property 7 holds after the probing if it holds before the probing.

If $w\in A$ is probed, by property 7, conditionally on $\F_w$, the set of labels of edges between $w$  and elements of $V^G\setminus |A^G|$ is a Poisson point process of intensity $\frac{n-|A^G|}n$. Adding an independent Poisson point process of intensity $\frac{|A^G|}n$ gives the wanted distribution for the property 9. If $w$ is fake-probed, property 9 is a direct consequence of the definition of the fake-probing.
\end{proof}

Each time several vertices are probed (or fake-probed) at the same time by the algorithm, they are probed in breadth-first order.

\subsubsection{The initialisation of a component:}
This tool is used at the beginning of the construction of each of the two components of $\Gmod_\text{end}$. Let $j\in\{1,2\}$. The aim of this part is to construct the ball of radius $b_n$ centered at $\tilde v_j$ in $\Gmod$.

The description assume that, before the initialisation, the properties 2-7 of Lemma \ref{probeproperties} hold and that either $\Gmod$ is empty (if $j=1$) or 8a holds (if $j=2$).

For every $i,j$, $S_{j,i}$  is the sphere of radius $i$ centered at $\tilde v_j$ in $\Gmod_\text{end}$. Let $\B_{1,i}=\cup_{i'\leq i}S_{1,i}$ and $\B_{2,i}=\cup_{i'\leq i}S_{2,i}\cup \Gmod_{\text{end},1}$. For any $w\in \Gmod$, $\gen(w)$ is the ordered pair $(j,i)$ such that $w\in S_{j,i}$.

\begin{defi}[Initialisation failed]
If one of the following events happen, the initialisation is said to have failed:
\begin{itemize}
\item A previous initialisation has failed.
\item Before the initialisation, $v_j$ was already in $A^G$.
\item At any point of the construction of a $S_{j,i}$, a vertex is added to $A^C$.
\end{itemize}
\end{defi}

If $v_j\in A^G$ or if a previous initialisation has failed, then add to $\Gmod$ and $A^{Dum}$ a dummy vertex $\tilde v_j$, as the root of the second component. Otherwise add $v_j$ to $A^G$ and to $\Gmod$ and let $\tilde v_j=v_j$.

 Let $0\leq i <b_n$. We assume that $S_{j,i}$ has already been built. If the initialisation has not yet failed, probe every vertex	 $w$ of $S_{j,i}$. If the initialisation has failed, fake-probe every vertex $w$ of $S_{j,i}$. In both cases, $S_{j,i+1}$ is the set of children of elements of $S_{j,i}$ in $\Gmod$, \ie{} the sphere of radius $i+1$ in $\Gmod$ centered at $\tilde v_j$.

Starting with $S_{j,0}=\{\tilde v_j\}$, the sets $S_{j,i}$ are built until $S_{j,b_n}$ has been built. For any set, $\sigma$-algebra of graph $X$, let $X_{j,i}$ denote the value of $X$ after the construction of $S_{j,i}$.

\begin{lem}
\label{initialisationproperties}
\leavevmode
\begin{itemize}
\item After the construction of $S_{j,i}$, if the initialisation has not yet failed, $S_{j,i}\subset A^G\setminus (B^G\cup A^C)$.
\item After the construction of $S_{j,i}$, for every $i'<i$, $S_{j,i'}\subset B$.
\end{itemize}
\end{lem}
These properties are direct consequences of the algorithm: if at any point a vertex is added to $A^C$, the initialisation is said to have failed. The first property allows us to probe the elements of $S_{j,i}$ when the initialisation has not yet failed.

\subsubsection{Subsequent construction of the component of $v_j$ in $\Gmod$:}
\label{sssectionSuiteConstruction}
If the initialisation has not failed, the balls of radius $b_n$ centered at $v_j$ are the same in $\Gmod$ and in $G^\infty_{n,t'}$, by property 3 of Lemma~\ref{probeproperties}. Therefore, if $GE_1=1$, it is possible to determine the neighbors of $v_j$ in $G^k_{n,t'}$ (by Lemma \ref{range}). Let $D_{j,0}=\{v_j\}$, $E_{j,0}=\{v_j\}$ and $D_{j,1}$ be the set of neighbors of $v_j$ in $G^k_{n,t'}$. If the initialisation has failed or if it is not possible to determine the neighbors of $v_j$ in $G^k_{n,t'}$, let $D_{j,1}=D_{j,0}=E_{j,0}=\emptyset$.

After the initialisation, the component of $\tilde v_j$ in $\Gmod$ is a planar tree rooted at $\tilde v_j$ of depth at most $b_n$, the vertices of the $b_n-1$ first generations belonging to $B$ and the vertices of the $b_n$th generation belonging to $A\setminus B$.

\subsubsection{One iteration of the main part of the algorithm:}

We now describe the iteration $i$, constructing the sets $E_{j,i}$, $D_{j,i+1}$. For all $j,i$, $E_{j,i}\subset D_{j,i}\subset S_{j,i}$. For any element $w$ of $D_{j,i}$, if we are able to determine the children of $w$ in the forbidden degree version of $\Gmod$, $w$ is added to $E_{j,i}$ and its children are added to $D_{j,i+1}$. More precisely, assuming $D_{j,i}$ and $\B_{j,b_n+i-1}$ are known, the iteration $i$ is the following operations:
\begin{itemize}
\item Initially, $E_{j,i}$ and $D_{j,i+1 }$ are empty.
\item For each $w\in S_{j,i+b_n-1}$, in the breadth-first order, if $w$ is a $b_n-1$-children of an element of $D_{j,i}$, probe $w$. Otherwise, fake-probe $w$.
 \item Consider every $w$ in $D_{j,i}$ such that no $(b_n-1)$-child or $b_n$-child of $w$ is in $A^C$. If the knowledge of $\Gmod$ is sufficient to determine the set~$N$ of neighbors of $w$ in $G^k_{n,t'}$ add $w$ to $E_{j,i}$, and add every element of $N\setminus{p(w)}$ to $D_{j,i+1}$ where $p(w)$ is the parent of $w$ in $\Gmod$.
\end{itemize} 
\begin{rem}
If $GE_2=1$, the knowledge of the ball of radius $b_n$ centered at $w$ in $G^\infty_{n,t'}$ is sufficient to know the set $N$. Lemma \ref{basicpropertiesGmod} will prove that the balls of radius $b_n$ centered at $w$ in $\Gmod_\text{end}$ and $G^\infty_{n,t'}$ are identical and therefore that the knowledge of $\Gmod$ is sufficient to know the set $N$ if $GE_2=1$. Lemma \ref{basicpropertiesGmod} will also prove that for every $b_n1$-children of elements of $D_{j,i}$ is in $A^G_{j,i+b_n}$, and therefore can be probed.
\end{rem}

An iteration is the execution of the algorithm above, for a given value of $i$. For any set, $\sigma$-algebra or graph $X$, let $X_{j,i+b_n}$ denote the value of $X$  after the $i$th iteration in the construction of the component of $\tilde v_j$, \ie{} the construction of $E_{j,i}$ and $D_{j,i+1}$.

\subsubsection{Finalisation of a component:}
After $K_n$ iterations of this algorithm are done, and $\B_{j,K_n+b_n}$ is constructed, graft to every vertex of $S_{j,b_n+K_n}$ an independent copy of $T^\infty_{t'}$. Add every added vertex to $A^{Dum}$ and $B^{Dum}$. Add every vertex of $S_{j,b_n+K_n}$ to $A^C$ and $B^C$. 

The value of the sets and $\sigma$-algebra $X$ after the finalisation of the $j$th component will be denoted by $X_{j,\infty}$. If the properties 2-7 and either 8a or 8b hold before the finalisation of a component, they hold after the finalisation of a component, as no edge of $G^k_{n,t'}$ is looked at, and the copies are independent.

\subsubsection{Construction of $\Gmod_\text{end}$:}
The construction of $\Gmod_\text{end}$ is done by starting with $A^G=B^G=A^C=B^C=A^{Dum}=B^{Dum}=\emptyset$ and doing the following operations:
\begin{enumerate}
\item First initialisation, for the component of $\tilde v_1$.
\item Do the iteration $i$ of the algorithm for $i$ from $1$ to $K_n$, to construct  $\B_{1,b_n+i}$ and the sets $E_{1,i}$ and $D_{1,i+1}$ for every $i\leq K_n$.
\item Finalise the construction of the first component.
\item Second initialisation, for the component of $\tilde v_2$.
\item Do the iteration $i$ of the algorithm for $i$ from $1$ to $K_n$, to construct $\B_{2,i+b_n}$ and the sets $E_{2,i}$ and $D_{2,i+1}$ for every $i\leq K_n$.
\item Finalise the construction of the second component.
\end{enumerate}
For any graph, $\sigma$-algebra of set $X$,  $X_\text{end}$ denotes the value of $X$ after the complete algorithm.

\subsubsection{First properties of $\Gmod$}

\begin{lem}

\label{basicpropertiesGmod}

For any $i<b_n$, $j\in\{1,2\}$, after the construction of $S_{j,i}$ in the initialisation:
\begin{enumerate}
\item the properties 2-7 of Lemma \ref{probeproperties} hold;
\item if $j=1$, 8a holds; if $j=2$, 8b holds.
\end{enumerate}

For every $i\leq K_n $ and $j\in\{1,2\}$, after the iteration $i$ constructing $E_{j,i}$ and  $D_{j,i+1}$:
\begin{enumerate}
\setcounter{enumi}{2}
\item for every vertex $w$ of $E_{j,i}$, the ball of radius $b_n$ centered at $w$ is the same in $\Gmod$ and $G^\infty_{n,t'}$;
\item every vertex of $D_{j,i+1}$ is a child of a vertex of $E_{j,i}$ in $\Gmod$;
\item the set of vertices of $\Gmod_{j,i+b_n}$ is $\B_{j,i+b_n}$;
\item $A_{j,i+b_n}=\B_{j,i+b_n}$;
\item $B_{j,i+b_n}=\B_{j,i+b_n-1}$;
\item the properties 2-7 of Lemma \ref{probeproperties} hold;
\item if $j=1$, 8a holds; if $j=2$, 8b holds.
\end{enumerate}
After the construction of $\Gmod_\text{end}$:
\begin{enumerate}
\setcounter{enumi}{9}
\item $\Gmodm_\text{end}$ has same law as two independent copies of $T^\infty_{t'}$. Moreover the second component of $\Gmodm_\text{end}$ is independent of $\F_{1,\infty}$.
\end{enumerate}

\end{lem}

\begin{proof}
  The properties 1, 2, 8 and 9 are consequences of Lemma \ref{probeproperties} and the fact that the properties 2-7 of Lemma \ref{probeproperties} are preserved by the initialisation or the finalisation of a component. The properties 4, 5, 6 and 7 are consequences of the description of the algorithm, as every vertex of $S_{j,i+b_n-2}$ is either probed or fake-probed for the construction of $D_{j,i}$.

The property 3 is proven by induction. As noted in the subsection \ref{sssectionSuiteConstruction}, property 3 holds for $i=0$ and $j\in\{1,2\}$. Let us assume that the property~3 holds for $i\geq 0$ and  $j\in\{1,2\}$. Let $w$ be a vertex of $E_{j,i+1}\subset D_{j,i+1}$. By property~4, $w$ is a children of a vertex $w'$ of $E_{j,i}$. As the ball of radius $b_n$ centered at $w'$ is the same in $\Gmod$ and $G^\infty_{n,t'}$, the only possible difference between the balls of radius $b_n$ centered at $w$ is the neighbors of the {$b_n-1$-children} of $w$. On one hand, if none of the $b_n-1$-children or the $b_n$-children of $w$ are added to $A^C$, then every ball of radius $1$ centered at such children is the same in $\Gmod$ and in $G^\infty_{n,t'}$ and therefore the ball of radius $b_n$ centered in $w$ is the same in $\Gmod$ and $G^\infty_{n,t'}$. One the other hand, if one such vertex is added to $A^C$, $w$ is not added to $E_{j,i}$.

To prove property 10, it is sufficient to show that for any $i$ and $j$, conditionally on $\B_{j,i}$, the law of the set of labels of the edges outgoing of the vertices of $S_{j,i}$ is an i.i.d. family of Poisson point processes of intensity $1$ on $[0,t']$. If $i<b_n$, conditionally on $\F_{j,i}$, if the initialisation have failed at an earlier stage, every fake-probing gives the correct law. If the initialisation has not yet failed, the algorithm is going to probe each vertex of $S_{j,i}$. By Lemma~	\ref{probeproperties}, the conditional law of the set of exiting edges is also a Poisson point process of intensity $1$ on $[0,t']$ (some of them might point to dummy vertices). As $\B_{j,i}$ is measurable with respect to $\F_{j,i}$, the conditional law is the same conditioned on $\B_{j,i}$.

For all $i\in\{b_n,\dots,K_n+b_n-1\}$,  $\B_{j,i}$ is measurable with respect to $\F_{j,i}$. Each vertex of the ball of radius $i$ is either probed or fake-probed. Therefore, by property 7 of Lemma \ref{probeproperties}, conditionally on $\F_{j,i}$, or on $\B_{j,i}$, the set of labels of edges between each vertex of the ball of radius $i$ and their children is an i.i.d. family of Poisson point processes of intensity $1$ on $[0,t']$.

If $i\geq K_n+b_n$, the law of the subsequent generations comes from the grafting of independent copies of $T^\infty_{t'}$, and gives the correct conditional law by the branching properties of $T^\infty_{t'}$.
\end{proof}
$\Gmod_\text{end}$ has the law of two independent copies of $T^\infty_{t'}$. Therefore ${\Gmod_\text{end}\in \Omega^+}$ a.s. Let $\Gmodk$ denote the components of $\tilde v_1$ and $\tilde v_2$ in $\Phi(\Gmod_\text{end})$, the forbidden degree version of $\Gmod$. For any $i\leq K_n$, $j\in\{1,2\}$, let $\D^{-}_{j,i}$ be the union of the sets $D_{j',i'}$ where $(j',i')$ is strictly smaller than $(j,i)$ for the lexicographic order. The set $\D^-_{j,i}$ is the union of the sets $D_{j',i'}$ constructed before $D_{j,i}$.  Let $\D^{+}_{j,i}=\D^{-}_{j,i}\cup D_{j,i}$. The following lemma gives a few useful observations about the sets $D_{j,i}$ and $E_{j,i}$ and the graph $\Gmod$ produced by the algorithm and allows to link them with the component of $v_1$ and $v_2$ in $G^k_{n,t'}$:
\begin{lem}
\label{observationDE}

\leavevmode
\begin{enumerate}
\item Every pair of vertices of $\D^{+}_{2,K_n}$ that are neighbors in $\Gmod_\text{end}$ are also neighbors in $G^k_{n,t'}$ and in $\Gmodk$, and the edge between them has the same label in $G^k_{n,t'}$ and $\Gmodk$.
\item Every vertex of $D_{j,i}$  belongs to the component of $v_j$  in $G^k_{n,t'}$ and in $\Gmodk$.
\item For $i\leq K_n$ and $j\in\{1,2\}$, every neighbor of vertices of $E_{j,i}$ in $G^k_{n,t'}$ (resp. in $\Gmodk$) belong to either $D_{j,i+1}$ or $E_{j,i-1}$. 
\end{enumerate}
\end{lem}
The first and third observations come from the definition of $D_{j,i}$. The second observation is consequence of the first observation of Lemma \ref{observationDE} and the property 4 of Lemma \ref{basicpropertiesGmod}.

The first two observations guarantee that the subgraph of $\Gmod_\text{end}$ restricted to vertices of $\D^+_{2,K_n}$ can be seen as a subgraph of $G^k_{n,t'}$. The last property guarantees that the degree of vertices of $E_{j,i}$  is the same in this subgraph and in $G^k_{n,t'}$.

\subsection{Finding  a path in $G^k_{n,t}$}

\begin{defi} 
For every $i<K_n$ and $j\in\{1,2\}$, $\tilde E_{j,i}$ is the set of vertices~$w$ of $E_{j,i}$ such that no edge is added to $w$ between $t'$ and $t$ in $G^\infty_{n,\cdot}$. Let $\tilde E_{j,K_n}=E_{j,K_n}$. 

The sets $\hat E_{j,i}$ are defined by induction:
\begin{itemize}
\item For $j\in \{1,2\}$, $\hat E_{j,0}=\tilde E_{j,0}$.
\item For $j\in \{1,2\}$, $i\in\{1,\dots, K_n\}$, $\hat E_{j,0}=\{u\in \tilde E_{j,i}:p(u)\in\hat E_{j,i-1}\}$, where $p(u)$ denotes the parent of $u$ in $\Gmodk$.
\end{itemize}

\end{defi}
The set $\hat E_{j,i}$ is the subset of vertices $u$ of $\tilde E_{j,i}$ such that no edge is added between $t'$ and $t$ to any vertex in the path from $\tilde v_j$ to $u$ (including $u$ if $i<K_n$, excluding $u$ if $i=K_n$). By Lemma \ref{observationDE}, any neighbor in $\Gmodk$ of a vertex $v$ of $E_{j,i}$ is either in $D_{j,i+1}$ or in $D_{j,i-1}$. Therefore the following lemma holds:
\begin{lem}
Starting from $\Gmodk$, remove every elements of $\cup D_{j,i}\setminus \tilde E_{j,i}$. Let $\hat \Gmodk$ be the union of the components of $v_1$ and $v_2$ in the resulting graph. Then the set of vertices in the balls of radius $K_n$ centered at $v_1$ and $v_2$ in $\hat \Gmodk$ is $\cup \hat E_{j,i}$.
\end{lem}

\begin{lem}
\label{lemmett}
\leavevmode
\begin{enumerate}

\item For $j\in\{1,2\}$, every vertex of the ball of radius $K_n-1$ in $\hat\Gmodk$ is in the component of $v_j$  in the graph $G^k_{n,u}$ for any $u\in[t',t]$.

\item For $j\in\{1,2\}$, if $w$ of $ \hat E_{j,K_n}$ is such that
\begin{itemize}
\item the degree of $w$ in $\Gmodk$ is smaller than $k-2$, and
\item at most one edge is added to $w$ in $G^\infty_{n,}$ between $t'$ and $t$
\end{itemize} then $w$ is in the component of $v_j$ in $G^k_{n,u}$ for any $u\in[t',t]$.
\end{enumerate}
\end{lem}

By Lemma \ref{observationDE}, all the elements of $\cup_i \tilde E_{j,i}$ and the edges between them in $\Gmodk$ are present in $G^k_{n,t'}$. Therefore $\hat\Gmodk$ is a subgraph of $G^k_{n,t'}$. As by definition of $\tilde E_{j,i}$ no edge is added to any vertex in the balls of radius $K_n-1$ between $t'$ and $t$, no vertex of the balls of radius $K_n$ reaches the forbidden degree between $t'$ and $t$ and therefore the ball of radius $K_n-1$ is a subgraph of $G^k_{n,u}$ for any $u\in[t',t]$. 

If $w$ is a non-saturated vertex of $\hat\Gmodk$ such that at most one edge is added to $w$ in $G^\infty_{n,\cdot}$ between $t'$ and $t$, the edge between $w$ and its parent in $\Gmodk$ is present in $G^k_{n,u}$ for any $u\in[t',t]$.

\begin{thm}
\label{mainthm}
For any $\eta>0$, for $\epsilon_1$ and $\epsilon_2$ small enough and $n$ large enough, with probability at least $(a^k_t)^2-\eta$, there exists a vertex $w_1$ in $\hat E_{1,K_n}$ and a  vertex $w_2$ in $\hat E_{2,K_n}$ such that:
\begin{itemize}
\item The vertices $w_1$ and $w_2$ are non-saturated in $G^k_{n,t'}$.
\item There is an edge added between time $t'$ and time $t$ between $w_1$ and $w_2$ in $G^\infty_{n,.}$.
\item No other edge is added between time $t'$ and time $t$ to either $w_1$ or $w_2$ in $G^\infty_{n,.}$.
\end{itemize}

\end{thm}
Theorem \ref{mainthm} implies Lemma \ref{lemmelien}, as by Lemma \ref{lemmett}, $w_1$ and $w_2$ are respectively in the component of $v_1$ and $v_2$ in $G^k_{n,t}$ and the edge between $v_1$ and $v_2$ is present in $G^{k}_{n,t}$.

To prove this result, the following facts will be proven:
\begin{enumerate}
\item For $\epsilon_2$ small enough, with probability larger than $(a^k_t)^2-\eta$, for all $j$, there are at least $n^\frac 2 3$ non saturated vertices in $|\hat E_{j,K_n}|$.
\item The probability that the previous condition happens without the existence of $w_1$ and $w_2$ satisfying the conditions of Theorem \ref{mainthm} converges to~$0$.
\end{enumerate}

It will be done by finding a subgraph of $\hat\Gmodk$ that is a branching process with offspring law close to the law of $T^k_{t'}$. Therefore, the probability of survival of each component is close to $a^k_{t'}$, and conditionally on survival, the components will be large enough.

\begin{lem}
\label{sizeA}
Let $z_n=n^{1-\frac{\epsilon_2}2}$. With high probability, at any point of the algorithm $|A^G|\leq z_n$.
\end{lem}

Every element of $A^G$ belongs to a ball of radius $b_n$ centered at an element of $\cup_{j,i} D_{j,i}$ in $G^\infty_{n,t'}$. Moreover, by Lemma \ref{range}, if $GE_1=1$, there exists $x_n=n^{o(1)}$ such that every ball of radius $b_n$ in $G^\infty_{n,t'}$ contains less than $x_n$ vertices. Therefore, 

$$GE_1|A^G_{\text{end}}|\leq x_n\sum_{1\leq j\leq 2,0\leq i\leq K_n}|D_{j,i}|.$$

As every element of $D_{j,i}$ belongs to the ball of radius $i$ centered at $\tilde v_j$ in $\Gmodk$, and $\Gmodk$ has the same law as $T^k_{t'}$:

$$\E(\sum_{1\leq j\leq 2,0\leq i\leq K_n}|D_{j,i}|)\leq  2\E(|B_{K_n}(T^k_{t'},\emptyset)|)$$

\noindent{} where $B_{K_n}(T^k_{t'},\emptyset)$ is the ball of radius $K_n$ centered at the root $\emptyset$. As $T^k_{t'}$ is a (two-stages) supercritical branching process, by Lemma \ref{lemmeprobasurviedeuxetages}:
$$\frac{|B_{i}(T^k_{t'},\emptyset)|}{\rho_{t'}^i}\xrightarrow[i\rightarrow \infty]{}W.$$
\noindent{}with $W$ a random variable with a finite expectation. As $K_n=\frac{(1-\epsilon_2)\ln n}{\ln (\rho_{t'})}$, this limit implies that
\begin{eqnarray*}
\sum_{1\leq j\leq 2,0\leq i\leq K_n}\E|D_{j,i}|&\leq& n^{1-\epsilon_2+o(1)}\\
\E (GE_1|A^G_\text{end}|)&\leq& x_n n^{1-\epsilon_2+o(1)}\\
&=&n^{1-\epsilon_2+o(1)}
\end{eqnarray*}
Therefore, with high probability, $|A^G_\text{end}|\leq z_n$. Let $GE_2$ be the random variable equal to $1$ if this inequality holds and $GE_1=1$. Otherwise, let $GE_2=0$. As $A^G$ is an increasing subset, $|A^G|\leq|A^G_\text{end}|$ at any point of the algorithm.

\begin{lem}
\label{lemGE3}
With high probability, all the vertices in the balls of radius $K_n+b_n$ centered at $\tilde v_1$ and $\tilde v_2$ in $\Gmod_\text{end}$ have less than $\log n$ children. Let $GE_3$ be the random variable equal to $1$ if this property holds and $GE_2=1$. Otherwise, let $GE_3=0$.
\end{lem}

\begin{proof}
The expected number of vertices in the sphere of radius $i$ centered at $\tilde v_j$ in $\Gmod$ is $t'^i$. Therefore the expected total volume of each ball of radius $K_n+b_n$ centered at $\tilde v_1$ and $\tilde v_2$ is equal to $ \frac{t'^{K_n+b_n}-1}{t'-1}=o(n^{\alpha_1})$ for some $\alpha_1>0$. Therefore it is smaller than $n^{\alpha_1}$ with high probability by Markov's inequality. We assume that $n$ is large enough such that $\log n>2t'$. The law of the degree of a given vertex is a Poisson random variable of parameter~$t'$, therefore the probability that one of the first $n^{\alpha_1}$ vertices of each component on $\Gmod_\text{end}$ (in the breadth-first order) has degree larger than $\log n$ is smaller than

\begin{eqnarray*}
2n^{\alpha_1}\sum_{i=\lfloor \log n\rfloor}^{\infty}e^{t'}\frac {t'^i}{i!}&\leq&2n^{\alpha_1} e^{-t'}\frac{t'^{\log_n}}{\lfloor\log_n\rfloor!}\underbrace{\sum_{i\geq 0}\frac{t'^i}{(2t')^i}}_{=2}\\
&=&4n^{\alpha_1} e^{-t'}\frac{t'^{\log_n}}{\lfloor\log_n\rfloor!}\\
&\sim&4e^{-t'}n^{\alpha_1} n^{\log t'}\frac 1{\sqrt{2\pi\lfloor\log n\rfloor}}\left(\frac{e}{\lfloor\log n\rfloor}\right)^{\lfloor\log n\rfloor}\\
&\leq&\alpha_2n^{\alpha_3-\log \log n}\\
&\rightarrow &0
\end{eqnarray*}
With $\alpha_2$ and $\alpha_3$ two constants. The Stirling's approximation is used for the factorial.
\end{proof}

\begin{lem}

With high probability, no initialisation fails.
\end{lem}
As $A$ is initially empty, it is not possible that $v_1\in A$ before the first initialisation. If $GE_2=1$, the size of the set $A$ is smaller than $z_n$ before the second initialisation. As $v_2$ is a uniform vertex independent of $\F_{1,\infty}$, the probability that $v_2$ belongs to $A$ before its initialisation, conditionally on $\F_{1,\infty}$ before the second initialisation is equal to $\frac{|A^G|}n$. This quantity is smaller than $\frac{z_n}n$ if $GE_2=1$, therefore 
\begin{eqnarray*}
\P(v_2\in A_{1,\infty})&\leq&\P(GE_2=0)+\P(GE_2=1\text{ and }v_2\in A_{1,\infty})\\
&=&\P(GE_2=0)+\E(GE_2\frac{|A_{1,\infty}|}n)\\
&\leq&\P(GE_2=0)+\frac{z_n}n\\
&=&o(1)
\end{eqnarray*}
A vertex is added to $A^C$ for the first time during the construction of $S_{j,i+1}$ if one of the following event happens:

\begin{enumerate}
\item Two vertices of $S_{j,i}$ are linked by an edge in $G^\infty_{n,t'}$.
\item A vertex of $S_{j,i}$ is linked to a vertex of $A^G_{j,i}\setminus B^G_{j,i}$ by an edge in $G^\infty_{n,t'}$.
\item Two vertices of $S_{j,i}$ are linked by an edge to the same vertex $w\notin S_{j,i-1}$ in $G^\infty_{n,t'}$.
\item There is a multiple edge in $G^\infty_{n,t'}$ between a vertex $w$ of $S_{j,i}$ and a vertex in $V^G\setminus A_{j,i}$.
\item When a vertex $w$ is probed, the Poisson point process of intensity $\frac{|A^G|}n$ creating dummy vertices is non-empty.
\end{enumerate}
By Lemma \ref{basicpropertiesGmod}, conditionally on the $\sigma$-algebra $\F_{j,i}$, the labels of the edges in $ \Gmod_{n,t'}$ with no endpoint in $B_{j,i}$ are an i.i.d. family of Poisson point processes if intensity $\frac 1 n$ on $[0,t']$. Therefore, conditionally on $\F_{j,i}$:
\begin{enumerate}
\item   the probability of the first event is smaller than $\binom{|S_{j,i}|}2\P(Poi(\frac{t'}n)\geq 1)$;
\item The probability of the second event is smaller than $|S_{j,i}||A^G_{i,j}|\P(Poi(\frac{t'}n)\geq 1)$;
\item   the probability of the third event is smaller than $n\binom{|S_{j,i}|}2\P(Poi(\frac{t'}n)\geq 1)^2$ (there are less than $n$ possible choices for $w$);
\item the probability of the fourth event is smaller than $|S_{j,i}|n\P(Poi(\frac{t'}n)\geq 2)$;
\item the probability of the last event is smaller than $\frac{t'|S_{j,i}||A^G_{j,i+1}|}n$.
\end{enumerate}
Recall that $GE_3=1$ implies that $|S_{j,i}|\leq x_n$ and $|A^G_{j,i+1}|\leq z_n$. Let  $y_n=n^{-\frac {\epsilon_2}3}$. Either $GE_3=0$ or the sum of these  probabilities is smaller than $\frac {y_n} {2b_n} $, for $n$ large enough. By union bound, the probability than one of these events happens for $i\leq b_n$ during one of the initialisations and $GE_2=3$ is smaller than  $y_n=o(1)$.

Let $GE_4$ be the random variable equals to $1$ if no initialisation fails and $GE_3=1$. Otherwise, let $GE_4=0$.

\subsection{Construction of an included branching process}
The goal of this part is to find a branching process simultaneously included in $\Gmod_\text{end}$, $G^k_{n,t'}$ and $G^k_{n,t}$ with law close to $T^k_t$.

\subsubsection{Construction of the included branching process for time~$t'$}

The idea is to find a two-stages branching process whose set of vertices is a subset of $\cup \tilde E_{j,i}$ and whose law will be close to $T^k_{t}$. In order to do this, some bounds of the probability that a given vertex $w$ is in $D_{j,i}\setminus \tilde E_{j,i}$ will be needed. For technical reasons, $w\in D_{j,i}\setminus E_{j,i}$ and $w\in E_{j,i}\setminus \tilde E_{j,i}$ will be treated separately. Let us add the following families of random variables, that will be used as extra randomness when needed, such that, conditionally on $\F_\text{end}$, these families are i.i.d. families of uniform random variables on $[0,1]$.
\begin{itemize}
\item  $(U^a_w)_{w\in \Gmod_\text{end}}$;
\item  $(U^b_{w,w'})_{w\breadthl w'\in \Gmod_\text{end}}$;
\item  $(U^c_{w,w'})_{w\breadthl w'\in \Gmod_\text{end}}$;
\item  $(U^d_w)_{w\in \Gmod_\text{end}}$;
\end{itemize}
These random variables will be used to construct independent variables, by using the following straightforward construction:
\begin{lem}
\label{lemmecouplagebernoulli}
If $U$ and $V$ are random variables, $\F$ a $\sigma$-algebra and $ p\in[0, 1)$ such that:
\begin{itemize}
\item $V\in\{0,1\}$ a.s.
\item $\P(V=1|\F)<p$ a.s.
\item $U$ is a uniform random variable on $[0,1]$, independent of $\F$.
\end{itemize}
Then the random variable $V^{ind}$ equals to $1$ if either $V=1$ or $U<\frac{p-\P(V=1|\F)}{\P(V=0|\F)}$ is a Bernoulli random variable of parameter $p$, larger than $V$, independent of $\F$.
\end{lem}
For any $w\in \Gmod_\text{end}$, let us define the following sets:
\begin{itemize}
\item $\Gamma^a_w=\{w'\in\Gmod_{\text{end}};w'\breadth w\}$. 
\item $\Gamma^{a}_{w+}=\{w'\in\Gmod_{\text{end}};w'\breadthl w\}=\Gamma^a_w\cup\{w\}$. 
\item $\Gamma^{b}_w=\{(w',w'')\in{\Gmod_{\text{end}}}^2;w'\breadthl w''\breadth w\}$. 
\item $\Gamma^{b}_{w+}=\{(w',w'')\in{\Gmod_{\text{end}}}^2;w'\breadthl w''\breadthl w\}$.
\item $\Gamma^{c}_w=\{(w',w'')\in{\Gmod_{\text{end}}}^2;w'\breadthl w'',w'\breadth w,\gen(w'')\leq \gen(w)\}$.
\item $\Gamma^{c}_{w+}=\{(w',w'')\in{\Gmod_{\text{end}}}^2;w'\breadthl w'',w'\breadthl w,\gen(w'')\leq \gen(w)\}$.
\item $\Gamma^{d}_w=\{w'\in\Gmod_{\text{end}}:\gen(w')+(0,b_n)\leq \gen(w)\}$.
\end{itemize}
These sets will denote the set of uniform variables that have been used via Lemma \ref{lemmecouplagebernoulli} when $w$ is probed or fake-probed. For every $w\in B^G_\text{end}$, let $\F^{aug}_{w}$ be the $\sigma$-algebra generated by $\F_w$, $(U^a_{w'})_{w'\in\Gamma^{a}_w}$, $(U^b_{(w',w'')})_{(w',w'')\in\Gamma^{b}_w}$, $(U^c_{(w',w'')})_{(w',w'')\in\Gamma^{c}_w}$ and $(U^d_{w'})_{w'\in\Gamma^{d}_w}$. 

For every $(j,i)$, $i\geq b_n$, let $\F^{aug}_{j,i}$ be the $\sigma$-algebra generated by $\F_{j,i}$, $(U^a_w)_{w\in\B_{j,i-1}}$, $(U^b_{w,w'})_{(w,w')\in \B_{j,i-1}}$, $(U^c_{w,w'})_{(w,w')\in \B_{j,i-1}}$ and $(U^d_{w})_{w\in \B_{j,i-b_n}}$.

For every vertex $w\in B^G_\text{end}$, let $\F^{aug}_{w+}$ be the $\sigma$-algebra generated by $\F_w$, $(U^a_{w'})_{w'\in\Gamma^{a}_{w+}}$, $(U^b_{(w',w'')})_{(w',w'')\in\Gamma^{b}_{w+}}$, $(U^{c}_{(w',w'')})_{(w',w'')\in\Gamma^{c}_{w+}}$ and $(U^d_{w'})_{w'\in\Gamma^{d}_{w+}}$.

For any $w$ probed or fake-probed during the construction of $E_{j,i}$ (therefore such that $\gen(w)=(j,i+b_n-1)$), $\F^{aug}_{j,i+b_n-1}\subset \F^{aug}_w\subset \F^{aug}_{w+}\subset\F^{aug}_{j,i+b_n}$.

\begin{lem}
\label{auglemma}
For every $w\in B_\text{end}$, every $\F_\text{end}$-measurable random variable~$X$, the distribution of $X$ is the same conditionally on $\F_w$ (resp. $\F_{w+}$) and conditionally on $\F^{aug}_w$ (resp. $\F^{aug}_{w+}$).
\end{lem}
For every $i\in \{a,b,c,d\}$, every $w\in B^G_\text{end}$, $\Gamma^i_w$ is $\F_w$-measurable and $\Gamma^{i}_{w+}$ is $\F_{w+}$-measurable. Lemma \ref{auglemma} comes from the fact that the law of every $U$-variables is independent of $\F_\text{end}$. Lemma \ref{auglemma} simply means that any event that can be expressed without the extra randomness is independent of the extra randomness.

In order to construct the included branching process, one needs to study when a vertex belongs to $\cup D_{j,i}\setminus E_{j,i}$. Assuming $GE_1=1$, a vertex $w_1$ is in $D_{j,i}\setminus E_{j,i}$ for some $(i,j)$ if a $b_n-1$-child or a $b_n$-child of $w_1$ is added to $A^C$ during the construction of $D_{j,i+1}$. Let $w_2$ be a $b_n-1$-child of $w_1$. During the construction of $D_{j,i+1}$, only $b_n-1$-children of elements of $D_{j,i}$ are probed. 

\begin{itemize}
\item The vertex $w_2$ is added to $A^C$ while probed if one of the following happens:
\begin{itemize}
\item There exists a multiple edge between $w_2$ and another vertex in $G^\infty_{n,t'}$.
\item When $w_2$ is probed, $\poissondummy_{w_2}$ is not empty.
\item There is an edge between $w_2$ and a vertex $ w_3$ of $A^G_{w_2}\setminus B^G_{w_2}$. 
\end{itemize}

\item The vertex $w_2$ is added to $A^C$ while another vertex $w_3$ is probed if there is an edge between $w_2$ and $w_3$ in $G^\infty_{n,t'}$.

\item A child $w_3$ of $w_2$ is added to $A^C$ while $w_2$ is probed if there is a multiple edge between $w_2$ and $w_3$ in $G^\infty_{n,t'}$.

\item A child $w_3$ of $w_2$ is added to $A^C$ while another vertex $w_4$ is probed if there is an edge between $w_4$ and $w_3$ in $G^\infty_{n,t'}$.
\end{itemize}
For any $(i,j)\in\{0,\dots,K_n\}\times\{1,2\}$, let $\Delta_{j,i+b_n-1}$ be the set of $b_n-1$-children of elements of $D_{j,i}$. For any $i$, $j$,  $\Delta_{j,i+b_n-1}$ is a subset of $S_{j,i+b_n-1}$. Then $w_2$ or one of its children is added to $A^C$ during the construction of 
$D_{j,i+1}$ if at least one of the following events happens:
\begin{enumerate}
\item while $w_2$ is probed, $\poissondummy_{w_2}$ is not empty;
\item there is a multiple edge in $G^\infty_{n,t'}$ between $w_2$ and vertex not in $A^G_{w_2}$;
\item there is an edge between $w_2$ and a vertex of $A^G$ not in $\Delta_{j,i+b_n-1}$ and not a children of an element of $\Delta_{j,i+b_n-1}$;
\item there is an edge between $w_2$ and  a vertex $w_3\in \Delta_{j,i+b_n-1}$;
\item there are a vertex $w_3\neq w_2$ in $\Delta_{j,i+b_n-1}$ and $w_4\notin A^G_{j,i+b_n-1}$,  such that there are edges between $w_2$ and $w_4$ and between $w_3$ and $w_4$ in $G^\infty_{n,t'}$.

\end{enumerate}

For any $(i,j)$ and any vertices $w_2$ and $w_3$ in $\Delta_{j,i+b_n-1}$, let $X^a_{w_2}$ be equal to $1$ if any of the first three events occurs, $X^b_{w_2,w_3}$ be equal to $1$ if the fourth event occurs and $X^c_{w_2,w_3}$ be equal to $1$ if the fifth event occurs. This classification corresponds to the following criteria:
\begin{enumerate}[a.]
\item one vertex of $\Delta_{j,i+b_n-1}$ is added to $A^C$;
\item two vertices $w_2\breadth w_3$ of $\Delta_{j,i+b_n-1}$ are added to $A^C$, when $w_2$ is considered;
\item two vertices $w_2\breadth w_3$ of $\Delta_{j,i+b_n-1}$ are added to $A^C$, when $w_3$ is considered.
\end{enumerate}

$\F^a_w$ be the $\sigma$-algebra generated by $\F_w$ and $\Gmodm_{w+}$. $\F^a_w$ represents the knowledge obtained by the algorithm before $w$ is probed, and the set of labels of edges outgoing from $w$ (without knowing their other endpoints). Let $\F^{aug,a}_w$ be the $\sigma$-algebra generated by $\F^{aug}_w$ and $\Gmodm_{w+}$, \ie{} $\F^a_w$ with the knowledge of the appropriate extra-randomness.

For any $w\in\cup_{j,i}\Delta_{j,i+b_n-1}$, let $GE^{c}_w$ be equal to $1$ if for every $w'\breadth w$  the degree of $w'$ in $\Gmod_\text{end}$ is smaller than $\log n$, and let $GE^a_w$ be equal to $1$ if the degree of $w$ is smaller than $\log n$ and $|A^G_w|\leq z_n$. For any $w$, $GE^a_w\geq GE_4$ and $GE^{c}_w\geq GE_4$. $GE^a_w$ is $\F^a_w$-measurable, and $GE^c_w$ is $\F_{w-}$-measurable.

Let $ \eta_n^a:=\frac {(\ln n+t')z_n+\ln^2 n}  {n} $, $\eta^b_n:=\frac {t'} {n} $ and $\eta^c_n:=\frac {t'\ln n} {n} $. For $n$ large enough, all these variables are in $[0,1]$.

\begin{lem}
\label{lembound1}
For any $w\in\cup\Delta_{j,i+b_n-1}$, conditionally on $\F_w^{aug,a}$,  the families $(X^{a}_{w})$, $(X^{b}_{w,w'})_{w'\in \Delta_{j,i+b_n-1},w\breadth w'}$, $(X^{c}_{w',w})_{w'\in \Delta_{j,i+b_n-1}, w'\breadth w}$  are independent Bernoulli variables. Moreover:
\begin{itemize}
\item $\P(GE^a_wX^a_w=1|\F_w^{aug,a})\leq \eta^a_n$ a.s.;
\item For all $w'\in\Delta_{j,i+b_n-1}$, $w\breadth w'$, $\P(X^b_{w,w'}=1|\F_w^{aug,a})\leq \eta^b_n$ a.s.;
\item For all $w'\in\Delta_{j,i+b_n-1}$, $w'\breadth w$, $\P(GE^c_wX^c_{w,w'}=1|\F_w^{aug,a})\leq \eta^c_n$ a.s.;
\end{itemize}

\end{lem}
\begin{proof}
By Lemma \ref{auglemma}, it is sufficient to prove Lemma \ref{lembound1} with $\F_w^a$ instead of $\F_w^{aug,a}$. 

$\F_{j,i+b_n-1}$, $\Gmodm_{w}$, $GE_{w}^{c}$ are measurable with respect to $\F_{w}$. The only difference between $\Gmodm_{w}$ and $\Gmodm_{w+}$ is the set of edges outgoing from $w$. Let denote in this proof by $e_{w}$ the set of labels of outgoing edges. Therefore $\F_w^a$ is the $\sigma$-algebra generated by $e_w$ and $\F_w$.

$V^G\setminus B^G_{w}$ can be decomposed in these disjoint subsets:
\begin{itemize}
\item $V^{a1}_{w}=V^G\setminus A^G_w$;
\item $V^{a2}_{w}$ is the subset of vertices in $A^G_{w}$, but not in $\Delta_{j,i+b_n-1}$ or children of elements of $\Delta_{j,i+b_n-1}$;
\item $V^b_{w}=\Delta_{j,i+b_n-1}\cap A^G_w$;
\item For any $w'\in \Delta_{j,i+b_n-1}$, $w'\breadth w$, $V^c_{w,w'}$ is the subset of vertices that are in $A^G_{w}$ and are children of $w'$.
\end{itemize}

Conditionally on $\F_{w}$:
\begin{itemize}
\item $\poissondummy_w$ is a Poisson point process of intensity $\frac{|A^G_{w}|}n$ on $[0,t']$,
\item $(\poisson{(w,w')})_{w'\in V^G_w\setminus B^G_w}$ is an i.i.d. family of Poisson point processes of intensity $\frac 1 n$ on $[0,t']$, independent of $\poissondummy_w$.
\end{itemize}
 
$e_{w}$ is the union of $\poissondummy_w$ and the $n-|A^G_{w}|$ processes associated to elements on $V^{a1}_w$. Therefore, conditionally on $\F_{w}$, $e_w$ is a Poisson point process of intensity $1$ on $[0,t']$.

\begin{itemize}
\item The variable $(X^{a}_{w},\Gmod_{w+})$ depends only on $\poissondummy_w$ and on $(\poisson{(w,w')})_{w'\in V^{a1}_w\cup V^{a2}_w}$; 
\item for $w'\in \Delta_{j,i+b_n-1}$, $w\breadth w'$, $X^{b}_{w,w'}$ depends only on the $\poisson{(w,w')}$;
\item for $w'\in \Delta_{j,i+b_n-1}$, $w>w'$, $X^{c}_{w,w'}$ depends only on the $(\poisson{(w,w'')})_{w''\in V^c_{w,w'}}$,
\end{itemize}
 therefore  conditionally on $\F_{w}$, the following variables are independent:
 \begin{itemize}
 \item $(X^a_w,\Gmodm_{w+})$,
 \item  $(X^{b}_{w,w'})$ for $w'\in\Delta_{j,i+b_n-1}$, $w\breadth w'$ 
 \item $(X^{c}_{w,w'})$ for $w'\in\Delta_{j,i+b_n-1}$. 
 \end{itemize}
This implies that conditionally on $\F_w$ and $\Gmodm_{w+}$ (\ie{} conditionally on $\F_w^a$),  the following variables are independent:
 \begin{itemize}
 \item $X^a_w$,
 \item  $(X^{b}_{w,w'})$ for $w'\in\Delta_{j,i+b_n-1}$, $w\breadth w'$ 
 \item $(X^{c}_{w,w'})$ for $w'\in\Delta_{j,i+b_n-1}$. 
 \end{itemize}
  The conclude the proof, we now prove that the parameters of these variables are respectively smaller than $\eta^a_n$, $\eta^b_n$ and $\eta^c_n$.

For any element $y$ of $e_{w}$, let $\alpha_y$ be equal to $0$ if $y\in\poissondummy_w$ and equal to $w'$ if $y\in\poisson{(w,w')}$. Conditionally on $\F^a_w$, the family $(\alpha_y)_{y\in e_{w}}$ is an i.i.d. family, with $\P(\alpha_y=0|\F^{a}_{w})=\frac{|A^G_{w}|}n$ and $\P(\alpha_y=v|\F^a_{w})=\frac 1n$ for any $v\in V^G\setminus A^G_{w}$. Therefore, conditionally on $\F_{w}$ and $\Gmodm_{w+}$, the probability that there is a multiple edge between $w$ and an element of $V^G\setminus A^G_{w}$ is smaller than 
$(n-|A^G_{w}|)\P(\text{Binom}(|e_{w}|,\frac 1 n)\geq 2)\leq \frac{|e_{w}|^2}{n}$
 and the probability that there is an element  $y$ of $e_{w}$ such that $\alpha_y=0$ is smaller than $\frac{|e_{w}||A^G_{w}|}n.$ If $GE^a_{w}=1$, then $|e_{w}|\leq \log n$ and $|A^G_{w}|\leq z_n$. Conditionally on $\F^{a}_{w}$, the probability that there is at least an edge between $w$ and an element of $V^{a2}_{w}$, conditionally on $\F_{w}$, is smaller than $\frac{t'|V^{a2}_{w}|}n\leq \frac{t'|A^G_{w}|}n$. If $GE^a_{w}=1$, this quantity is smaller than $\frac {t'z_n}n$. By summing the previous results, we obtain that $\P(GE^a_{w}X^a_{w}=1|\F^a_{w})\leq \eta^a_n$.

$X^b_{w,w'}$ is equal to $1$ if $\poisson{(w,w')}\neq \emptyset$. Therefore $\P(X^b_{w,w'}|\F^{a}_{w})=e^{-\frac{t'}n} \frac {t'}n=\eta^b_n$. 

$X^c_{w',w}$ is equal to $1$ if $\cup_{w''\in V^c_{w,w'}}\poisson{(w,w'')}\neq \emptyset$. The probability of this event is smaller than $\frac {t'|V^c_{w,w'}|}n$. If $GE^c_{w}=1$, then $|V^c_{w,w'}|\leq \log n$ and therefore $\P(GE^c_{w}X^c_{w',w})\leq \frac{t'\log n}n=\eta^c_n$. 
\end{proof}
The following constructions use the tool explained in Lemma \ref{lemmecouplagebernoulli}.

For any $w\in \cup_{j,i}\Delta_{j,i+b_n-1}$, let $X^{a,indep}_w=1$ if either $GE^a_wX^a_w=1$ or ${U^a_w\leq \frac{\eta^a_n-\P(GE^a_wX^a_w=1|\F^{aug,a}_w)}{\P(GE^a_wX^a_w=0|\F^{aug,a}_w)}}$ . Otherwise, let $X^{a,indep}_w=0$.

For any $(w, w')\in \cup_{j,i}\Delta_{j,i+b_n-1}^2$ with $w\breadth w'$, let $X^{b,indep}_{w,w'}=1$ if either $X^{b}_{w,w'}=1$ or $U^b_{w,w'}\leq \frac{\eta^b_n-\P(X^b_{w,w'}=1|\F^{a,aug}_w)}{\P(X^b_{w,w'}=0|\F^{a,aug}_w)}$. Otherwise, let $X^{b,indep}_{w,w'}=0$.

For any $(w, w')\in \cup_{j,i}\Delta_{j,i+b_n-1}^2$ with $w'\breadth w$, let $X^{c,indep}_{w',w}=1$ if either $GE^c_wX^{c}_{w',w}=1$ or $U^c_{w',w}\leq \frac{\eta^c_n-\P(GE^c_wX^c_{w',w}=1|\F^{a,aug}_w)}{\P(GE^c_wX^c_{w',w}=0|\F^{a,aug}_w)}$. Otherwise let $X^{c,indep}_{w',w}=0$.

\begin{cor}
\label{corindepenfants}
For all $w\in\Delta_{j,i+b_n-1}$ and $\bar w\in\{w,w+\}$, conditionally on $\F^{aug}_{j,i-1}$ and $\Gmodm_{\bar w}$, the families of random variables $(X^{a,indep}_{w'})_{w'\in\Delta_{j,i+b_n-1}\cap \Gamma^1_{\bar w}}$, $(X^{b,indep}_{w',w''})_{(w',w'')\in \Gamma^2_{\bar w}\cap(\Delta_{j,i+b_n-1})^2}$, $(X^{c,indep}_{w',w''})_{(w',w'')\in \Gamma^c_{\bar w}\cap(\Delta_{j,i+b_n-1})^2}$ are independent i.i.d. families of Bernoulli random variables of respective parameters $\eta^a_n$, $\eta^b_n$ and $\eta^c_n$.
\end{cor}
The version with $w+$ is equivalent to the version with $w$ by substituting $w$ by the next vertex of $\Delta_{j,i+b_n-1}$. Corollary \ref{corindepenfants} is proven by induction. Corollary \ref{corindepenfants} holds for $w$ if $w$ is the first element of $\Delta_{j,i+b_n-1}$ (for the breadth-first order), as the families are empty. Assuming that Corollary \ref{corindepenfants} holds for a vertex $w$, let $F$ be the $\sigma$-algebra generated by the random variables $(X^{a,indep}_{w'})_{w'\in\Delta_{j,i+b_n-1}\cap \Gamma^1_{w}}$, $(X^{b,indep}_{w',w''})_{(w',w'')\in \Gamma^2_{w}\cap(\Delta_{j,i+b_n-1})^2}$, $(X^{c,indep}_{w',w''})_{(w',w'')\in \Gamma^c_{w}\cap(\Delta_{j,i+b_n-1})^2}$. All these variables are $\F^{aug}_{w}$-measurable, therefore conditionally on $\F^{aug}_{j,i+b_n-1}$, $F$ and $\Gmodm_{w+}$:
\begin{itemize}
\item  $(X^{a}_{w})$, $(X^{b}_{w,w'})_{w'\in \Delta_{j,i},w\breadth w'}$, $(X^{c}_{w',w})_{w'\in \Delta_{j,i}, w'\breadth w}$  are independent Bernoulli variables, of respective parameters smaller than $\eta^a_n$, $\eta^b_n$ and $\eta^c_n$, by Lemma \ref{lembound1};
\item $(X^{a,indep}_{w})$, $(X^{b,indep}_{w,w'})_{w'\in \Delta_{j,i},w\breadth w'}$, $(X^{c,indep}_{w',w})_{w'\in \Delta_{j,i}, w'\breadth w}$  are independent Bernoulli variables, of respective parameters equal to $\eta^a_n$, $\eta^b_n$ and $\eta^c_n$, by Lemma \ref{lemmecouplagebernoulli}.
\end{itemize}
And therefore Corollary \ref{corindepenfants} holds for $w+$. By applying Corollary \ref{corindepenfants} to the last element of $\Delta_{j,i}$, one obtains:

\begin{cor}
Conditionally on the $\sigma$-algebra $\F_{j,i+b_n-1}$ and the unlabelled graph $\Gmodm_{j,i+b_n}$, the families of random variables $(X^{a,indep})_{w\in\Delta_{j,i+b_n-1}}$, $(X^{b,indep}_{w,w'})_{w\breadth w'\in\Delta_{j,i+b_n-1}}$, $(X^{c,indep}_{w,w'})_{w\breadth w'\in\Delta_{j,i+b_n-1}}$ are independent i.i.d. families of Bernoulli random variables of respective parameters $\eta^a_n$, $\eta^b_n$ and $\eta^c_n$.
\end{cor}

For any $w\in D_{j,i}$, let $X^{a,parent}_w=1$ if there exists $w'$, a $b_n-1$-child of $w$ such that $X^{a,indep}_{w'}=1$. For any $w_1,w_2\in D_{j,i}$, let $X^{bc,parent}_{w_1,w_2}=1$ if there exists $w_1'$, a $b_n-1$-child of $w$ and $w_2'$ a $b_n-1$-child of $w_2$ such that $X^{b,indep}_{w_1',w_2'}=1$ or $X^{c,indep}_{w_1',w_2'}=1$. Let $GE_{j,i}$ be the random variable equal to $1$ if all the balls of radius $b_n-1$ centered at vertices of $D_{j,i}$ in $\Gmod$ contains at most $x_n$ vertices and $|A^G_{j,i}|\leq z_n$. Otherwise, let $GE_{j,i}=0$. $GE_{j,i}$ is $\F_{j,i+b_n-1}$-measurable, and $GE_{j,i}\geq GE_4$.
\begin{cor}
\label{lemmeparentmaj}
Conditionally on $\F^{aug}_{j,i+b_n-1}$ and $\Gmodm_{j,i+b_n}$, the families of random variables $(X^{a,parent}_w)_{w\in D_{j,i}}$ and $(X^{bc,parent}_{w,w'})_{w,w'\in D_{j,i}}$ are independent. Moreover, 
\begin{eqnarray*}
\E(GE_{j,i}X^{a,parent}_w|\F^{aug}_{j,i+b_n-1},\Gmodm_{j,i+b_n})&\leq &x_n\eta^a_n\\
\E(GE_{j,i}X^{bc,parent}_{w,w'}|\F^{aug}_{j,i+b_n-1},\Gmodm_{j,i+b_n})&\leq & x_n^2(\eta^b_n+\eta^c_n)=:\eta^{bc,parent}_n.
\end{eqnarray*}
\end{cor}

\begin{lem}
\label{independentisation}
Let $I$ be an ordered finite set, and let $(X_{i,i'})_{i,i'\in I, i\leq i'})$ be a family of independent Bernoulli random variables of parameters smaller than $\eta$. For $i<i'$, let $X_{i',i}=X_{i,i'}$. Let $\F_2$ be a $\sigma$-algebra. Then there exists a family of random variables $(X^{bound}_i)_{i\in I}$  such that:
\begin{itemize}
 \item $(X^{bound}_i)_{i\in I}$ is an i.i.d. family of Bernoulli random variables of parameter $\min(1,2\sqrt{\eta |I|})$.
\item For any $i,\,i'\in I$, $X_{i,i'}\leq X^{bound}_i$.
\end{itemize}
For any $Y$ $\F_2$-measurable, $\E(Y|(X_{i,i'})_{i,i'\in I})=\E(Y|(X_{i,i'})_{i,i'\in I},(X^{bound}_i)_{i\in I})$.
\end{lem}
Lemma \ref{independentisation} allows to take an i.i.d.  family of Bernoulli variables indexed by two elements of $I$ and bound it by an i.i.d. indexed family of Bernoulli variables indexed by one element of $I$. The bound $2\sqrt{\eta|I|}$ does not seem to be optimal, but is sufficient for our use. As we need a specific version of Lemma \ref{independentisation}, the proof will be done in our particular case. A proof of Lemma \ref{independentisation} can be found in Section \ref{preuvelemme}.

Let $\eta^{parent}_n=x_n\eta^a_n+2\sqrt{z_n\eta^{bc,parent}_n}=o(1)$. We assume that $n$ is large enough such that $\eta^{parent}_n\leq 1$. As a consequence, $z_n\eta^{bc,parent}_n\leq 1$.

For any $w\in D_{j,i}$, let $X^{parent}_w=1$ if either $GE^{parent}_{j,i}X^{a,parent}_w=1$ or there exists $w'\in D_{j,i}$ such that $GE^{parent}_{j,i}X^{bc,parent}_{w,w'}=1$. Otherwise, let  $X^{parent}_w=0$.
Let us define $(X^{parent,indep}_w)_{w\in D_{j,i}}$ as $X^{parent,indep}_w=1$ if either
\begin{itemize}
\item $X^{parent}_w=1$;
\item $U^d_w\leq \frac{\eta^{parent}_n-\P(X^{parent}_w=1|\F^{aug}_{j,i-1},\Gmodm_{j,i},(X^{parent,indep}_{w'})_{w'\breadth w, w'\in D_{j,i}})}{\P(X^{parent}_w=0|\F^{aug}_{j,i+b_n-1},\Gmodm_{j,i+b_n},(X^{parent,indep}_{w'})_{w'\breadth w, w'\in D_{j,i}})}$.
\end{itemize}
If neither condition happens, let $X^{parent,indep}_w=0$.

This family of variables is defined by induction on $w$. Let $\F^{parent,indep}_w$ be the $\sigma$-algebra generated by $\F^{aug}_{j,i-1}, \Gmodm_{j,i}$ and $(X^{parent,indep}_{w'})_{w'\breadth w, w'\in D_{j,i}}$.

\begin{lem}
\label{independentisationcaspart}
For all $w\in D_{j,i}$, 
$$\P(X^{parent}_w=1|\F_w^{parent,indep})\leq x_n\eta^a_n+\sqrt{z_n\eta^{bc,parent}_n}.$$
\end{lem}

\begin{proof}
Let $p_w= \frac{\eta^{parent}_n-\P(X^{parent}_w=1|\F^{parent,indep}_{w})}{\P(X^{parent}_w=0|\F^{parent,indep}_{w})}$. 

Lemma \ref{independentisationcaspart} is proven by induction. We assume that Lemma \ref{independentisationcaspart} holds for every $w'\in D_{j,i}\breadth w$. Therefore for every $w'\breadth w$:
\begin{eqnarray}
\notag p_{w'}&\geq& \frac{\sqrt{z_n\eta^{bc,parent}_n}}{\P(X^{parent}_w=0|\F^{parent,indep}_{w})}\\
&\geq&\sqrt{z_n\eta^{bc,parent}_n}\label{eqnpw}
\end{eqnarray}
\begin{eqnarray}
\notag  X^{parent}_w&\leq& X^{a,parent}_w+\sum_{w'\in D_{j,i},w'\breadth w}X^{bc,parent}_{w',w}+\\
\notag &&+\sum_{w'\in D_{j,i},w'\breadthinv w}X^{bc,parent}_{w,w'}\\
\notag  \E(X^{parent}_w|\F^{parent,indep}_w)&\leq&\E(X^{a,parent}_w|\F^{parent,indep}_w)+\\
\notag &&+\hspace{-0.3cm}\sum_{w'\in D_{j,i},w'\breadth w}\E(X^{bc,parent}_{w',w}|\F^{parent,indep}_w)+\\
 &&+\hspace{-0.3cm}\sum_{w'\in D_{j,i},w'\breadthinv w}\E(X^{bc,parent}_{w,w'}|\F^{parent,indep}_w) \label{decompositioninegalite}
\end{eqnarray}
As $D_{j,i}$ is $\F^{parent,indep}_w$-measurable. Each term of the right-hand-side of the inequality (\ref{decompositioninegalite}) will be bound separately.

\begin{eqnarray}
\notag x_n\eta^a_n&\geq &\E\left(X^{a,parent}_w\Big |\F^{aug}_{j,i},\Gmodm_{j,i}\right)\\
&=&\E\left(X^{a,parent}_w\Bigg |\begin{array}{l}\F^{aug}_{j,i},\Gmodm_{j,i},(X^{a,parent}_{w'})_{w'\breadth w,w'\in D_{j,i}},\\(X^{bc,parent}_{w',w''})_{w'\breadth w'' \in D_{j,i}}\end{array}\right) \notag\\
&= &\E\left(X^{a,parent}_w\Bigg|\begin{array}{c}
\F^{aug}_{j,i},\Gmodm_{j,i},(X^{a,parent}_{w'})_{w'\breadth w\in D_{j,i}},\\ (X^{bc,parent}_{w',w''})_{w'\breadth w'' \in D_{j,i}}, (U^d_{w'})_{w'\in D_{j,i}}\end{array}\right)\label{finaletribu}.
\end{eqnarray}
The first inequality comes from Corollary \ref{lemmeparentmaj}, the second equality comes from  the  conditional independence of the variables $X^{a,parent}_{w}$ and $X^{bc,parent}_{w,w'}$ in Corollary \ref{lemmeparentmaj} and the last equality comes from the conditional independence of $(U^d_w)_{w \in D_{j,i+1}}$.  As $\F^{parent,indep}_w$ is included in the $\sigma$-algebra used in (\ref{finaletribu}), we have:
\begin{equation}\E(X^{a,parent}_w|\F^{parent,indep}_w)\leq x_n\eta^a_n.\label{majorationxaparent}
\end{equation}
 
We are now going to prove that for all $w,w'$:
\begin{equation}
\E(X^{bc,parent}_{w,w'}|\F^{parent,indep}_w)\leq\sqrt{\frac{\eta^{bc,parent}_n}{z_n}}\text{ a.s.}
\label{equationXbcparent}
\end{equation}
Let $w'\breadth w\in D_{j,i}$. Let $\F^{parent}_{w',w}$ be the $\sigma$-algebra generated by $\F^{aug}_{j,i-1}$, $\Gmodm_{j,i}$, $X^{a,parent}_{w^1\in D_{j,i}}$, $(X^{bc,parent}_{w^1,w^2})_{(w^1,w^2)\neq(w',w),\,w^1\breadth w^2\in D_{j,i}}$, $(U^d_{w^1})_{w^1\in D_{j,i} \setminus \{w'\}}$. As $\F^{parent,indep}_w$ is included in the $\sigma$-algebra generated by $\F^{parent}_{w,w'}$ and $X^{parent,indep}_{w'}$, we will prove:
\begin{equation}\E(X^{bc,parent}_{w',w}|\F^{parent}_{w',w},X^{parent,indep}_{w'})\leq \sqrt{\frac{\eta^{bc,parent}_n}{z_n}}\text{ a.s.}\label{equationXbcparentmodif}
\end{equation}
\noindent{}as (\ref{equationXbcparent}) is a consequence of (\ref{equationXbcparentmodif}).	
There are two possibilities. If ${X^{parent,indep}_{w'}=0}$, then $X^{bc,parent}_{w',w}=0$ and (\ref{equationXbcparentmodif}) holds. Otherwise, if $X^{parent,indep}_{w'}=1$:

\begin{eqnarray}
\notag \P(X^{bc,parent}_{w',w}=1|\F^{parent}_{w',w},X^{parent,indep}_{w'}=1)&=&\frac{\P(X^{bc,parent}_{w',w}=1,X^{parent,indep}_{w'}=1|\F^{parent}_{w',w})}{\P(X^{parent,indep}_{w'}=1|\F^{parent}_{w',w})}\\
\notag &=&\frac{\P(X^{bc,parent}_{w',w}=1|\F^{parent}_{w',w})}{\P(X^{parent,indep}_{w'}=1|\F^{parent}_{w',w})}\\
\notag &\leq&\frac{\P(X^{bc,parent}_{w',w}=1|\F^{parent}_{w',w})}{\P(U^d_{w'}\leq p_{w'}|\F^{parent}_{w',w})}\\
&=&\frac{\P(X^{bc,parent}_{w',w}=1|\F^{parent}_{w',w})}{\E( p_{w'}|\F^{parent}_{w',w})}\label{equationXbcparentpart1}
\end{eqnarray}
 By using Corollary \ref{lemmeparentmaj} and the conditional independence of $(U^d_w)_{w \in D_{j,i}}$ we have:
\begin{eqnarray}
\notag \P(X^{bc,parent}_{w',w}=1|\F^{parent}_{w',w})&=&\P(X^{bc,parent}_{w',w}=1|\F^{aug}_{j,i+b_n-1},\Gmodm_{j,i+b_n})\\
&\leq&\eta^{bc,parent}_n\label{equationXbcparentpart2}
\end{eqnarray}
By using (\ref{eqnpw}) and (\ref{equationXbcparentpart2}) in (\ref{equationXbcparentpart1}), the equation (\ref{equationXbcparentmodif}) holds, and therefore equation (\ref{equationXbcparent}) is proven for all $w'<w$.

Let $w'>w$. By Corollary \ref{lemmeparentmaj} and the conditional independence of the $(U^d_{v})_{v\in D_{j,i}}$:
 \begin{eqnarray*}
 \E\left(X^{bc,indep}_{w,w'}\left|\begin{array}{c}\F^{aug}_{j,i+b_n-1},\Gmodm_{j,i+b_n},X^{a,parent}_{w^1\in D_{j,i}},\\ (X^{bc,parent}_{w^1,w^2})_{w^1\breadthl w, w_1\breadthl w_2\in D_{j,i}},\\ (U^d_{w^1})_{w^1\neq w\in D_{j,i} }\end{array}\right.\right)&=& \E(X^{bc,indep}_{w,w'}|\F^{aug}_{j,i+b_n-1},\Gmodm_{j,i+b_n})\\
 &\leq& \eta^{bc,parent}_n\\
&\leq&\sqrt{\frac{\eta^{bc,parent}_n}{z_n}}
 \end{eqnarray*}
as by hypothesis $\eta^{bc,parent}z_n\leq 1$. As $\F^{parent,indep}_w$ is measurable with respect to this $\sigma$-algebra, the bound also holds conditionally on $\F^{parent}_w$, proving inequality (\ref{equationXbcparent}) for $w'>w$.

By summing the bounds obtained in (\ref{majorationxaparent}) and (\ref{equationXbcparent}), we obtain:

\begin{eqnarray*}
 \E(GE_{j,i}X^{parent}_w|\F^{parent,indep}_w)&\leq & x_n\eta^a_n+z_n\sqrt{\frac{\eta^{bc,parent}_n}{z_n}}\\
&=&x_n\eta^a_n+\sqrt{z_n\eta^{bc,parent}_n}.
\end{eqnarray*}
\end{proof}
\begin{cor}
\label{coroindependanceparent}
Conditionally on $\F^{aug}_{j,i+b_n-1}$ and $\Gmodm_{j,i+b_n}$, $(X^{parent,indep}_w)_{w\in D_{j,i+b_n}}$ is an i.i.d. family of Bernoulli variables of parameter $\eta^{parent}_n$.
\end{cor}
Corollary \ref{coroindependanceparent} uses the construction of Lemma \ref{lemmecouplagebernoulli} to obtain i.i.d. variables.

\begin{defi}For every $i\in \N$, $j\in\{1,2\}$, $w\in S_{j,i}\setminus D_{j,i}$, let ${X^{parent,indep}_w=1}$ if $U^d_w\leq \eta^{parent}_n$. Otherwise, let  $X^{parent,indep}_w=0$. 
\label{defiindep}
\end{defi}
Definition \ref{defiindep} allows us to extend the family $(X^{parent,indep}_w)_{w\in D_{j,i}}$ to vertices not in a $D_{j,i}$, in such a way that the family is still i.i.d:
\begin{lem}
\label{lemmeperco}
Conditionally on $F^{aug}_{j,i+b_n-1}$ and $\Gmodm_{j,i+b_n}$, $(X^{parent,indep}_w)_{w\in S_{j,i}}$ is an i.i.d. family of Bernoulli variables of parameter $\eta^{parent}_n$.
\end{lem}
By construction, $(X^{parent,indep}_w)_{w\in S_{j,i}\setminus D_{j,i}}$ is an i.i.d. family of Bernoulli variables of parameter $\eta^{parent}_n$, independent of $(X^{parent,indep}_w)_{w\in D_{j,i}}$  and $\Gmodm_{j,i+b_n}$, conditionally on $F^{aug}_{j,i+b_n-1}$.
\begin{lem}
Conditionally on $\Gmodme$, $(X^{parent,indep}_w)_{w\in \Gmodme}$ is an i.i.d. family of Bernoulli variables of parameter $\eta^{parent}_n$.
\end{lem}

The proof is done by proving by induction the following two claims:

\begin{itemize}
\item $C^1_{j,i}$: Conditionally on $\Gmodm_{j,i+b_n}$, $(X^{parent,indep}_w)_{w\in \B_{j,i}}$ is an i.i.d. family of Bernoulli variables of parameter $\eta^{parent}_n$.

\item $C^2_{j,i}$: Conditionally on $\Gmodm_{j,i+b_n+1}$, $(X^{parent,indep}_w)_{w\in \B_{j,i}}$ is an i.i.d. family of Bernoulli variables of parameter $\eta^{parent}_n$.
\end{itemize}

\begin{lem}
\label{lemmeintermediaireconditionnel}
Conditionally on $\F^{aug}_{j,i+b_n}$, the law of  $\Gmodm_{j,i+b_n+1}$ can be described as follows:
\begin{itemize}
\item Let $(P_w)_{w\in S_{j,i+b_n}}$ be an i.i.d. family of Poisson point processes of intensity $1$ on $[0,t']$.
\item Starting with $\Gmodm$, for each vertex $w\in S_{j,i+b_n}$, for each element $y$ of $P_w$ graft an edge labelled by $y$ between $w$ and a new vertex. The resulting graph is $\Gmodm_{j,i+b_n+1}$.
\end{itemize}
\end{lem}
\begin{proof}
 By Lemma \ref{probeproperties}, this construction gives the law of $\Gmod_{j,i+b_n+1}$ conditionally on $\F_{j,i+b_n}$. By Lemma \ref{auglemma}, this construction therefore gives the law of $\Gmod_{j,i+b_n+1}$ conditionally on $\F_{j,i+b_n}$.
\end{proof}
\begin{enumerate}
\item
$C^1_{1,0}$ comes from Lemma \ref{lemmeperco} for $j=1$ and $i=0$.
\item
Let $(j,i)\in\{1,2\}\times \N$. The graph $\Gmodm_{j,i+b_n}$ and the variables $(X^{parent,indep}_w)_{w\in\B_{j,i}}$ are $\F^{aug}_{j,i+b_n}$-measurable. Therefore, Lemma \ref{lemmeintermediaireconditionnel} entails that, $C^1_{j,i}$ implies $C^2_{j,i}$.

\item
For any $j$, $i$, $C^2_{j,i}$ implies $C^1_{j,i+1}$ by Lemma \ref{lemmeperco}. 

\item The graph $\Gmodm_{j,\infty}$ is obtained from $\Gmodm_{j,K_n+b_n}$ by grafting i.i.d. copies of $T^\infty_{t'}$ to every vertex of $S_{j,K_n+b_n}$.  The family $(X^{parent,indep}_w)_{w\in \Gmodm_{j,\infty}\setminus \Gmodm_{j,K_n+b_n}}$ is constructed as an i.i.d. family of Bernoulli variables of parameter $\eta^{parent}_n$ via Definition \ref{defiindep}. Therefore for all $j$, $C^2_{j,K_n}$ implies $C^2_{j,\infty}$.
\item
$C^1_{1,\infty}$ implies $C^1_{2,0}$ from Lemma \ref{lemmeperco} for $j=2$ and $i=0$.
\end{enumerate}

To summarize the result so far:
\begin{itemize}
\item $GE^4$ is Bernoulli variable, such that $\P(GE^4=0)\xrightarrow[n\rightarrow\infty]{}0$.
\item Conditionally on $\Gmodme$, $(X^{parent,indep}_w)_{w\in\Gmodme}$ is an i.i.d. family of Bernoulli variables of parameter $\eta^{parent}_n$.
\item $\eta^{parent}_n\xrightarrow[n\rightarrow\infty]{}0$.
\item If $GE^4=1$, then $w\in D_{j,i}\setminus E_{j,i}$ implies $X^{parent,indep}_w=1$.
\item If $GE^4=1$, no initialisation fails.
\end{itemize}
	
This summary ends the first part of the proof of the construction of the included branching process, dealing with the vertices in $D_{j,i}\setminus E_{j,i}$. We now need to deal with the vertices in $E_{j,i}\setminus \tilde E_{j,i}$, \ie{} the vertices that might reach the forbidden degree between time $t'$ and $t$.
\subsubsection{Extension of the included branching process to the time~$t$}

As $G^\infty_{n,t'}$ and $G^\infty_{n,[t',t]}$ are independent, conditionally on $\F^{aug}_\text{end}$ the sets of labelled edges in $G^\infty_{n,[t',t]}$ is an i.i.d. family of Poisson point processes of intensity $\frac 1 n$ on $[t',t]$. We now remove remaining vertices adjacent to an edge added between $t'$ and $t$, except if this edge is added between two vertices of the $K_n$th generation. 
\begin{itemize}
\item Let $E^{t'-}$ be the set of vertices $w$ of $\cup_{i< K_n:j\in\{1,2\}}D_{j,i}$ such that $X^{parent,indep}_w=0$;
\item let $E^{t',K_n}$ be the set of vertices $w$ of $D_{1,K_n}\cup D_{2,K_n}$ such that $X^{parent}_w=0$;
\item let $E^{t'+}=E^{t'-}\cup E^{t',K_n}$.
\end{itemize}
 \begin{itemize}
 \item For each $w$ in $E^{t'-}$, let $X^{[t',t]}_w=1$ if at least on edge is added to $w$ in $G^\infty_{n,t}$ between $t'$ and $t$. 
 \item For each $w$ in $E^{t',K_n}$, let $X^{[t',t]}_w=1$ if there is at least one edge added between $w$ and an element of $V^G\setminus E^{t',K_n}$. 
\end{itemize} 
 The variable $X^{[t',t]}_w$ is used to know  if  $w$ might be removed because of an edge added to $w$ between time $t'$ and $t$. 

We split the set of edges according to how many endpoints belong to $E^{t'^+}$.
\begin{itemize}
\item For any $(w,w')\in (E^{t'+})^2\setminus (E^{t',K_n})^2$, let $X^{[t',t],a}_{w,w'}=1$ if at least one edge is added between $w$ and $w'$ between $t'$ and $t$. Otherwise, let $X^{[t',t],a}_{w,w'}=0$. 
\item For any $(w,w')\in (E^{t',K_n})^2$ let $X^{[t',t],a}_{w,w'}=0$. 
\item For any $w\in E^{t'+}$, let $X^{[t',t],b}_{w}=1$ if at least one edge is added between $w$ and an element of $V^G\setminus E^{t'+}$ between time $t'$ and $t$. Otherwise, let $X^{[t',t],b}_{w}=0$. 
\end{itemize}
For any $w\in E^{t'+}$, 
$$X^{[t',t]}=\max_{w'\in E^{t'+}}(X^{[t',t],a}_{w,w'},X^{[t',t],b}_w).$$

Conditionally on $\F^{aug}_\text{end}$, $(X^{[t',t],a}_{w,w'})_{(w,w')\in E^{t'+}}$ is a family of independent Bernoulli variables of parameter  $1-exp(-\frac {t-t'}n)\leq\frac {\epsilon_1} n$ and is independent of $(X^{[t',t],b}_w)_{w\in E^{t'+}}$. The set $E^{t'^+}$ is included in $A^G_\text{end}$, therefore if $GE_4=1$, then $|E^{t'+}|\leq z_n$. As $GE_4$ is $\F^{aug}_\text{end}$-measurable, Lemma \ref{independentisation} allows to construct $(X^{[t',t],a}_w)_{w\in E^{t'+}}$ such that, if $GE^4=1$:
\begin{itemize}
\item For all $w,w'$, $X^{[t',t],a}_{w,w'}\leq X^{[t',t],a}_w$;
\item Conditionally on $\F^{aug}_\text{end}$, $(X^{a,[t',t]}_w)_{w\in E^{t'+}}$ is an i.i.d. family of Bernoulli variables of parameter $2\sqrt{\frac{\epsilon_1} n{z_n}}$, independent of $((X^{[t',t],b}_w)_{w\in E^{t'+}},\F^{aug}_\text{end}, G^\infty_{n,t'})$.
\end{itemize}
Conditionally on $\F^{aug}_\text{end}$,  $(X^{[t't],b}_w)_{w\in E^{t'-}}$ is an independent i.i.d. family of Bernoulli variables of parameter smaller than $\epsilon_1\frac{n-|E^{t'+}|}n\leq \epsilon_1$.
Therefore conditionally on $\F^{aug}_\text{end}$, $\max(X^{[t',t],a}_w,X^{[t',t],b}_w)\geq GE_4X^{[t',t]}_w$ is an independent family of Bernoulli variables of parameter smaller than $2\sqrt{\frac{\epsilon_1} n{z_n}}+\epsilon_1=:\eta^{[t',t]}_n$. It should be noticed that $\eta^{[t',t]}_n=\epsilon_1+o(1)$.

Let $(U^e_w)_{w\in\Gmod_\text{end}}$ be a family of random variables such that conditionally on $(\F^{aug}_\text{end}, G^\infty_{n,t}, (X^{[t',t],a}_w)_{w\in E^{t'+}})$, the family $(U^e_w)_{w\in\Gmod}$ is i.i.d. uniform variables on $[0;1]$.
For all $w\in \Gmod_\text{end}\setminus E^{t'+}$, let $X^{[t',t],a}_w=X^{[t',t],b}=0$.

Let $X^{[t',t],indep}_w=1$ if either:
\begin{itemize}
\item  $X^{[t',t],a}_w=1$;
\item  $X^{[t',t],b}_w=1$;
\item $U^e_w\leq \frac{\eta^{[t',t]}_n-p'_w}{1-p'_w}$ with $p'_w=\P(X^{[t',t],a}_w=1 \text{ or }X^{[t',t],b}_w=1|\F^{aug}_\text{end})$.
\end{itemize}

Then conditionally on $\F^{aug}_\text{end}$, $(X^{[t',t],indep}_w)_{w\in\Gmod}$ is an i.i.d. family of Bernoulli variables of parameter $\eta^{[t',t]}_n$. Therefore conditionally on $\Gmodm$, $(X^{parent,indep}_w),(X^{[t',t],indep})$ are two independent families of Bernoulli variables of parameters $\eta^{parent}_n$ and $\eta^{[t',t]}_n$, and therefore their maximum is an i.i.d. family of Bernoulli variables of parameter $\eta^{[0,t]}_n=1-(1-\eta^{[t',t]}_n)(1-\eta^{parent}_n)=\epsilon_1+o(1)$.

\subsubsection{Use of the included branching process}
The subgraph of $\Gmodk$ with only vertices $w$ such that $X^{parent,indep}_w=X^{[t',t],final}_w=0$ has same law as an independent percolation of parameter $1-\eta^{[0;t]}_n$ on two independent copies of $T^k_{t'}$.

For $n$ large enough, $\eta^{[0,t]}_n\leq 2\epsilon_1$. For this reason, in the following part, we will study the percolation of parameter $1-2\epsilon_1$. The resulting graph is also a two-stages multitype branching process.

\begin{lem}
\label{continuiteprocessus}
Let $\Gmodke$ be the two-stages multitype branching process obtained by doing a percolation of parameter $1-2\epsilon_1$ on $\Gmodk$. Let $a^k_{t',\epsilon_1}$ be the probability of survival, and $\rho_{t',\epsilon_1}$ the spectral radius associated to $\Gmodke$. Then:
\begin{equation}
a^k_{t',{\epsilon_1}}\xrightarrow[{\epsilon_1}\rightarrow 0,t'\rightarrow t]{}a^k_t;\label{convergenceproba}
\end{equation}
\begin{equation}
\liminf_{\epsilon_1\rightarrow 0,t'\rightarrow t}\rho_{t',{\epsilon_1}}\geq\rho_t\label{convergencerayonspectral}.
\end{equation}
\end{lem}
The proof of Lemma \ref{continuiteprocessus} will be done in Section \ref{preuvecontinuiteprocessus}.

Recall that $K_n=\frac{(1-\epsilon_2)\ln n}{\ln (\rho_{t'})}$. Let $\epsilon>0$. When $\epsilon_1$ and $\epsilon_2$ tends to $0$, $(\rho_{t',{\epsilon_1}})^{\frac{1}{1-\epsilon_2}}$ converges to $\rho_t>(\rho_t)^\frac 2 3$. We assume that $\epsilon_1$ and $\epsilon_2$ are small enough so that $(\rho_{t',\epsilon_1})^{1-\epsilon_2}> (\rho_{t'})^{\frac 2 3}$, and that $a^k_{t',\epsilon_1}\geq a^k_t-\epsilon$.

The graph $\Gmodke$ is a two-stages branching process. Therefore, conditionally on survival, the size of the $i$th generation grows as $(\rho_{t',\epsilon_1})^i$, by Lemma \ref{lemmevitessecroissance}:
$$\frac{|S^{i}(\Gmodke)|}{(\rho_{t',\epsilon_1})^i}\xrightarrow[i\rightarrow\infty]{} W\text{ a.s.}$$
\noindent where $S^i$ denotes the sphere of radius $i$ and $W$ is a random variable, a.s. positive if $\Gmodke$ is infinite. Therefore, for any $1<\alpha<\rho_{t',\epsilon_1}$:
$$\P(|S^{i}(\Gmodke)|\geq \alpha^i)\xrightarrow[i\rightarrow\infty]{}a^k_{t',\epsilon_1}.$$

By choosing $\alpha$ close enough to $\rho_{t',\epsilon_1}$, $n^\frac 2 3=o(\alpha^{K_n})$, and therefore the probability  that the $K_n$th generation of $\Gmodke$ contains at least $n^\frac 2 3$ is larger than $a^k_{t',\epsilon_1}-\epsilon$ .

Conditionally on the first $K_n$ generations of $\Gmodke$, the offspring of each vertex $w$ of the $K_n$th generation of $\Gmodke$, seen as a subgraph of $\Gmodk$  is described by the functions $f_i$. By Lemma \ref{continuiteproba}, for $t'\leq t$, these density functions are bounded from below, let say by $\delta>0$.  Therefore, by the law of large numbers, with high probability when $n$ tends to infinity, if  the $K_n$th generation of $\Gmodke$ contains at least $n^\frac 2 3$ vertices, then at least $\frac \delta 2 n^\frac 2 3$ of them  have no children in $\Gmodk$.  As a consequence, for $n$ large enough and $\epsilon_1$ and $\epsilon_2$ small enough, the following events happen simultaneously with probability larger than $(a^k_t-\epsilon)^2-\epsilon$:
\begin{itemize}
\item The first and second component of $\Gmodke$ are infinite;
\item The $K_n$th generation of these components contains at least $n^\frac 2 3$ vertices.
\item Among these vertices, at least $2n^\frac 3 5$ have no children in $\Gmodk$, and are therefore not saturated in $G^k_{n,t'}$.
\end{itemize}

If these conditions are satisfied, let $H_1$ (resp. $H_2$) be the set of the first $2n^\frac 3 5$ (in the breadth-first order) vertices of the generation $K_n$ of the first (resp. second) component of $\Gmodke$ with no children in $\Gmodk$. By construction, any element $w$ of $H_j$ satisfies the following properties:
\begin{itemize}
\item There is a path from $v_j$ to $w$ in $G^k_{n,t'}$, such that no edge is added to any vertex of this path between $t'$ and $t$, except possibly to $w$.
\item No edge is added between $w$ and an element of $V^G\setminus E^{t'-}$ between $t'$ and $t$.
\end{itemize}
Let $\F^{aug2}$ be the $\sigma$-algebra generated by $\F^{aug}$, $(U^e_{w})_{w\in\Gmod\text{end}}$ and \\$(\poisson{(w,w')}\cap [t',t])_{(w,w')\in (V^G)^2\setminus (E^{t',K_n})^2}$.

 $H_1$ and $H_2$ are $\F^{aug2}_\text{end}$-measurable. Conditionally on $\F^{aug2}_\text{end}$, the sets of labels of edges between elements of $E^{t',K_n}$ with labels in $[t',t]$ are independent Poisson point processes of intensity $\frac 1 n$ on $[t',t]$. Remove any vertex $w$ in $H_1\cup H_2$, such that at least an edge is added between $w$ and element of $E^{t',K_n}\setminus (H_1\cup H_2)$. These events are independent over the vertices $w$, and of probability smaller than $\epsilon_1$. Let $\tilde H_1$ (resp. $\tilde H_2$) be the set of remaining vertices $w$ of $H_1$ (resp. $H_2$). Assuming $\epsilon_1<\frac 1 2$, by the law of large numbers, with high probability $|\tilde H_1|\geq n^\frac 3 5$ and $|\tilde H_2|\geq n^\frac 3 5$.

For any $(w_1,w_2)\in\tilde H_1\times\tilde H_2$, exactly one edge is added between $t'$ and $t$ between $w_1$ and $w_2$ with probability $e^{-\frac{\epsilon_1}n}\frac{\epsilon_1}n	=:a_n$. These events are independent, therefore conditionally on $|\tilde H_1|\geq n^\frac 3 5$, $|\tilde H_2|\geq n^\frac 3 5$, the probability that no simple edge is added between an element of $\tilde H_1$ and an element of $\tilde H_2$ is smaller than:

$$(1-a_n)^{n^\frac 6 5}\xrightarrow[n\rightarrow \infty]{}0.$$

Moreover, the probability that there exists a vertex $w_1\in H_1 \cup H_2$ such that two edges toward two different vertices of $H_1\cup H_2$ are added between $t'$ and $t$ is  smaller than $(4n^\frac 3 5)^3(\frac{\epsilon_1}n)^2=o(1)$ by union bound.

All these results implies that if $|H_1|\geq 2n^\frac 3 5$ and $|H_2|\geq 2n^\frac 3 5$, then with high probability there exist $w_1\in \tilde H_1$ and $w_2\in \tilde H_2$ such that:
\begin{itemize}
\item An edge is added between $w_1$ and $w_2$ between $t'$ and $t$.
\item No other edge is added to either $w_1$ or $w_2$ between $t'$ and $t$.
\end{itemize}

In that case there exists a path from  $v_1$ to $v_2$ in $G^k_{n,t}$, through $w_1$ and $w_2$, and this concludes the proof of Lemma \ref{lemmelien}.

\subsection{Proof of Lemmata \ref{lemmecontinuite1} and \ref{continuiteprocessus}}
\label{preuvecontinuiteprocessus}

For all $s\leq t$ and $\epsilon\geq 0$, let  $T^{k}_{s,\epsilon}$ be the component of the root obtained by taking a percolation of parameter $1-2\epsilon$ on $T^k_s$. $T^{k}_{s,\epsilon}$ is a two-stages multitype branching process, with one offspring law for the root, and another offspring law for all the other vertices. By definition, the spectral radius of a two-stages branching process is  the spectral radius of the associated one-stage branching process. Let  $(T^{k+,y}_{s,\epsilon})_{y\in [0,t']}$ denote the one-stage branching process associated to $T^{k}_{s,\epsilon}$ starting at a non root vertex, where $y$ denotes the type of the first vertex. Let $T^{k+,\epsilon,y}_{s,\epsilon}$ be the branching process obtained by doing a percolation of parameter $1-2\epsilon$ on $T^{k+,y}_{s}$. With the notation of section \ref{resultmultitype}, let $M_{s,\epsilon}$  be the operator $M$ for the branching process $T^{k+,\epsilon,\cdot}_{s,\epsilon}$. If $\epsilon=0$, $M_s$ will be used to denote $M_{s,0}$. As $M$ corresponds to the expected number of vertices of each type, $M_{s,\epsilon}=(1-2\epsilon)M_s$ and
$$\rho_{s,\epsilon}=(1-2\epsilon)\rho_{s}.$$
Therefore Lemma \ref{lemmecontinuite1} implies the limit (\ref{convergencerayonspectral}) in Lemma \ref{continuiteprocessus}.

For $s\leq t$, let $m_s(y,z)$ denote the expected density of the number of children of type $z$ that $T^{k+,y}_s$ has. With the functions $g_i$ introduced in Lemma \ref{probaegaldensite}, for all $y,z\in[0,s]$, :
$$m_s(y,z)=\sum_{i=1}^{k-2}\frac{1}{i-1!}\underbrace{\int_0^s\int_0^s\dots\int_0^s}_{i-1\text{ times}}g_{i}(t,y,z,x_2,x_3,\dots,x_{i})dx_2dx_3\dots dx_{i}.$$

By the boundedness and continuity property of the functions $g_i$ shown in Lemma \ref{continuiteproba}, this means that $s\rightarrow M_{s}$ is a continuous application, and therefore that its spectral radius is a upper semi-continuous function of $s$, by \cite{newburgh}, proving Lemma \ref{lemmecontinuite1}.

For any $s\leq t$ and $\epsilon\geq 0$, let $q_{s,\epsilon}(x)$ be the extinction probability of the one-stage multitype branching process $T^{k+,\epsilon,x}_{s,\epsilon}$ starting with a vertex of type~$x$. As previously, to simplify notations, let $q_{s}=q_{s,0}$. By Lemma \ref{lemmeprobasurvie}, $q_{s,\epsilon}$ is the smaller positive solution of:
$$\phi_{s,\epsilon}(f)=f$$
where $\phi_{s,\epsilon}$ is the operator defined by 
$$\phi_{s,\epsilon}f(y)=\E_y(\prod_{i=1}^{Z^\epsilon}f(X^\epsilon_i))$$
where $(X^\epsilon_1,\dots,X^\epsilon_{Z^\epsilon})$ has the law of the types of the vertices of the first generation of $T^{k+,x}_{s,\epsilon}$. By the percolation properties,  $\phi_{s,\epsilon}$ can be computed as:
$$\phi_{s,\epsilon}f(x)=\E_x(\prod_{i=1}^Z f(X_i)^{B_i})$$
where $(X_1,\dots,X_Z)$ has the law of the types of the vertices in the first generation of $T^{k+,x}_s$, and $(B^i)_{i\in\N}$ is an i.i.d. family of Bernoulli variables equal to $0$ with probability $1-2\epsilon$. Therefore, for $s=t$ and $\epsilon=0$:
\begin{eqnarray}
\phi_{t,0}f(y)&=&g^0(t,y)+\int_{x_1=0}^tg^1(t,y,x_1)f(x_1)dx_1+\dots+\label{phit}\\
&+&\frac{1}{k-2!}\int_{x_1,x_2\dots x_{k-2}=0}^{t}\hspace{-1.7cm}g^{k-2}(t,y,x_1,\dots,x_{k-2})f(x_1)\dots f(x_{k-2})dx_1,\dots dx_{k-2}.\nonumber
\end{eqnarray}

For any , $\delta\in[0,1]$, let $q^\delta=\delta \1{}+(1-\delta)q_t$. For any $\epsilon\geq 0$, $s\in[0,t]$, $\delta\in[0,1]$, $x\in[0,t]$, let 
$$l(\epsilon,s,\delta,x)=(\phi_{s,\epsilon}q^\delta)(x)-q^\delta(x).$$

As $T^k_t$ is supercritical by hypothesis, $q^k_t$ is not uniformly equal to $1$, as we are working in the supercritical case. As all the functions $g$ are positive and continuous by Lemma \ref{continuiteproba}, $q^k_t$ is continuous and not equal to $1$ at any point.  For any $x$, $\delta\rightarrow q^\delta(x)$ is therefore a strictly increasing positive linear function. Therefore for any $i\geq 2$ and any $x_1\dots x_i$, $\delta\rightarrow q^{\delta}(x_1)q^{\delta}(x_2)\dots q^{\delta}(x_i)$ is a strictly convex function of $\delta\in[0,1]$, and therefore so is:
$$\delta\rightarrow h^i(t,y,\delta):=\int_{x_1,x_2\dots x_{i}=0}^{t}g^{i}(t,y,x_1,\dots,x_{i})q^\delta(x_1)\dots q^\delta(x_{i})dx_1,\dots dx_{i}.$$
By equation (\ref{phit}), for a fixed  $y$, $l(0,t,\delta,y)$ is a polynomial in $\delta$, of degree at most $k-2$, and its coefficients can be expressed as integrals of the  functions $q^k_t$ and $g_i$. As these functions are continuous, the coefficient of the polynomial $\delta\rightarrow l(0,t,\delta,y)$ are continuous in $y$ and its derivative $(\delta,y) \rightarrow\frac{\partial l(0,t,\delta,y)}{\partial \delta}$ is a continuous function of $(\delta,y)$ and is negative for $\delta=0$ for all $y$. As $[0,1]\times[0,t]$ is a compact set, there exists $\eta>0$ and $\delta_0>0$ such that, for all $y\in[0,t]$ and all $0\leq \delta\leq \delta_0$, 
\begin{eqnarray*}
\frac{\partial l(0,t,\delta,y)}{\partial \delta}&<&-\eta\\
l(0,t,\delta,y)&<&-\eta\delta\text{\hspace{1cm}as $l(0,t,0,y)=0$}
\end{eqnarray*}

Let $\delta\in(0,\delta_0)$. By Lemma \ref{continuiteproba}, $(\epsilon,s,y)\rightarrow l(\epsilon,s,\delta,y)$ is a continuous function on the compact set $\{(\epsilon,s,\delta,y):\epsilon\in[0,\frac 1 2], s\in[0,t],y\in[0,s]\}$, and therefore uniformly continuous. There exists $\epsilon_{\max}>0$ and $s_{\min}<t$ such that, for all $\epsilon\in[0,\epsilon_{\max}]$, $s\in[s_{\min},t]$, $y\in[0,s]$:
\begin{eqnarray*}
|l(\epsilon,s,\delta,y)-l(0,t,\delta,y)|&\leq&\eta \delta\\
l(\epsilon,s,\delta,y)&\leq&0\\
(\phi_{s,\epsilon}q^{\delta})(y)&\leq &q^{\delta}(y)
\end{eqnarray*}

As this inequality holds for all $y$, by Corollary \ref{coromonotonie}, $q^{\delta}\geq q_{s,\epsilon}$. As the law of the set of labels of the  vertices of the first generation of $T^{k}_{s,\epsilon}$ converges to the law of the set of labels of the vertices of the first generation of $T^k_t$, and as $\delta$ can be chosen arbitrarily small, this inequality implies that
$$\liminf_{\epsilon_\rightarrow 0, s\rightarrow t}\tilde q^k_{t',{\epsilon}}\leq\tilde q^k_{t,0}$$
\noindent where $\tilde q^k_{s,\epsilon}$ is the probability of extinction of $T^{k,\epsilon}_t$. The other side of the limits is a consequence of the following facts:
\begin{itemize}
\item by definition, $\tilde q^k_{s,\epsilon}$ is the limit the non-decreasing sequence $(\P(|T^{k,\epsilon}_s|\leq n))_{n\geq 0}$, 
\item the application $(s,\epsilon)\rightarrow T^{k,\epsilon}_s$ is continuous for the local limit,
\item The event $|T|\leq n$ is continuous for the local limit.
\end{itemize} 
Therefore the equation (\ref{convergenceproba}) of Lemma \ref{continuiteprocessus} holds (the probability of survival $a^k_{s,{\epsilon}}$ is equal to $1-\tilde q^k_{s,\epsilon}$).

\subsection{Proof of Lemma \ref{independentisation}}
\label{preuvelemme}
\begin{proof}
If $2\sqrt{\eta|I|}\geq 1$, Lemma \ref{independentisation} is straightforward, so w.l.o.g. we can assume that $2\sqrt{\eta|I|}\leq 1$. In particular, $\eta |I|\leq 1$.

For any $i$, let $X_i=\max_{j}X_{i,j}$.

Let $(U_i)_{i\in I}$ be an independent i.i.d. family of uniform variables. For any~$i$, let $X^{indep}_i\in\{0,1\}$ be defined by $X^{indep}_i=1$ if and only if either:
\begin{itemize}
\item $X_i=1$,
\item or $U_i\leq \frac{2\sqrt{\eta|I|}-\P(X_i=1|X^{indep}_1,\dots,X^{indep}_{i-1})}{\P(X_i=0|X^{indep}_1,\dots,X^{indep}_{i-1})}$.
\end{itemize}

\begin{lem}
\label{intermediaire}
For any $i\in I$, $P(X_i=1|(X^{indep}_{i'})_{i'<i})\leq \sqrt{\eta |I|}$.
\end{lem}
If Lemma \ref{intermediaire} holds,  by Lemma \ref{lemmecouplagebernoulli}, $(X^{indep}_i)_{i\in I}$ is an i.i.d. family of Bernoulli variables of parameter $2\sqrt{\eta|I|}$. 

\begin{proof}

Lemma \ref{intermediaire} is proven by induction. By union bound,
\begin{eqnarray*}
\E(X_i|(X^{indep}_{i'})_{i'<i})&\leq&\sum_{i'\in I}\E(X_{i,i'}|(X^{indep}_{i'})_{i'<i})\\
&=&\sum_{j\in I, j< i}\E(X_{i,j}|(X^{bound}_{i'})_{i'<i})+\sum_{j\in I, j\geq  i}\E(X_{i,j}|(X^{bound}_{i'})_{i'<i})
\end{eqnarray*}
For any $j\in I$, let $\F_{i,j}$ be the $\sigma$-algebra generated by $(X_{i_1,i_2})_{i_1\,i_2\in I,\,(i_1,i_2)\neq (i,j)}$ and $(U_{i_1})_{i_1\in I, i_1\neq j}$. For all $i_1\notin \{i,j\}$, $X^{indep}_{i_1}$ is $\F_{i,j}$-measurable. Conditionally on $\F_{i,j}$, $U_j$ and $X_{i,j}$ are independent, $U_j$ is a uniform variable and $X_{i,j}$ is a  Bernoulli random variable of parameter smaller than $\eta$. If $j<i$:

$$\E(X_{i,j}|(X^{indep}_{i'})_{i'<i})=\E(\E(X_{i,j}|\F_{i,j},X^{indep}_{j})| (X^{indep}_{i'})_{i'<i}).$$

There are two possibilities. If $X^{indep}_{j}=0$, then $X_{i,j}=0$. Otherwise, if $X^{indep}_{j}=1$:
\begin{eqnarray*} 
\E(X_{i,j}|\F_{i,j},X^{indep}_j=1)&=&\frac{\P(X_{i,j}=1,X^{indep}_j=1|\F_{i,j})}{\P(X^{indep}_j=1|\F_{i,j})}\\
  &\leq&\frac{\P(X_{i,j}=1|\F_{i,j})}{\P(U_j\leq p_j|\F_{i,j})}\\
  &=&\frac{\P(X_{i,j}=1|\F_{i,j})}{\E(p_j|\F_{i,j})}
\end{eqnarray*}
\noindent where $p_j=\frac{2\sqrt{\eta|I|}-\P(X_j=1|X^{indep}_1,\dots,X^{indep}_{j-1})}{\P(X_j=0|X^{indep}_1,\dots,X^{indep}_{j-1})}$. By induction, 

\begin{eqnarray*} 
\sqrt{\eta|I|}&\leq &2\sqrt{\eta|I|}-\P(X_j=1|X^{indep}_1,\dots,X^{indep}_{j-1})\\&\leq& p_j\\
\E(X_{i,j}|\F_{i,j},X^{indep}_j=1)&\leq&\frac{\eta}{\sqrt{\eta|I|}}\\
&=&\sqrt{\frac{\eta}{|I|}}
\end{eqnarray*}
Therefore, almost surely $\E(X_{i,j}|\F_{i,j},X^{indep}_j)\leq \sqrt{\frac{\eta}{|I|}}$, and therefore $\E(X_{i,j}|(X^{indep}_{i'})_{i'<i})\leq \sqrt{\frac{\eta}{|I|}}$.

If $j\geq i$, $\E(X_{i,j}|\F_{i,j})\leq\eta\leq \frac{\eta}{\sqrt{\eta |I|}}=\sqrt{\frac{\eta}{|I|}}$, and therefore $\E(X_{i,j}|(X^{indep}_{i'})_{i'<i})\leq \sqrt{\frac{\eta}{|I|}}$, as all the $X^{indep}_{i'}$, for $i'<i$ are $\F_{i,j}$-measurable.

By summing over $I$, one obtains Lemma \ref{independentisation}. 

\end{proof}

\end{proof}
\bibliographystyle{amsalpha}
\bibliography{dynamicERgraphprocess}

\end{document}